\UseRawInputEncoding
\RequirePackage{fix-cm}
\pdfoutput=1
\documentclass[UTF8,fontset=windows,final, leqno]{siamltex}
\usepackage{ctex}
\usepackage{amsmath}
\usepackage{epsfig}
\usepackage{graphicx}
\usepackage{amssymb}
\usepackage{CJK}
\usepackage{color}
\usepackage{subfigure}
\usepackage{caption}
\usepackage{float}
\usepackage{latexsym, bm}
\usepackage{hyperref}
\usepackage[noadjust]{cite}
\usepackage{comment}

\newtheorem{remark}{Remark}[section]

\allowdisplaybreaks[4]

\normalsize
\renewcommand\figurename{Figure}
\renewcommand\tablename{Table}

\renewcommand\abstractname{Abstract}

\renewcommand\figurename{Figure}
\renewcommand\tablename{Table}
\renewcommand\refname{References}

\def\Ome{{\Omega}}

\def\reff#1{\eqref{#1}}

\def\Norm#1#2{\left\|\,#1\,\right\|_{#2}}

\def\no{{\nonumber}}

\def\div{{\mbox{\rm div\,}}}

\def\Ome{\Omega}

\newcommand{\ow}{\hat{\varpi}}

\def\bv{\mathbf{v}}

\newcommand{\op}{\hat{p}}
\newcommand{\ot}{\hat{T}}

\newcommand{\bt}{\boldsymbol{\Theta}}

\def\bn{\mathbf{n}}

\def\bu{\mathbf{u}}

\def\bw{\mathbf{w}}

\def\bn{\mathbf{n}}

\def\bV{\mathbf{V}}

\def\bV{\mathbf{V}}

\CTEXoptions[today=old]
\begin{document}


\title{A new mixed finite element method for arbitrary element pair for a quasi-static nonlinear permeability thermo-poroelasticity model\footnote{Last update: \today}}

\author{
Zhihao Ge\thanks{School of Mathematics and Statistics, Henan University, Kaifeng 475004, P.R. China ({\tt zhihaoge@henu.edu.cn}).
	The work of this author was supported by the National Natural Science Foundation of China under grant No.12371393 and 11971150.}
\and
Wenshuai Hu\thanks{School of Mathematics and Statistics, Henan University, Kaifeng 475004, P.R. China.}
}

\maketitle


\setcounter{page}{1}


\begin{abstract}
In this paper, we develop a multiphysics finite element method for solving the quasi-static thermo-poroelasticity model with nonlinear permeability. The model involves multiple physical processes such as deformation, pressure, diffusion and heat transfer. To reveal the multi-physical processes of
deformation, diffusion and heat transfer, we reformulate the original model into a fluid coupled problem that is general Stokes equation coupled with two reaction-diffusion equations. Then, we prove the existence and uniqueness of weak solution for the original problem by the $B$-operator technique and by sequence approximation for the reformulated problem. As for the reformulated problem we propose a fully discrete finite element method which can use arbitrary finite element pairs to solve the displacement $\bu$  pressure $\tau $ and variable $\varpi,\varsigma$, and the backward Euler method for time discretization. Finally, we give the stability analysis of the above proposed method, also we prove that the fully discrete multiphysics finite element method has an optimal convergence order. Numerical experiments show that the proposed method can achieve good results under different finite element pairs and are consistent with the theoretical analysis.
\end{abstract}

\pagestyle{myheadings}
\thispagestyle{plain}
\markboth{ZHIHAO GE, WENSHUAI HU}{A NEW MFEM FOR A NONLINEAR THERMO-POROELASTICITY MODEL}

\section{Introduction}
Thermo-poroelasticity is a branch of poroelasticity that describes the thermo-fluid-solid interaction and deals with the coupling between mechanical effects, fluid flow and heat transfer in porous media.This branch has attracted more attention in recent years due to its applications in various engineering and environmental fields, such as soil, medicine pills, polymer gels, mechanical, biomedical and chemical systems, one can refer to  \cite{ Hamley2007,Terzaghi1943,Schowalter2000,Tanaka1979,coussy2004}.
In these fields, the porous medium may experience deformations, fluid pressure changes and temperature variations due to external loading or fluid injection/withdrawal. To model these phenomena, different thermo-poroelasticity models have been proposed and studied. For example, Ge and Chen show the convergence order of variables under a linear condition  \cite{Chen2020}, Brun proves the well-posedness of the solution under a nonlinear condition with nonlinear thermal convection terms  \cite{Brun2021}, Zhang uses the standard Galerkin method  to analyze the fully coupled nonlinear thermo-poroelasticity model, gives the fully discrete scheme of the model, and proves the convergence and error estimation \cite{Zh2022}.
Brun  derives the linear form of the thermo-poroelasticity model from the energy conservation equation \cite{Brun2020}, which combines the momentum and mass equations. Luo  derives the specific expression of permeability by assuming that permeability, stress and pressure follow a negative exponential function \cite{Luo2016}, and analyzes the effects of stress and pressure. As for the  thermo-poroelasticity model, we further consider the nonlinear relationship between permeability and stress, pressure, temperature following a negative exponential function, and obtain the following  quasi-static thermo-poroelasticity model:
\begin{alignat}{2}\label{1-1} 
	-\mu\nabla^{2}\mathbf{u}-(\lambda+\mu)\nabla\mathrm{div}\mathbf{u}+ \alpha \nabla p +\beta \nabla T&=\mathbf{f}
	&&\qquad \mbox{in } \Omega_t:=\Omega\times (0,t_f)\\ 
	(c_{0}p-b_0 T+\alpha \mathrm{div} \mathbf{u})_t - \mathrm{div} \left(k(\mathbf{u},p,T)\nabla p\right)  &=g &&\qquad \mbox{in } \Omega_t,\label{1-2}\\
	(a_0 T -b_0 p +\beta \mathrm{div} \bu)_{t}-\mathrm{div}(\bt \nabla T) &=\phi && \qquad \mbox{in } \Omega_t,\label{1-3}
\end{alignat}
where  $\Omega \subset \mathbb{R}^{d}(d=2,3)$  is a bounded polygonal domain with the boundary  $\partial \Omega$.   $  \partial_{t}$  is the derivative with respect to time,  $a_{0}$  is the effective thermal capacity,  $b_{0}$  is the thermal dilation coefficient, $ c_{0} $ is the specific storage coefficient, $ \alpha$  is the Biot-Willis constant, $\beta$  is thermal stress coefficient,  $k$ is the permeability and $ \bt=\left(\Theta_{i j}\right)_{i, j=1}^{d}$  is the effective thermal conductivity, $ \mu $ and $ \lambda$  are the Lam¨¦ parameters. The primary variables are the temperature distribution $ T $, displacement $ \mathbf{u}$  and fluid pressure $p $.

In a homogeneous and isotropic medium, where the permeability, viscosity and conductivity $k$ are constants. In this case, the well posedness and regularity of solution of problem \eqref{1-1}--\eqref{1-3} are studied in  \cite{Santos2021}. In this paper, we consider a permeability $k$ which depends on the dilation $\nabla \cdot \bu$, pressure $p$ and temperature $T$. For the convenience of analysis and calculation, we assume that the relationship satisfies a negative exponential function $k=k(\mathbf{u},p,T)=ak_{0}e^{-b \left((\lambda+\mu)\mathrm{div} \mathbf{u}-\alpha p- \beta T\right )}$, where $a$ is a mutation coefficient to account for the increase in conductivity of the material during fracture formation, $k_{0}$ is the initial conductivity, $b$ is
a coupling coefficient which reflects the influence of stress
on the coefficient of conductivity and $\nabla \cdot \bu$  represents the change in fluid content due to the local volume change.

We mainly encounter the following difficulties: The proof of the well-posedness of equations with nonlinear terms $k(\mathbf{u},p,T)\nabla p$. As for the linear case, we can obtain the existence, uniqueness and stability of solutions  \cite{Chen2020}, but nonlinear terms may cause energy to be non-conservative or grow unbounded  in the nonlinear case \cite{Lin2007}. Therefore, we need to find other methods to prove the well-posedness of  nonlinear equations or give the conditions and range for the well-posedness to hold. Then it is difficult for    the numerical solution. The original model is a complex system of multi-field coupling with nonlinear terms that is difficult to solve  directly. Multi-physics  coupling with the solid skeleton
behaving as an incompressible medium may cause locking phenomenon   \cite{Wheeler2007,Wheeler2009}, and complex system may cause matrix ill-conditioning, excessive computation or memory shortage. Therefore, we need to design suitable numerical schemes and algorithms to not only perform numerical stability and optimal error analysis but also  solve the model effectively. 

 This work has two main innovations: the first innovation is to consider the nonlinear situation of the permeability $k$ according to the changes of displacement, pressure, temperature which more matches the actual circumstances; the second innovation is to introduce 3 new variables to reformulate the original model, and by imposing restrictions on the parameter $\lambda$, avoid the occurrence of saddle point problems, so that any finite element pairs can be used for numerical solution of the variables in the reformulated model, avoiding numerical instability and locking phenomenon.

The rest of this paper is organized as follows. In Section \ref{sec-2}, we introduce the mathematical model of thermo-poroelasticity with nonlinear permeability, present the multiphysics reformulation of the model and its weak formulation. and prove the the existence and uniqueness of a weak solution for both the original problem and the reformulated problem. In Section \ref{sec-3}, we propose the fully discrete finite element method and prove its stability and convergence properties. In Section \ref{sec-4}, we provide some numerical results to illustrate the performance of the proposed method and the effect of nonlinear permeability on the thermo-poroelasticity problem. Finally, we conclude with some remarks.
	
\section{reformulation and PDE analysis}\label{sec-2}
\subsection{Preliminaries}
To close the \eqref{1-1}--\eqref{1-3}, we set the following boundary and initial conditions in this paper:
\begin{alignat}{2} \label{2-1}
	(\mu\nabla\mathbf{u}+(\lambda+\mu)\mathrm{div} \mathbf{u}I-\alpha pI-\beta T I) \mathbf{n} &= \mathbf{f}_1
	&&\qquad \mbox{on } \partial\Omega_t:=\partial\Omega\times (0,t_f),\\
	k(\mathbf{u},p,T)\nabla p\cdot \mathbf{n}
	&=g_1 &&\qquad \mbox{on } \partial\Omega_t,\label{2-2} \\
	\bt \nabla T \cdot \bn &= \phi_1 &&\qquad \mbox{on} ~\partial\Omega_t,\\
	\mathbf{u}=\mathbf{u}_0,\qquad p=p_0,\qquad T&=T_0&&\qquad \mbox{in } \Omega\times\{t=0\}. \label{2-3}
\end{alignat}
Introduce new variables
\[
q:=\mathrm{div} \mathbf{u},\quad \varpi:=c_{0}p-b_0 T+\alpha q,\quad 
\tau:=\alpha p -(\lambda+\mu) q+\beta T,\quad 
\varsigma:=a_0 T -b_0 p +\beta q.
\]
In some engineering literature, Lam\'e constant
$\mu$ is also called the {\em shear modulus} and denoted by $G$, and
$B:=\lambda +\frac23 G$ is called the {\em bulk modulus}. $\lambda,~\mu$ and $B$
are computed from the {\em Young's modulus} $E$ and the {\em Poisson ratio}
$\nu$ by the following formulas
\[
\lambda=\frac{E\nu}{(1+\nu)(1-2\nu)},\qquad \mu=G=\frac{E}{2(1+\nu)},\qquad
B=\frac{E}{3(1-2\nu)}.
\]
It is easy to check that
\begin{align}\label{2-4}
	T=\gamma_{1} \tau +\gamma_{2} \varpi+\gamma_{3} \varsigma,\qquad
	p=\gamma_4 \tau + \gamma_5 \varpi+\gamma_2 \varsigma,\qquad 
	q=-\gamma_6 \tau + \gamma_4 \varpi +\gamma_{1} \varsigma,
\end{align}
where 
\begin{align*}
	\centering
	&\gamma_{1}=\frac{\alpha \beta c_0+\alpha^2 b_0}{\mathcal{M}},
	\gamma_{2}=\frac{ab_0 (\lambda+\mu)-\alpha^2 \beta}{\mathcal{M}},
	\gamma_{3}=\frac{\alpha^3+\alpha c_0 (\lambda+\mu)}{\mathcal{M}},\\
	&\gamma_{4}=\frac{a_0 \alpha^2+\alpha \beta b_0}{\mathcal{M}},
	\gamma_{5}=\frac{a_0 \alpha (\lambda+\mu)+\alpha \beta^2}{\mathcal{M}},
	\gamma_{6}=\frac{\alpha c_0 a_0-\alpha b_0^{2}}{\mathcal{M}},\\
	&\mathcal{M}=\alpha c_{0} \beta^{2}+2 \alpha^{2} \beta b_{0}+a_{0} \alpha^{3}+\left(c_{0} a_{0} \alpha-b_{0}^{2} \alpha\right) (\lambda+\mu).
\end{align*}
Then the problem \eqref{1-1}--\eqref{1-3} can be rewritten as
\begin{alignat}{2} \label{2-5}
	-\mu\mathrm{div}\nabla\mathbf{u} + \nabla \tau &= \mathbf{f} &&\qquad \mbox{in } \Omega_t,\\
	\gamma_6\tau +\mathrm{div} \mathbf{u} &=\gamma_4\varpi+\gamma_{1} \varsigma &&\qquad \mbox{in } \Omega_t,\label{2-6}\\
	\varpi_{t} - \mathrm{div}( k(\tau)\nabla (\gamma_4 \tau + \gamma_5 \varpi+\gamma_2 \varsigma))&=g
	&&\qquad \mbox{in } \Omega_t,\label{2-7}\\
	\varsigma_{t}-\mathrm{div}  (\bt(\gamma_{1} \tau +\gamma_{2} \varpi+\gamma_{3} \varsigma)) &=\phi &&\qquad \mbox{in } \Omega_t.\label{2--7}
\end{alignat}
The boundary and initial conditions \eqref{2-1}--\eqref{2-3} can be rewritten as
\begin{alignat}{2} \label{2-8}
(\mu\nabla\mathbf{u}-\tau I) \mathbf{n} &= \mathbf{f}_1
&& \mbox{on } \partial\Omega_t,\\
k(\tau)\nabla(\gamma_4 \tau + \gamma_5 \varpi+\gamma_2 \varsigma)\cdot \mathbf{n}
&=g_1 && \mbox{on } \partial\Omega_t,\label{2-9} \\
\bt \nabla (\gamma_{1} \tau +\gamma_{2} \varpi+\gamma_{3} \varsigma)\cdot \bn &= \phi_1 && \mbox{on } \partial\Omega_t,\label{b}\\
\varpi=\varpi_0,\qquad \varsigma=\varsigma_0,\quad  && \mbox{in } \Omega\times\{t=0\}. \label{a}
\end{alignat}

To prove the existence and uniqueness of solution of the problem \eqref{1-1}--\eqref{1-3}, we need the following assumptions:

$\mathrm{A} 1$: We assume that the permeability $k(\cdot):(\mathbb{R}^d,\mathbb{R},\mathbb{R}) \rightarrow \mathbb{R}$ is continuous and there exist constants $k_m,k_M$ such that 
$
0 <k_m|\zeta|^{2}<k(\bu,p,T)\zeta^{\top}\zeta<k_M|\zeta|^{2},~ \forall  (\bu,p,T) \in (\mathbb{R}^d,\mathbb{R},\mathbb{R}).
$

A2:  $\boldsymbol{\Theta}: \mathbb{R}^{d} \rightarrow \mathbb{R}^{d \times d}$  is symmetric and uniformly positive definite in the sense that there exist positive constants  $\theta_{m}>0 $ and  $\theta_{M}>0 $ such that  $\theta_{m}|\zeta|^{2} \leq \zeta^{\top} \boldsymbol{\Theta}(x) \zeta \leq   \theta_{M}|\zeta|^{2},~\forall \zeta \in \mathbb{R}^{d} \backslash\{0\} $.

A3: The coefficients  $a_{0}, b_{0}, c_{0},$ $\alpha $ and  $\beta $ are non-negative constants, and  $c_{0}-b_{0}>0,  a_{0}-b_{0}>0$.

For convenience, we assume that $\mathbf{f},~ \mathbf{f}_1,~ \phi$
and $\phi_1$ all are independent of $t$ in the remaining of the paper. We note that all the results of this paper can be easily extended to the case of time-dependent source functions.
\begin{remark}
	From  \cite{Karel1982}, it is easy to verify that $k(\bu,p,T)$ is a  Nemytskii operator.
	According to the assumption condition A3, we can see that from \eqref{2-6} under the condition of limited $\lambda$, the divergence of displacement u is not equal to 0, which means that the first two equations of the reformulated poroelasticity equation system no longer have the characteristics of saddle point problem. Therefore, the selection of finite element pairs after reformulation does not need to follow the format requirements of stabilization, but can be flexibly chosen according to the actual situation. We will verify this in the following numerical experiments and show the difference of different finite element pairs on the numerical experiment results.
\end{remark}
\subsection{Definition of weak solution}
 The standard function space notation is adopted in this paper, their
precise definitions can be found in  \cite{bs08,cia,temam}.
In particular, $(\cdot,\cdot)$ and $\langle \cdot,\cdot\rangle$
denote respectively the standard $L^2(\Omega)$ and $L^2(\partial\Omega)$ inner products. For any Banach space $B$, we let $\mathbf{B}=[B]^d$, and use $\mathbf{B}^\prime$ to denote its dual space. In particular,  $\left\|\cdot\right\|_{L^p(B)}$ is a shorthand notation for
$\left\| \cdot\right\|_{L^p(0,t_f;B)}.$\\
We also introduce the function spaces
\begin{align*}
	&L^2_0(\Omega):=\{q\in L^2(\Omega);\, (q,1)=0\},\qquad \mathbf{X}:= \mathbf{H}^1(\Omega).
\end{align*}
From  \cite{temam}, it is well known  that the following  inf-sup condition holds in the space $\mathbf{X}\times L^2_0(\Omega)$:
\begin{align}\label{2-14}
	\sup_{\mathbf{v}\in \mathbf{X}}\frac{(\mathrm{div} \mathbf{v},\varphi)}{\left\|\mathbf{v}\right\|_{H^1(\Omega)}}
	\geq \alpha_0 \left\|\varphi\right\|_{L^2(\Omega)} \qquad \forall
	\varphi\in L^2_0(\Omega),\quad \alpha_0>0.
\end{align}
Let
\[
\mathbf{RM}:=\{\mathbf{r}:=\mathbf{a}+\mathbf{b} \times x;\,\mathbf{a},\mathbf{b}, x\in \mathbb{R}^3;\, \mathbf{r}:=\mathbf{a}+b(x_2-x_1)^t,\mathbf{a} \in \mathbb{R}^2,b\in \mathbb{R}\},
\]
denote the space of infinitesimal rigid motions. It is well known  \cite{bs08, gra, temam} that $\mathbf{RM}$ is the kernel of
the strain operator $\varepsilon$, that is, $\mathbf{r}\in \mathbf{RM}$ if and only if
$\varepsilon(\mathbf{r})=0$. Hence, we have
\begin{align}
	\varepsilon(\mathbf{r})=0,\quad \mathrm{div} \mathbf{r}=0 \qquad\forall \mathbf{r}\in \mathbf{RM}. 
\end{align}
Let $\mathbf{L}^2_\bot(\partial\Omega)$ and $\mathbf{H}^1_\bot(\Omega)$ denote respectively the subspace of $\mathbf{L}^2(\partial\Omega)$ and $\mathbf{H}^1(\Omega)$ which are orthogonal to $\mathbf{RM}$, that is,
\begin{align*}
	&\mathbf{H}^1_\bot(\Omega):=\{\mathbf{v}\in \mathbf{H}^1(\Omega);\, (\mathbf{v},\mathbf{r})=0\,\,\forall \mathbf{r}\in \mathbf{RM}\},
	\\
	&\mathbf{L}^2_\bot(\partial\Omega):=\{\mathbf{g}\in \mathbf{L}^2(\partial\Omega);\,\langle \mathbf{g},\mathbf{r}\rangle=0\,\,
	\forall \mathbf{r}\in \mathbf{RM} \}.
\end{align*}
It is well known  \cite{dautray} that there exists a constant $c_1>0$ such that
\begin{eqnarray}
	\inf_{\mathbf{r}\in \mathbf{RM}}\|\mathbf{v}+\mathbf{r}\|_{L^2(\Omega)}
	\le c_1\|\varepsilon(\mathbf{v})\|_{L^2(\Omega)} \qquad\forall \mathbf{v}\in\mathbf{H}^1(\Omega).\label{2-16}
\end{eqnarray}
From  \cite{fgl14} we know that for each $\mathbf{v}\in \mathbf{H}^1_\bot(\Omega)$ there holds the following alternative version of the inf-sup condition
\begin{eqnarray}
	\sup_{\mathbf{v}\in \mathbf{H}^1_\bot(\Omega)}\frac{(\mathrm{div} \mathbf{v},\varphi)}{\left\| \mathbf{v}\right\|_{H^1(\Omega)}}
	\geq \alpha_1 \left\|\varphi\right\|_{L^2(\Omega)} \qquad \forall
	\varphi\in L^2_0(\Omega),\quad \alpha_1>0.\label{2-17}
\end{eqnarray}

\begin{definition}\label{weak1}
	Let $\mathbf{u}_0\in\mathbf{H}^1(\Omega),~ \mathbf{f}\in\mathbf{L}^2(\Omega),~
	\mathbf{f}_1\in \mathbf{L}^2(\partial\Omega),~ p_0\in L^2(\Omega),~ g,\phi\in L^2(\Omega)$,
	and $g_1,\phi_1\in  L^2(\partial\Omega)$.  Assume $c_0>0$ and
	$(\mathbf{f},\mathbf{v})+\langle \mathbf{f}_1,~ \mathbf{v} \rangle =0$ for any $\mathbf{v}\in \mathbf{RM}$.
	Given $t_f > 0$, a tuple $(\mathbf{u},p,T)$ with
	\begin{alignat*}{2}
		&\mathbf{u}\in L^\infty\bigl(0,t_f; \mathbf{H}_\perp^1(\Omega)),
		&&\qquad p,T\in L^\infty(0,t_f; L^2(\Omega))\cap L^2 \bigl(0,t_f; H^1(\Omega)\bigr),\\
		&p_t, T_t,(\mathrm{div}\mathbf{u})_t \in L^2(0,t_f;H^{1}(\Omega)')
		&&\qquad 
	\end{alignat*}
	is called a weak solution to the problem \eqref{1-1}--\eqref{1-3}, if there hold for almost every $t \in [0,t_f]$
	\begin{alignat}{2}\label{2-15}
		&\mu\bigl(\nabla\mathbf{u},\nabla\mathbf{v} \bigr)+(\lambda+\mu)(\mathrm{div}\mathbf{u},\mathrm{div}\mathbf{v})
		-\alpha( p,\mathrm{div} \mathbf{v})-\beta (T,\mathrm{div} \bv)=(\mathbf{f},\mathbf{v})+\langle \mathbf{f}_1,\mathbf{v}\rangle
		&&\quad\forall \mathbf{v}\in \mathbf{H}^1(\Omega),\\
		&((c_0 p - b_0 T+\alpha\mathrm{div}\mathbf{u})_t,\varphi)
		+ (k(\mathbf{u},p,T)\nabla p,\nabla \varphi )=(g,\varphi)
		+\langle g_1,\varphi \rangle
		&&\quad\forall \varphi \in H^1(\Omega),
		\label{c} \\
		&((a_0 T -b_0 P +\beta \mathrm{div} \bu)_{t},y)+(\bt \nabla T,\nabla y) =(\phi,y)+\left \langle \phi_1,y \right \rangle
		&&\quad\forall \psi \in H^1(\Omega),\label{cc}\\
		&\mathbf{u}(0) = \mathbf{u}_0,\qquad p(0)=p_0.  && \no
	\end{alignat}
\end{definition}
Similarly, we can define the weak solution to the problem \eqref{2-5}--\eqref{a} as follows:
\begin{definition}\label{weak2}
	Let $\mathbf{u}_0\in \mathbf{H}^1(\Omega),\mathbf{f} \in \mathbf{L}^2(\Omega),
	\mathbf{f}_1 \in \mathbf{L}^2(\partial\Omega), p_0\in L^2(\Omega),\phi\in L^2(\Omega)$
	and $\phi_1\in L^2(\partial\Omega)$.  Assume $c_0>0$ and
	$(\mathbf{f},\mathbf{v})+\langle \mathbf{f}_1,\mathbf{v} \rangle =0$ for any $\mathbf{v}\in \mathbf{RM}$.
	Given $t_f > 0$, a $7$-tuple $(\mathbf{u},\tau,\varpi,\varsigma,p,T,q)$ with
	\begin{alignat*}{2}
		&\mathbf{u}\in L^\infty\bigl(0,t_f; \mathbf{H}_\perp^1(\Omega)), &&\qquad
		\tau\in L^\infty \bigl(0,t_f ; L^2(\Omega)\bigr),\\
		&\varpi,\varsigma \in L^\infty\bigl(0,t_f ; L^2(\Omega)\bigr)
		\cap H^1\bigl(0,t_f ; H^{1}(\Omega)'\bigr),
		&&\qquad q\in L^\infty(0,t_f ;L^2(\Omega)),\\
		&p,T\in L^\infty \bigl(0,t_f ; L^2(\Omega)\bigr) \cap L^2 \bigl(0,t_f ; H^1(\Omega)\bigr)  &&
	\end{alignat*}
	is called a weak solution to the problem \eqref{2-5}--\eqref{a},
	if there hold for almost every $t \in [0,t_f]$
	\begin{alignat}{2}\label{2-18}
		&\mu(\nabla\mathbf{u},\nabla\mathbf{v})-(\tau,\mathrm{div}\mathbf{v})= (\mathbf{f},\mathbf{v})+\langle \mathbf{f}_1,\mathbf{v}\rangle
		&&\quad\forall \mathbf{v}\in \mathbf{H}^1(\Omega),\\
		&\gamma_6(\tau,\varphi) +(\mathrm{div} \mathbf{u},\varphi) -\gamma_4(\varpi,\varphi)-\gamma_{1} (\varsigma,\varphi)=0 &&\quad\forall \varphi \in L^2(\Omega),\label{2-19}  \\
		&\bigl(\varpi_t, y \bigr)+(k(\tau)\nabla p,\nabla y \bigr)= (g, y)+\langle g_1,y\rangle &&\quad\forall y \in H^1(\Omega),\label{2-20} \\
		&(\varsigma_{t},z)+( \bt\nabla T,\nabla z)=(\phi,z)+\left \langle \phi_1,z \right \rangle &&\quad\forall z \in H^1(\Omega),\label{2-22}
	\end{alignat}
	where $q_0:=\mathrm{div} \mathbf{u}_0$, $\mathbf{u}_0$ and $p_0$ are same as in Definition \ref{weak1}.
\end{definition}
The bilinear form $a(\bu,\bv)$ is defined by
\begin{alignat}{2}
	a(\bu,\bv):=(\lambda+\mu)(\nabla \cdot \bu,\nabla\cdot\bv)+\mu(\nabla\bu,\nabla\bv).
\end{alignat}
It is easy to check that the bilinear form $a(\cdot,\cdot)$ is continuous and coercive on $H_0^1(\Omega)$. Thus, for a fixed $t \in I$ and $p(\cdot,t),T(\cdot,t)\in L^2(\Omega),$ the first equation of \eqref{2-15} can be solved for $\bu$.  Define the associated operator $B: L^2(\Omega)\longrightarrow  L^2(\Omega)$ for $p,T\in L^2(\Omega), B(\alpha p+ \beta T):=\nabla \cdot \bu$, where $\bu$ satisfies
\begin{align}
	a(\bu,\bv)&=(\nabla(\alpha p + \beta T),\bv) \quad \forall \bv\in H_0^1(\Omega)\no,\\
	\bu&=0\quad  {\rm on} ~~\partial\Omega.\label{eq20231008-1}
\end{align}
As for the theoretical results about $B$ operator, from \cite{cao}, we know that
\begin{lemma}\label{lem2.5}
	The operator B defined above is linear, continuous, monotone, and self-adjoint with $Ker(B)=Ker(\nabla)$ and $Rng(B)=Ker(\nabla)^{\perp}$; there exist constants $C_{B_0},C_{B_1}>0$ such that
	\begin{alignat}{2}
		\left \|B(p,T)   \right \|&\le C_{B_0}\left \| \alpha p+\beta T \right \|  
		\quad \forall p,T\in L^2(\Omega),\no\\
		\left \|B(p,T)   \right \|_1 &\le C_{B_1}\left \| \alpha p+\beta T \right \|_1 
		\quad \forall p,T\in H_0^1(\Omega).
	\end{alignat}
	\end{lemma}

Using \eqref{c}--\eqref{cc}  and \eqref{eq20231008-1}, we obtain 
\begin{alignat}{2}
	((c_0 p-b_0 T+\alpha B(p,T))_t,\varphi)+(k(B(p,T),p,T)\nabla p,\nabla\varphi)&=(g,\varphi)	\quad\forall \varphi \in H_0^1(\Omega),\label{2-43}\\
	((a_0 T -b_0 P +\beta B(p,T))_{t},y)+(\bt \nabla T,\nabla y) &=(\phi,y)
	\quad\forall y \in H_0^1(\Omega),\label{2-44}\\
	a_0 T(\cdot,0)-b_0 p(\cdot,0)+\beta B(p,T)(\cdot,0)&=l_p,\label{2-45}\\
	a_0 T(\cdot,0) -b_0 P(\cdot,0) +\beta B(p,T)(\cdot,0)&=l_T.\label{2-46}
\end{alignat}

\subsection{Existence and uniqueness} To prove the existence of a weak solution, we firstly use 
 backward Euler approximation of the time derivative and use the modified Rothe's method to construct a convergent sequence of approximate solutions of \eqref{2-43}--\eqref{2-44}. Let $m=t_f/n$ for some positive integer $n$. Partition $I$ uniformly with time step $k$ and denote nodal points by $t_i=t_i^n=im$, for $i=0,1,\ldots, n.$  Let  $  p_{n}^{0},T_{n}^{0} \in L^{2}(\boldsymbol{\Omega})$ satisfy \eqref{2-45}--\eqref{2-46}   and define
\begin{alignat}{2}
	&g_{n}^{i}:=\frac{1}{m} \int_{t_{i-1}}^{t_{i}} g(t) d t,\no\\
	&\delta p_{n}^{i}:=\left(p_{n}^{i}-p_{n}^{i-1}\right) / m,\quad 
	&&\delta T_{n}^{i}:=\left(T_{n}^{i}-T_{n}^{i-1}\right) / m \quad i=1,2,\ldots, n,\label{2-47}\no\\
	&\delta \varpi_{n}^{i}:=\left(\varpi_{n}^{i}-\varpi_{n}^{i-1}\right) / m,\quad 
	&&\delta \varsigma_{n}^{i}:=\left(\varsigma_{n}^{i}-\varsigma_{n}^{i-1}\right) / m \quad i=1,2,\ldots, n.
\end{alignat}
We apply the following scheme inductively to obtain a sequence  $p_{n}^{i},T_{n}^{i}, i=1,2,\cdots, n $
\begin{alignat}{2}
	(\delta (c_0 p_n^i-b_0 T_n^i+\alpha B(p_n^i,T_n^i)),\varphi)+(k(B(p_n^i,T_n^i),p_n^i,T_n^i)\nabla p_n^i,\nabla\varphi)&=(g,\varphi)	\quad\forall \varphi \in H_0^1(\Omega),\label{1-49}\\
	(\delta (a_0 T_n^i -b_0 p_n^i +\beta B(p_n^i,T_n^i)),y)+(\boldsymbol{\Theta} \nabla T_n^i,\nabla y) &=(\phi,y)
	\quad\forall y \in H_0^1(\Omega). \label{1-50}
\end{alignat}
Multiplying \eqref{1-49} and \eqref{1-50} by $m$, and defining
\begin{alignat}{2}
	&(\overline{g},y):=m(g,y)-(c_0 p_n^{i-1}-b_0 T_n^{i-1}+\alpha B(p_n^{i-1},T_n^{i-1}),y),\no
	\\
	&(\overline{\phi},z):=m(\phi,z)-(a_0 T_n^{i-1} -b_0 p_n^{i-1} +\beta B(p_n^{i-1},T_n^{i-1}),y), \no
\end{alignat}
we have that
\begin{alignat}{2}
		(c_0 p_n^i-b_0 T_n^i+\alpha B(p_n^i,T_n^i),\varphi)+m(k(B(p,T_n^i),p_n^i,T_n^i)\nabla p_n^i,\nabla\varphi)&=(\overline{g},\varphi)	\quad\forall \varphi \in H_0^1(\Omega),\label{2-50}\\
	(a_0 T_n^i -b_0 p_n^i +\beta B(p_n^i,T_n^i),y)+m(\bt \nabla T_n^i,\nabla y) &=(\overline{\phi},y)
	\quad\forall y \in H_0^1(\Omega). \label{2-51}
\end{alignat}
\begin{lemma}
	Given $p_n^{i-1},T_n^{i-1}\in L^2(\Omega)$ and $g_n^{i},\varphi_n^{i}\in L^2(\Omega)$,  \eqref{2-50}--\eqref{2-51} has a weak solution $p_n^{i},T_n^{i}$.
\end{lemma}
\begin{proof}
For this purpose, we let  $\left\{w_{i}\right\}_{i=1}^{\infty}$  be an orthogonal basis of  $H_{0}^{1}(\boldsymbol{\Omega})$. Let's   $ V_{m}$ denote the finite dimensional space spanned by  $\left\{w_{1}, w_{2},\ldots, w_{m}\right\} $, $\mathbf{q}:=(p,T)$, $\mathbf{w}:= (w_1,w_2)$. Define the mapping  $\Phi_{m}: V_{m}\times V_{m} \rightarrow V_{m}\times V_{m}$  by
\begin{alignat}{2}
	(\Phi_{m}\mathbf{q},\mathbf{w})=( c_0 p-b_0 T+\alpha B(p,T),w_1)+m(k(B(p,T),p,T)\nabla p,\nabla w_1)-(\overline{g},w_1)\label{2-53}\\
	+	(a_0 T -b_0 p +\beta B(p,T),w_2)+m(\bt \nabla T,\nabla w_2) -(\overline{\phi},w_2).\quad \no
\end{alignat}
From \eqref{lem2.5} and Poincare inequality, we have that
\begin{alignat}{2}
	(\Phi_{m}\mathbf{q},\mathbf{q})&=( c_0 p-b_0 T+\alpha B(p,T),p)+m(k(B(p,T),p,T)\nabla p,\nabla p)-(\overline{g},p)\no\\
	&+	(a_0 T -b_0 p +\beta B(p,T),T)+m(\boldsymbol{\Theta} \nabla T,\nabla T) -(\overline{\phi},T)\no\\
	&\ge -b_0 (T,p)+c_0 \left \| p \right \|^2+ \alpha (B(p,T),p)+mk_m(\nabla p,\nabla p)-\left \| \overline{g} \right \| \left \| p \right \| \no\\
	&+a_0\left \| T \right \|^2 -b_0(p,T)+\beta (B(p,T),T)+m \theta_m \left \| \nabla T\right \|^2 -
	\left \| \overline{\phi} \right \|\left \| T \right \| \no\\
	&\ge c_0 \left \| p \right \|^2+a_0\left \| T \right \|^2-2b_0 \left \|  T\right \| \left \|  p\right \|+\frac{mk_m}{C_p}\left \| p \right \|_1^2+\frac{m\theta_m}{C_T}\left \| T \right \|_1^2\no\\
	&-\left \| \overline{g} \right \| \left \| p \right \|_1-\left \| \overline{\phi} \right \|\left \| T \right \|_1 \no\\
	&\ge  \frac{mk_m}{C_p}\left \| p \right \|_1^2+\frac{m\theta_m}{C_T}\left \| T \right \|_1^2
	-\left \| \overline{g} \right \| \left \| p \right \|_1-\left \| \overline{\phi} \right \|\left \| T \right \|_1.
\end{alignat}
Thus $(\Phi_{m}\mathbf{q},\mathbf{q})\ge 0$ for all $q,T$ with $ \frac{mk_m}{C_p}\left \| p \right \|_1^2+\frac{m\theta_m}{C_T}\left \| T \right \|_1^2=\left \| \overline{g} \right \| \left \| p \right \|_1+\left \| \overline{\phi} \right \|\left \| T \right \|_1 $. Because $V_{m}\times V_{m}$ is finite dimensional and $\Phi$ defined above is continuous, from a well known corollary of Browser's fixed point theorem, see  \cite{gra}, there exists $p_n^i,T_n^i$ in $V_{m}\times V_{m}$ such that 
\begin{align*}
 \frac{mk_m}{C_p}\left \| p_n^i \right \|_1^2+\frac{m\theta_m}{C_T}\left \| T_n^i \right \|_1^2
\le  \left \| \overline{g} \right \| \left \| p_n^i \right \|_1+\left \| \overline{\phi} \right \|\left \| T_n^i \right \|_1, 
\end{align*}
and $p_n^i,T_n^i $ satisfies $\Phi_{m}\mathbf{q}=0$, i.e.
\begin{alignat}{2}
	( c_0 p_n^i-b_0 T_n^i+\alpha B(p_n^i,T_n^i),w_1)+m(k(B(p_n^i,T_n^i),p_n^i,T_n^i)\nabla p_n^i,\nabla w_1)-(\overline{g},w_1)=0\quad \forall w_1 \in V_j \quad j \le m\no,\\
	(a_0 T_n^i -b_0 p_n^i +\beta B(p_n^i,T_n^i),w_2)+m(\bu \nabla T_n^i,\nabla w_2) -(\overline{\phi},w_2)=0
	\quad \forall w_2 \in V_j\quad j \le m\no.
\end{alignat}	
Define $\hat{p}:=p_n^i,~\hat{T}:=T_n^i.$
 Since $\left \{ \op_m,\ot_m \right \}_{m=1} ^{\infty}$ is a uniformly bounded sequence in $H_0^1(\Ome)$, there exists a sub-sequence of $\left \{ \op_m,\ot_m \right \}_{m=1} ^{\infty}$, and a function $\op,\ot \in H_0^1(\Ome)$ such that 
\begin{alignat}{2}
	&\op_m \longrightarrow  \op \quad{\rm in}\quad H_0^1(\Ome)\no,\\
	&\ot_m \longrightarrow  \ot \quad{\rm in}\quad H_0^1(\Ome).\label{2-48}
\end{alignat}
Due to the compact embedding of $H_0^1(\Ome) \hookrightarrow \hookrightarrow L^2(\Ome) $, we have that
\begin{align}
	&\op_m \longrightarrow  \op \quad{\rm in}\quad L^2(\Ome)\no,\\
	&\ot_m \longrightarrow  \ot \quad{\rm in}\quad L^2(\Ome)\label{2-49}.
\end{align}

Then we  show that the weak limit $\op,\ot$ is a solution of \eqref{2-50}--\eqref{2-51}. Choosing a test function $ w \in W^{1,\infty}(\boldsymbol{\Omega}) \cap H_{0}^{1}(\boldsymbol{\Omega}) $ and
applying the Cauchy-Schwartz inequality, we have that
\begin{alignat}{2}
	&\left | (c_0\op_m -b_0 \ot_m +\alpha B(\op_m,\ot_m),w)-(c_0\op -b_0 \ot +\alpha B(\op,\ot),w) \right | \no\\
	&\le \Norm{c_0(\op_m-\op) -b_0 (\ot_m-\ot) +\alpha B(\op_m-\op,\ot_m-\ot)}{}\Norm{w}{}\no\\
	&\le C\Norm{\op_m-\op}{}\Norm{\ot_m-\ot}{}\Norm{w}{}\no\\
	&\rightarrow 0 \quad m \rightarrow \infty \label{1-61},
\end{alignat}
\begin{alignat}{2}
	&\left |  (\ow_m,w)-(\ow,w)\right | 
	& \leq \Norm{\ow_m-\ow}{} \Norm{w}{}
	&\rightarrow 0 \quad m \rightarrow \infty.
\end{alignat}
From the boundary of $B$, the Nemytskii property of $k$, and \eqref{2-49} we have that
\begin{alignat}{2}
	k(B(\op_m,\ot_m),\op_m,\ot_m)&\rightarrow k(B(\op,\ot),\op,\ot) \quad m \rightarrow \infty\label{2-60}.
\end{alignat}
Defining
\begin{align*}
	\hat{k}&:=k(B(\op,\ot),\op,\ot),\\
	\hat{k}_m&:=k(B(\op_m,\ot_m),\op_m,\ot_m),
\end{align*}
 we can write \eqref{2-60} as
\begin{alignat}{2}
	\hat{k}_m\rightarrow \hat{k} \quad m \rightarrow \infty\label{2-61}.
\end{alignat}
Then from \eqref{1-61}--\eqref{2-61},  we have that
\begin{alignat}{2}
	&\left |  (\hat{k}_m\nabla \op_m,\nabla w)-(\hat{k}\nabla \op,\nabla w)\right | \no\\
	\le &\left |  (\hat{k}_m\nabla \op_m,\nabla w)-(\hat{k}_m\nabla \op,\nabla w)\right |+\left |  (\hat{k}_m\nabla \op,\nabla w)-(\hat{k}\nabla \op,\nabla w)\right |\no\\
	\le& \left |  (\hat{k}_m\nabla (\op_m-\op),\nabla w)\right |+\left |  ((\hat{k}_m-\hat{k})\nabla \op_m,\nabla w)\right |\no\\
	\rightarrow& 0 \qquad m \rightarrow \infty\label{2-55}.
\end{alignat}
Through the above conclusion of \eqref{2-48}--\eqref{2-55} and $W^{1,\infty}(\boldsymbol{\Omega}) \cap H_{0}^{1}(\boldsymbol{\Omega})$ is dense in $H_{0}^{1}(\boldsymbol{\Omega})$ we have that
\begin{alignat}{2}
	\lim_{x \to \infty} &(c_0\op_m -b_0 \ot_m +\alpha B(\op_m,\ot_m),w)+(\hat{k}_m\nabla \op_m,\nabla w)\no\\
	=&(c_0\op -b_0 \ot +\alpha B(\op,\ot),w)+(\hat{k}\nabla \op,\nabla w)\qquad\forall w \in H_{0}^{1}(\boldsymbol{\Omega})\label{1-68}.
\end{alignat}
Using the same method, we have that
\begin{alignat}{2}
	\lim_{x \to \infty} &(a_0\ot_m -b_0 \op_m +\beta B(\op_m,\ot_m),v)+(\bt \nabla \ot_m,\nabla v)\no\\
	=&(a_0\ot -b_0 \op +\beta B(\op,\ot),v)+(\bt \nabla \ot,\nabla v)\qquad\forall v \in H_{0}^{1}(\boldsymbol{\Omega})\label{1-70}.
\end{alignat}
Combining \eqref{1-68}--\eqref{1-70}, we have that
\begin{alignat}{2}
	(c_0\op -b_0 \ot +\alpha B(\op,\ot),w)+(\hat{k}\nabla \op,\nabla w)=(\overline{g},w)\qquad\forall w \in H_{0}^{1}(\boldsymbol{\Omega})\label{4-71},\\
	(a_0\ot -b_0 \op +\beta B(\op,\ot),v)+(\bt\nabla \ot,\nabla v)=(\overline{\phi},v)\qquad\forall v \in H_{0}^{1}(\boldsymbol{\Omega}).\label{4-72}
\end{alignat}
Since finite linear combinations of $\hat{p},\hat{T}$ are dense in $H_{0}^{1}(\boldsymbol{\Omega})$, we conclude that $\hat{p},\hat{T}$ is a solution of \eqref{2-50}--\eqref{2-51}.
\end{proof}
\subsubsection{Existence of the original model}\label{sub2.3.1}
In this section, we prove the existence of  solution to the original model.
\begin{lemma}
	There exists a constant $C$ independent of $n$ such that
	\begin{align}
		m \sum_{i=1}^{n} \left \| p_i^n \right \|_1^2 \le C,\qquad\qquad m \sum_{i=1}^{n} \left \| T_i^n \right \|_1^2 \le C.
	\end{align}
\end{lemma}
\qquad{\em Proof.}
	Let's substitute $w_1=p_n^i,w_2=T_n^i$ into \eqref{1-49}--\eqref{1-50} and multiply by $m$, then we have that 
	\begin{alignat}{2}
		(c_0 (p_n^i-p_n^{i-1})-b_0 (T_n^i-T_n^{i-1})+\alpha B(p_n^i-p_n^{i-1},T_n^i-T_n^{i-1}),p_n^i)\label{1-558}\\
		+m(k(B(p_n^i,T_n^i)\nabla p_n^i,\nabla p_n^i)&=m(g_n^i,p_n^i),\no\\
		(a_0 (T_n^i-T_n^{i-1}) -b_0 (p_n^i-p_n^{i-1}) +\beta B(p_n^i-p_n^{i-1},T_n^i-T_n^{i-1}),T_n^i)\no\\
		+m(\bt \nabla T_n^i,\nabla T_n^i) &=m(\phi_n^i,T_n^i).\label{1-59}
	\end{alignat}
	Summing from $i=1$ to $n$  and adding \eqref{1-558}--\eqref{1-59}, we  get
	\begin{alignat}{2}
		&\sum_{i=1}^{n} (c_0 (p_n^i-p_n^{i-1})-b_0 (T_n^i-T_n^{i-1})+\alpha B(p_n^i-p_n^{i-1},T_n^i-T_n^{i-1}),p_n^i)\no\\
		+&\sum_{i=1}^{n}(a_0 (T_n^i-T_n^{i-1}) -b_0 (p_n^i-p_n^{i-1}) +\beta B(p_n^i-p_n^{i-1},T_n^i-T_n^{i-1}),T_n^i)\no\\
		+&m\sum_{i=1}^{n}(k(B(p_n^i,T_n^i)\nabla p_n^i,\nabla p_n^i)+m\sum_{i=1}^{n}(\bt \nabla T_n^i,\nabla T_n^i)	 \no\\
		=&m\sum_{i=1}^{n}(\phi_n^i,T_n^i)+m\sum_{i=1}^{n}(g_n^i,p_n^i),\label{1-60}
	\end{alignat}
	so  we have that
	\begin{align}
		m k_m \sum_{i=1}^{n} \left \| p_i^n \right \|_1^2+m \theta_m \sum_{i=1}^{n} \left \| T_i^n \right \|_1^2&\le 
		m\sum_{i=1}^{n}(\phi_n^i,T_n^i)+m\sum_{i=1}^{n}(g_n^i,p_n^i)\no\\
		&\le C.
	\end{align}

For $ \varphi $ in $ C^{\infty}(\bar{I}) $, we define two piece-wise constant functions $ \varphi_{n}$  and  $\tilde{\varphi}_{n}$  such that
	\begin{align}
		\varphi_{n}(t)&=\varphi\left(t_{i}\right) \quad \text { for } t \in\left(t_{i-1}, t_{i}\right],\label{2-666}\\
		\tilde{\varphi}_{n}(t)&=\frac{\varphi\left(t_{i+1}\right)-\varphi\left(t_{i}\right)}{k} \quad \text { for } t \in\left(t_{i-1}, t_{i}\right],
	\end{align}
	 for $ i=1,2,\ldots, n $,  and
	\begin{align}{c}
		\varphi_{n}(0)&=\varphi\left(t_{1}\right),\\
		\tilde{\varphi}_{n}(0)&=\tilde{\varphi}\left(t_{1}\right),\quad \text { and } \tilde{\varphi}_{n}\left(t_{n+1}\right)=\tilde{\varphi}\left(t_{n}\right) \label{2-667}.
	\end{align}
	Then
	$$\left\|\varphi_{n}-\varphi\right\|_{L^{2}(I)} \leq C(\varphi) k ~ \text { and } ~\left\|\tilde{\varphi_{n}}-\varphi^{\prime}\right\|_{L^{2}(I)} \leq C(\varphi) k^{1 / 2},$$
	where the constant $ C(\varphi) $ depends only on $ \varphi^{\prime} $ and  $\varphi^{\prime \prime}.$\label{lem2,9}
\begin{lemma}
	(Aubin-Lions, see  \cite{Lions1969}). Let $X_0, X, X_1$ be three Banach spaces with
	$X_0 \subset X \subset X_1$. Suppose that $X_0$ is compactly embedded in $X$ and $X$ is continuously
	embedded in $X_1$. Furthermore, assume that $X_0$ and $X_1$ are reflexive spaces. Let
	$1 < p < \infty, 1 < q < \infty.$ Define
	\begin{alignat}{2}
		W=\left\{u \in L^{p}\left([0, t_f] ; X_{0}\right): u^{\prime} \in L^{q}\left([0, t_f] ; X_{1}\right)\right\}.
	\end{alignat}
\end{lemma}
Then the embedding of  $W $ into  $L^{p}([0, t_f] ; X)$  is compact.
\begin{remark}
	Specifically, let  $X_{0}=H_{0}^{1}(\boldsymbol{\Omega}), X=L^{2}(\boldsymbol{\Omega}), X_{1}=H^{-1}(\boldsymbol{\Omega}), I=[0, t_f] $ and  $p=q=2$, then
\begin{align*}
		W=\left\{u \in L^{2}\left(I ; H_{0}^{1}(\boldsymbol{\Omega})\right): u^{\prime} \in L^{2}\left(I ; H^{-1}(\boldsymbol{\Omega})\right)\right\} \hookrightarrow \hookrightarrow L^{2}\left(I ; L^{2}(\boldsymbol{\Omega})\right)
\end{align*}
(where the above embedding is compact).
\end{remark}

Armed with the above lemmas, we are now ready to show the existence of a
solution of equation \eqref{2-15}--\eqref{cc}.
\begin{theorem}
	Given $p_0,T_0$ and $g,\phi $, the problem \eqref{2-15}--\eqref{cc} have at least one weak solution $(p,T)$.
\end{theorem}
\begin{proof}
	In the previous content, we have proved that \eqref{4-71}--\eqref{4-72} have at least one solution. For each positive integer $n$, define $P_n,T_n \in L^2(I,H_0^1(\Omega))$ as
	\begin{align}
		P_n(t)=p_n^i\qquad {\rm for } ~~t\in (t_{i-1},t_i] \quad i=1,2,\ldots,n,\no\\
		T_n(t)=T_n^i\qquad {\rm for } ~~t\in (t_{i-1},t_i] \quad i=1,2,\ldots,n.
	\end{align}
Recall that $\left \{ P_n\right \}_{n=1} ^{\infty},\left \{ T_n\right \}_{n=1} ^{\infty}$ is bounded in $L^2(I,H_0^1(\Omega))$,  hence there exist a pair $p, T \in L^2(I,H_0^1(\Omega))$ such that 
\begin{align}
	&P_n \to p \quad{\rm in }~~L^2(I,H_0^1(\Omega)),\\
	&	T_n \to T \quad{\rm in }~~L^2(I,H_0^1(\Omega)).
\end{align}
Multiplying \eqref{1-49}--\eqref{1-50} by $m \phi(t_i)$, summing up from $i=1$ to $n$, and using summation by parts, we obtain 
\begin{align}
	\sum_{i=1}^{n} (c_0 (p_n^i-p_n^{i-1})-b_0 (T_n^i-T_n^{i-1})+\alpha( B(p_n^i-p_n^{i-1},T_n^i-T_n^{i-1}),w_1)\varphi(t_i)\no\\
	=(c_0 p_n^n-b_0 T_n^n+\alpha B(p_n^n,T_n^n),w_1)\varphi(t_n)-(c_0 p_n^0-b_0 T_n^0+\alpha B(p_n^0,T_n^0),w_1)\varphi(t_1)
	\no\\-m \sum_{i=1}^{n-1}(c_0 p_n^i-b_0 T_n^i+\alpha B(p_n^i,T_n^i),w_1)(\varphi(t_{i+1})-\varphi(t_i))/m).
\end{align}
Since $\varphi(t_n)=0$, letting $n \to \infty$, $\varphi(t_1)=\varphi(m)=0 $, we have that
\begin{alignat}{2}
	-\sum_{i=1}^{n-1}( c_0 p_n^i-b_0 T_n^i+\alpha B(p_n^i,T_n^i),w_1)(\varphi(t_{i+1})-\varphi(t_i))\no\\+m\sum_{i=1}^{n}(k(B(p_n^i,T_n^i),p_n^i,T_n^i)\nabla p_n^i,\nabla w_1)\varphi(t_i)&=m\sum_{i=1}^{n}(g_n^i,w_1)\varphi(t_i),	 \no\\
	-\sum_{i=1}^{n-1}( a_0 T_n^i -b_0 p_n^i +\beta B(p_n^i,T_n^i),w_2)(\varphi(t_{i+1})-\varphi(t_i))\no\\+m\sum_{i=1}^{n}(\bt \nabla T_n^i,\nabla w_2)\varphi(t_i) &=m\sum_{i=1}^{n}(\phi,w_2)\varphi(t_i).
\end{alignat}
Then since $P_n,T_n,\varphi $ is all piecewise continuous, we can write the above equation as
\begin{align}
		-\int_{0}^{t_f} ( c_0 P_n-b_0 T_n+\alpha B(P_n,T_n),w_1)\tilde{\varphi}_n dt\no\\+\int_{0}^{t_f} (k(B(P_n,T_n),P_n,T_n)\nabla P_n,\nabla w_1)\varphi_ndt&=\int_{0}^{t_f} (g_n^i,w_1)\varphi_n	dt,\no\\
	-\int_{0}^{t_f} ( a_0 T_n -b_0 P_n +\beta B(P_n,T_n),w_2)\tilde{\varphi}_n dt\no\\+\int_{0}^{t_f} (\bt \nabla T_n,\nabla w_2)\varphi_ndt &=\int_{0}^{t_f} (\phi,w_2)\varphi_ndt.
\end{align}
Notice that 
\begin{align*}
	&\int_{0}^{t_f} (c_0 P_n-b_0 T_n+\alpha B(P_n,T_n),w_1)\tilde{\varphi}_n dt\\
	=&\int_{0}^{t_f} (c_0 P_n-b_0 T_n+\alpha B(P_n,T_n),w_1)(\tilde{\varphi}_n- \varphi') dt+\int_{0}^{t_f} (c_0 P_n-b_0 T_n+\alpha B(P_n,T_n),w_1){\varphi}' dt,\\
		&\int_{0}^{t_f} (a_0 T_n -b_0 P_n +\beta B(P_n,T_n),w_2)\tilde{\varphi}_n dt\\
	=&\int_{0}^{t_f} (a_0 T_n -b_0 P_n +\beta B(P_n,T_n),w_2)(\tilde{\varphi}_n- \varphi') dt+\int_{0}^{t_f} (a_0 T_n -b_0 P_n +\beta B(P_n,T_n),w_2){\varphi}' dt.
\end{align*}
From \eqref{2-666}--\eqref{2-667}, we have that
\begin{align}
	\left |  \int_{0}^{t_f} ( c_0 P_n-b_0 T_n+\alpha B(P_n,T_n),w_1)(\tilde{\varphi}_n- \varphi') dt\right | \le \Norm{w_1}{} \Norm{P_n}{} \Norm{T_n}{} \Norm{\tilde{\varphi}_n- \varphi'}{L^2(I)}\to 0,\no\\
	\left |  \int_{0}^{t_f} ( a_0 T_n -b_0 P_n +\beta B(P_n,T_n),w_2)(\tilde{\varphi}_n- \varphi') dt\right | \le \Norm{w_1}{} \Norm{P_n}{} \Norm{T_n}{} \Norm{\tilde{\varphi}_n- \varphi'}{L^2(I)}\to 0,\no
\end{align}
and  let $n \to \infty$
\begin{align*}
	\int_{0}^{t_f} (c_0 P_n-b_0 T_n+\alpha B(P_n,T_n),w_1){\varphi}' dt \to \int_{0}^{t_f} (c_0 p-b_0 t+\alpha B(p,t),w_1){\varphi}' dt\\
	=-\int_{0}^{t_f} (c_0 p-b_0 t+\alpha B(p,t)',w_1){\varphi} dt.
\end{align*}
Hence, we have
\begin{align*}
		-\int_{0}^{t_f} (c_0 P_n-b_0 T_n+\alpha B(P_n,T_n),w_1)\tilde{\varphi}_n dt \to \int_{0}^{t_f} (c_0 p-b_0 t+\alpha B(p,t)',w_1){\varphi} dt,\\
				-\int_{0}^{t_f} (a_0 T_n -b_0 P_n +\beta B(P_n,T_n),w_2)\tilde{\varphi}_n dt \to \int_{0}^{t_f} (a_0 t -b_0 t +\beta B(p,t),w_2)'{\varphi} dt.
	\end{align*}
Letting $k_n$ denote $k(B(P_n,T_n),P_n,T_n)$, we can get
\begin{align*}
	&\int_{0}^{t_f} (k_n\nabla P_n,\nabla w_1)\varphi_ndt-\int_{0}^{t_f} (k\nabla P,\nabla w_1)\varphi dt\\
	= & \int_{0}^{t_f} (k_n\nabla P_n,\nabla w_1)(\varphi_n-\varphi)dt+\int_{0}^{t_f} (k\nabla (P_n-p),\nabla w_1)\varphi dt\\
	+& \int_{0}^{t_f} ((k_n-k)\nabla P_n,\nabla w_1)\varphi_ndt\\
	\le& k_M \Norm{w_1}{} \Norm{P_n}{} \Norm{T_n}{} \Norm{\tilde{\varphi}_n- \varphi'}{L^2(I)}+C\Norm{k_n-k}{L^2(I;L^2(\Omega))} \Norm{P_n}{} \Norm{T_n}{}\\
	\to &0,\\
	&\int_{0}^{t_f} (\bt\nabla T_n,\nabla w_2)\varphi_ndt-\int_{0}^{t_f} (\bt\nabla t,\nabla w_2)\varphi dt\\
	=&\int_{0}^{t_f} \bt(\nabla (T_n-t),\nabla w_2)dt\\
	\to 0.
\end{align*}
It is straightforward to show that 
\begin{align*}
	\int_{0}^{t_f} (g_n,w_1)\varphi_{n}dt=
	\int_{0}^{t_f} (g(t),w_1)\varphi_{n}(t)dt
	\to \int_{0}^{t_f} (g,w_1)\varphi dt,\\
		\int_{0}^{t_f} (\phi_n,w_2)\varphi_{n}dt=
	\int_{0}^{t_f} (\phi(t),w_2)\varphi_{n}(t)dt
	\to \int_{0}^{t_f} (\phi,w_2)\varphi dt.
\end{align*}
Letting $n \to \infty$, using the above conclusion, we conclude that for any test function $\hat{\bw}$ of the form $\bw \varphi$ satisfying \eqref{2-43}--\eqref{2-46}. The proof is complete.
\end{proof}

\subsubsection{Existence of reformulated model solution}
In this section, we prove the existence of the reformulated model.

From subsection \ref{sub2.3.1}, we can find a series construct sequences $ \left\lbrace \bu_n\right\rbrace _{n=0}^{\infty}\subset H^{1}(\Omega) $, $ \left\lbrace p_n\right\rbrace _{n=0}^{\infty}\subset L^{2}(\Omega) $, $ \left\lbrace T_n\right\rbrace _{n=0}^{\infty}\subset L^{2}(\Omega) $ which satisfy 
	\begin{alignat}{2}
	&\mu\bigl(\nabla\mathbf{u_n},\nabla\mathbf{v} \bigr)+(\lambda+\mu)(\mathrm{div}\mathbf{u_n},\mathrm{div}\mathbf{v})
	-\alpha( p_n,\mathrm{div} \mathbf{v})-\beta (T_n,\mathrm{div} \bv)=(\mathbf{f},\mathbf{v})+\langle \mathbf{f}_1,\mathbf{v}\rangle
	&&\quad\forall \mathbf{v}\in \mathbf{H}^1(\Omega),\no\\
	&((c_0 p_n - b_0 T_n+\alpha\mathrm{div}\mathbf{u_n})_t,\varphi)
	+ (k(\mathbf{u_n},p_n,T_n)\nabla p_n,\nabla \varphi )=(g,\varphi)
	+\langle g_1,\varphi \rangle
	&&\quad\forall \varphi \in H^1(\Omega),\no\\
	&((a_0 T -b_0 P_n +\beta \mathrm{div} \mathbf{u_n})_{t},y)+(\boldsymbol{\Theta} \nabla T_n,\nabla y) =(\phi,y)+\left \langle \phi_1,y \right \rangle
	&&\quad\forall \psi \in H^1(\Omega), \no
\end{alignat}
\begin{align*}
		\lim_{n\rightarrow \infty}\nabla\cdot\bu_n=\nabla\cdot\bu,\quad \lim_{n\rightarrow \infty}\bu_n=\bu,\quad 
	\lim_{n\rightarrow \infty}p_n=p,\quad \lim_{n\rightarrow \infty}T_n=T,
\end{align*}
and $\bu,p,T$ satisfy
	\begin{alignat}{2}
	&\mu\bigl(\nabla\mathbf{u},\nabla\mathbf{v} \bigr)+(\lambda+\mu)(\mathrm{div}\mathbf{u},\mathrm{div}\mathbf{v})
	-\alpha( p,\mathrm{div} \mathbf{v})-\beta (T,\mathrm{div} \bv)=(\mathbf{f},\mathbf{v})+\langle \mathbf{f}_1,\mathbf{v}\rangle
	&&\quad\forall \mathbf{v}\in \mathbf{H}^1(\Omega),\no\\
	&((c_0 p - b_0 T+\alpha\mathrm{div}\mathbf{u})_t,\varphi)
	+ (k(\mathbf{u},p,T)\nabla p,\nabla \varphi )=(g,\varphi)
	+\langle g_1,\varphi \rangle
	&&\quad\forall \varphi \in H^1(\Omega),\no\\
	&((a_0 T -b_0 P +\beta \mathrm{div} \bu)_{t},y)-(\boldsymbol{\Theta} \nabla T,\nabla y) =(\phi,y)+\left \langle \phi_1,y \right \rangle
	&&\quad\forall \psi \in H^1(\Omega).\no
\end{alignat}
Define \[
q_n:=\mathrm{div} \mathbf{u_n},\quad \varpi_n:=c_{0}p_n-b_0 T_n+\alpha q_n,\quad 
\tau_n:=\alpha p_n -(\lambda+\mu) q_n+\beta T_n,\quad 
\varsigma_n:=a_0 T_n -b_0 p_n +\beta q_n,
\]
So we have that
	\begin{alignat}{2}
	&\mu(\nabla\mathbf{u_n},\nabla\mathbf{v})-(\tau_n,\mathrm{div}\mathbf{v})= (\mathbf{f},\mathbf{v})+\langle \mathbf{f}_1,\mathbf{v}\rangle
	&&\quad\forall \mathbf{v}\in \mathbf{H}^1(\Omega),\\
	&\gamma_6(\tau_n,\varphi) +(\mathrm{div} \mathbf{u_n},\varphi) -\gamma_4(\varpi_n,\varphi)-\gamma_{1} (\varsigma_n,\varphi)=0 &&\quad\forall \varphi \in L^2(\Omega),  \\
	&\bigl(\varpi_n, y \bigr)_t+(k(\tau_n)\nabla (\gamma_4 \tau_n + \gamma_5 \varpi_n+\gamma_2 \varsigma_n ),\nabla y \bigr)= (g, y)+\langle g_1,y\rangle &&\quad\forall y \in H^1(\Omega),  \\
	&(\varsigma_n,z)_{t}+( (\boldsymbol{\Theta} \nabla (\gamma_{1} \tau_n +\gamma_{2} \varpi_n+\gamma_{3} \varsigma_n),\nabla z)=(\phi,z)+\left \langle \phi_1,z \right \rangle &&\quad\forall z \in H^1(\Omega).
\end{alignat}
Then there exists $\bu \subset H^{1}(\Omega),\tau \subset L^{2}(\Omega),\varpi\subset L^{2}(\Omega),\varsigma\subset L^{2}(\Omega) $ satisfy 
\begin{align*}
	\lim_{n\rightarrow \infty}\bu_n=\bu,\quad 
	\lim_{n\rightarrow \infty}\tau_n=\tau,\quad \lim_{n\rightarrow \infty}\varpi_n=\varpi,\quad \lim_{n\rightarrow \infty}\varsigma_n=\varsigma.
\end{align*}
From these above conclusions, we have that
\begin{alignat}{2}
	&\lim_{n\rightarrow \infty}\mu(\nabla\mathbf{u_n},\nabla\mathbf{v})-(\tau_n,\mathrm{div}\mathbf{v})\no\\
	=&\mu(\nabla\mathbf{u},\nabla\mathbf{v})-(\tau,\mathrm{div}\mathbf{v})\no,\\
	&\lim_{n\rightarrow \infty}\gamma_6(\tau_n,\varphi) +(\mathrm{div} \mathbf{u_n},\varphi) -\gamma_4(\varpi_n,\varphi)-\gamma_{1} (\varsigma_n,\varphi)\no\\=&
	\gamma_6(\tau,\varphi) +(\mathrm{div} \mathbf{u},\varphi) -\gamma_4(\varpi,\varphi)-\gamma_{1} (\varsigma,\varphi)\no,\\
	&\lim_{n\rightarrow \infty}(\varsigma_n,z)_{t}+( (\boldsymbol{\Theta} \nabla (\gamma_{1} \tau_n +\gamma_{2} \varpi_n+\gamma_{3} \varsigma_n),\nabla z)\no\\=&
	(\varsigma,z)_{t}+( (\boldsymbol{\Theta} \nabla (\gamma_{1} \tau +\gamma_{2} \varpi+\gamma_{3} \varsigma),\nabla z).\label{2-85}
\end{alignat}
Then
\begin{alignat}{2}
	&\lim_{n\rightarrow \infty} (k(\tau_n)\nabla (\gamma_4 \tau_n + \gamma_5 \varpi_n+\gamma_2 \varsigma_n ),\nabla y \bigr)-(k(\tau)\nabla (\gamma_4 \tau + \gamma_5 \varpi+\gamma_2 \varsigma ),\nabla y \bigr)\no\\=&\lim_{n\rightarrow \infty}
	((k(\tau_n)-k(\tau))\nabla (\gamma_4 \tau_n + \gamma_5 \varpi_n+\gamma_2 \varsigma_n ),\nabla y \bigr)\no\\&+(k(\tau)\nabla (\gamma_4 (\tau-\tau_n) + \gamma_5 (\varpi-\varpi_n)+\gamma_2 (\varsigma-\varsigma_n)
	 ),\nabla y \bigr)\no,\\
	 =&\lim_{n\rightarrow \infty}
	 (-e^{-\tau_n}(\tau_n-\tau)\nabla (\gamma_4 \tau_n + \gamma_5 \varpi_n+\gamma_2 \varsigma_n ),\nabla y \bigr)\no\\&+(k(\tau)\nabla (\gamma_4 (\tau-\tau_n) + \gamma_5 (\varpi-\varpi_n)+\gamma_2 (\varsigma-\varsigma_n)
	 ),\nabla y \bigr)\no,\\
	 \rightarrow&0,
\end{alignat}
so we have that
\begin{alignat}{2}
	&\lim_{n\rightarrow \infty}\bigl(\varpi_n, y \bigr)_t+(k(\tau_n)\nabla (\gamma_4 \tau_n + \gamma_5 \varpi_n+\gamma_2 \varsigma_n ),\nabla y \bigr)\no\\
	=&\bigl(\varpi, y \bigr)_t+(k(\tau)\nabla (\gamma_4 \tau + \gamma_5 \varpi+\gamma_2 \varsigma ),\nabla y \bigr)\no,\\
	&\lim_{n\rightarrow \infty}(\varsigma_n,z)_{t}+( \boldsymbol{\Theta} \nabla (\gamma_{1} \tau_n +\gamma_{2} \varpi_n+\gamma_{3} \varsigma_n),\nabla z)\no\\
	=&(\varsigma,z)_{t}+( \boldsymbol{\Theta} \nabla (\gamma_{1} \tau +\gamma_{2} \varpi+\gamma_{3} \varsigma),\nabla z)\label{2-86}.
\end{alignat}
Combing \eqref{2-85} and \eqref{2-86}, we get
\begin{alignat}{2}
	&\mu(\nabla\mathbf{u},\nabla\mathbf{v})-(\tau,\mathrm{div}\mathbf{v})= (\mathbf{f},\mathbf{v})+\langle \mathbf{f}_1,\mathbf{v}\rangle
	&&\quad\forall \mathbf{v}\in \mathbf{H}^1(\Omega),\\
	&\gamma_6(\tau,\varphi) +(\mathrm{div} \mathbf{u},\varphi) -\gamma_4(\varpi,\varphi)-\gamma_{1} (\varsigma,\varphi)=0 &&\quad\forall \varphi \in L^2(\Omega),  \\
	&\bigl(\varpi, y \bigr)_t+(k(\tau)\nabla (\gamma_4 \tau + \gamma_5 \varpi+\gamma_2 \varsigma ),\nabla y \bigr)= (g, y)+\langle g_1,y\rangle &&\quad\forall y \in H^1(\Omega),  \\
	&(\varsigma,z)_{t}+( (\boldsymbol{\Theta} \nabla (\gamma_{1} \tau +\gamma_{2} \varpi+\gamma_{3} \varsigma),\nabla z)=(\phi,z)+\left \langle \phi_1,z \right \rangle &&\quad\forall z \in H^1(\Omega).
\end{alignat}
This show  $ (\bu,\tau,\varpi,\varsigma) $ is a weak solution of \eqref{2-18}--\eqref{2-22}, the proof is complete.

\subsubsection{Uniqueness}
In this section, we give some conditions and guarantee the uniqueness of solutions of both original and reformulated equation.
\begin{theorem}Assume:\\ $\sqrt{\frac{c_0 k_m}{C_p^2}-C_1(\epsilon)-M} \sqrt{\frac{\theta_m a_0 }{C_T^2}-C_2(\epsilon)-M}/(M |b_0-1|)\ge 1$, $M=max\left\{k_M,\theta_M\right\}$,\\
	 then the solution $p,T$ of equation \eqref{2-43}--\eqref{2-44} is unique.
\end{theorem}

{\em Proof.}
	Use $k_1,k_2$ denote $k(B(p_1,T_1),p_1,T_1),k(B(p_2,T_2),p_2,T_2) $ and suppose $p_1,T_1,p_2,T_2$ respectively satisfy \eqref{2-43} and \eqref{2-44},i.e.
	\begin{alignat}{2}
		(c_0p_1 -b_0 T_1 +\alpha B(p_1,T_1)_t,w)+(k_1 \nabla p_1,\nabla w)=({g},w)\qquad\forall w \in H_{0}^{1}(\boldsymbol{\Omega})\label{2-66},\\
		((a_0T_1 -b_0 p_1 +\beta B(p_1,T_1))_t,v)+(\bt\nabla T_1,\nabla v)=({\phi},v)\qquad\forall v \in H_{0}^{1}(\boldsymbol{\Omega})\label{2-67}.
	\end{alignat}
	\begin{alignat}{2}
		(c_0p_2 -b_0 T_2 +\alpha B(p_2,T_2)_t,w)+(k_2 \nabla p_2,\nabla w)=({g},w)\qquad\forall w \in H_{0}^{1}(\boldsymbol{\Omega})\label{2-68},\\
		((a_0T_2 -b_0 p_2 +\beta B(p_2,T_2))_t,v)+(\bt\nabla T_2,\nabla v)=({\phi},v)\qquad\forall v \in H_{0}^{1}(\boldsymbol{\Omega})\label{2-69}.
	\end{alignat}
Subtracting  \eqref{2-68}--\eqref{2-69} from \eqref{2-66}--\eqref{2-67} and set $w=\varpi_1-\varpi_2=c_0(p_1-p_2) -b_0 (T_1-T_2) +\alpha B(p_1-p_2,T_1-T_2)$, $v=\varsigma_1-\varsigma_2=a_0(T_1-T_2) -b_0 (p_1-p_2) +\beta B(p_1-p_2,T_1-T_2)$. 
We recall that 
\begin{alignat}{2}
	&(c_0(p_1-p_2) -b_0 (T_1-T_2) +\alpha B(p_1-p_2,T_1-T_2)_t,w)\\=&\frac{1}{2}\frac{d}{dt} \left \| c_0(p_1-p_2) -b_0 (T_1-T_2) +\alpha B(p_1-p_2,T_1-T_2) \right \|^2 \no\\
	=&\frac{1}{2} \frac{d}{dt}\Norm{\varpi_1-\varpi_2}{}^2\no,\\
	&(a_0(T_1-T_2) -b_0 (p_1-p_2) +\beta B(p_1-p_2,T_1-T_2),v)\\=&\frac{1}{2}\frac{d}{dt} \left \| a_0(T_1-T_2) -b_0 (p_1-p_2) +\beta B(p_1-p_2,T_1-T_2) \right \|^2 \no\\
		=&\frac{1}{2} \frac{d}{dt}\Norm{\varsigma_1-\varsigma_2}{}^2\no.
\end{alignat}
We can get  
\begin{alignat}{2}
	\frac{1}{2}\frac{d}{dt} \left \| c_0(p_1-p_2) -b_0 (T_1-T_2) +\alpha B(p_1-p_2,T_1-T_2) \right \|^2\label{2-72}\\+((k_1 - k_2) \nabla p_1,\nabla w)+(k_2 \nabla(p_1-p_2),\nabla w)=0,\no\\
	\frac{1}{2}\frac{d}{dt} \left \| a_0(T_1-T_2) -b_0 (p_1-p_2) +\beta B(p_1-p_2,T_1-T_2) \right \|^2\label{2-73}\\+(\bt \nabla(T_1-T_2),\nabla v)=0,\no
\end{alignat}
which implies that 
\begin{alignat}{2}
	&\left | (k_1 - k_2) \nabla p_1,\nabla (c_0(p_1-p_2) -b_0 (T_1-T_2) +\alpha B(p_1-p_2,T_1-T_2) \right | \no\\
	\le &L_k\epsilon  \Norm{k_1-k_2}{}+C(\epsilon )\Norm{c_0(p_1-p_2) -b_0 (T_1-T_2) +\alpha B(p_1-p_2,T_1-T_2)}{1}^2\no\\
	\le &\epsilon_1 \Norm{p_1-p_2}{} +\epsilon_2 \Norm{T_1-T_2}{} +C_1(\epsilon )\Norm{p_1-p_2}{1}^2
	+C_2(\epsilon )\Norm{T_1-T_2}{1}^2,\label{2-74}
\end{alignat}
and
\begin{alignat}{2}
	(k_2 \nabla(p_1-p_2),\nabla (c_0(p_1-p_2) -b_0 (T_1-T_2) +\alpha B(p_1-p_2,T_1-T_2))\label{2-75}\\
	\ge \frac{c_0 k_m}{C_p^2} \Norm{p_1-p_2}{1}^2-k_M b_0\Norm{p_1-p_2}{1}\Norm{T_1-T_2}{1}+\alpha k_2(\nabla(p_1-p_2),\nabla B(p_1-p_2,T_1-T_2)\no,\\
	(\bt \nabla(T_1-T_2),\nabla (a_0(T_1-T_2) -b_0 (p_1-p_2) +\beta B(p_1-p_2,T_1-T_2)))\label{2-76}\\
	\ge \frac{\theta_m a_0 }{C_T^2}\Norm{T_1-T_2}{1}^2-\theta^m b_0\Norm{p_1-p_2}{1}\Norm{T_1-T_2}{1}+\beta \theta(\nabla(T_1-T_2),\nabla B(p_1-p_2,T_1-T_2). \no
\end{alignat}
Substituting \eqref{2-74}--\eqref{2-76} into \eqref{2-72}--\eqref{2-73}  and adding it with  $M={\rm max}\{k_M,\theta_M\}$, we can get 
\begin{alignat}{2}
	\frac{1}{2}\frac{d}{dt} \left \| w \right \|^2+\frac{1}{2}\frac{d}{dt} \left \| v \right \|^2-\epsilon_1 \Norm{p_1-p_2}{} -\epsilon_2 \Norm{T_1-T_2}{}\\ \no+
	(\frac{c_0 k_m}{C_p^2}-C_1(\epsilon)-M) \Norm{p_1-p_2}{1}^2+
	(\frac{\theta_m a_0 }{C_T^2}-C_2(\epsilon)-M)\Norm{T_1-T_2}{1}^2\no\\-2M (b_0-1)\Norm{p_1-p_2}{1}\Norm{T_1-T_2}{1}\no \le 0.
\end{alignat}

From the above conclusion, we have when $M |b_0-1| \le \sqrt{\frac{c_0 k_m}{C_p^2}-C_1(\epsilon)-M} \sqrt{\frac{\theta_m a_0 }{C_T^2}-C_2(\epsilon)-M}$:
\begin{alignat}{2}
	\frac{1}{2}\frac{d}{dt} \left \| w \right \|^2+\frac{1}{2}\frac{d}{dt} \left \| v \right \|^2\le \epsilon_1 \Norm{p_1-p_2}{} +\epsilon_2 \Norm{T_1-T_2}{},
\end{alignat}
and that is
\begin{alignat}{2}
		\frac{1}{2}\frac{d}{dt} \left \| w \right \|^2 \le C \left \| w \right \|^2,\\
		\frac{1}{2}\frac{d}{dt} \left \| v \right \|^2\le C \left \| v \right \|^2.
\end{alignat}
Now Gronwall's inequality implies that $w=v=0$, which in turns that 
\begin{align*}
	c_0(p_1-p_2) -b_0 (T_1-T_2) +\alpha B(p_1-p_2,T_1-T_2)=0,\\
	a_0(T_1-T_2) -b_0 (p_1-p_2) +\beta B(p_1-p_2,T_1-T_2)=0,\\
	\varpi_1-\varpi_2=0,\qquad \varsigma_1-\varsigma_2=0.
\end{align*}
Then it's easy to confirm
 $$\varpi_1=\varpi_2,\varsigma_1=\varsigma_2,$$$$p_1=p_2,T_1=T_2.$$
\begin{remark}
	After obtaining $p,T$,we can solve $\bu$ from the \eqref{2-15}, thus linear dependence of $\bu$ on $p,T$ guarantee the uniqueness of $\bu$ as long as $p,T$ is
	unique.
\end{remark}
\begin{remark}
	After obtaining $\varpi,\varsigma$,we can solve $\bu,\tau$ from the \eqref{2-18}--\eqref{2-19}, thus linear dependence of $\bu,\tau$ on $\varpi,\varsigma$ guarantee the uniqueness of $\bu,\tau$ as long as $\varpi,\varsigma$ is
	unique.
\end{remark}

\section{A new mixed finite element method}\label{sec-3}
In this section, the time interval $[0, t_f]$ is divided into $N$ equal intervals, denoted by $[t_{n-1}, t_{n}], ~ n=1,2,\ldots,N$,  and $\Delta t=\frac{t_f}{N}$, then $t_n=n\Delta t$. We use backward Euler method and denote $ d_{t} v^{n}:=\frac{v^{n}-v^{n-1}}{\Delta t}$. Let $\mathcal{T}_h$ be a quasi-uniform triangulation or rectangular partition of $\Omega$ with maximum mesh size $h$, and  $\bar{\Omega}=\bigcup_{\mathcal{K}\in\mathcal{T}_h}\bar{\mathcal{K}}$. The triangulation is constructed by connecting the midpoints
of each side to uniformly refine the rough grid $T_{H}$ to form a fine grid $T_{h}$, where $T_{H}$ represents the coarse triangular mesh, $T_{h}$ represents the fine triangular mesh. Fine mesh $T_{h}$ is obtained after mesh refinement on $T_{H}$, where $T_{H}$ is the coarse grid size and $T_{h}$ is the fine grid size. The above-mentioned triangulation finite element space pairs are $(\mathbf{X}_h, M_h)$ with  $\mathbf{X}_h\subset \mathbf{H}^1(\Omega)$ and $M_h\subset L^2(\Omega)$. 
In general, to solve the stokes problem, we usually need to choose stable finite element pairs that satisfy the inf-sup condition, that is,   $\mathbf{X}_h$ and $M_h$ 
satisfy 
\begin{alignat}{2}\label{3.1}
	\sup_{\mathbf{v}_h\in \mathbf{X}_h}\frac{({\mathrm{div}} \mathbf{v}_h,\varphi_h)}{\|\mathbf{v}_h\|_{H^1(\Omega)}}
	\geq \beta_0\|\varphi_h\|_{L^2(\Omega)} &&\quad \forall~\varphi_h\in M_{0h}:=M_h\cap L_0^2(\Omega),\ \beta_0>0,
\end{alignat}
so that we can ensure the well-posedness and accuracy of the discrete problem.
But if we consider our reformulated model \eqref{2-5}--\eqref{2-6}, which satisfy $\nabla \cdot \bu$ is not identically zero, then we can avoid the saddle point problem when solving for the displacement, that is, we do not need to consider the inf-sup condition anymore, can  choose arbitrary finite element spaces for $\bu$ and $\tau$. 
We will choose the following finite element pairs in our analysis and numerical tests,
$P_k-P_j$ element:
\begin{align*}
	\mathbf{X}_{h} &=\bigl\{\mathbf{v}_h\in \mathbf{C}^0(\overline{\Omega});\,
	\mathbf{v}_h|_\mathcal{K}\in \mathbf{P}_k(\mathcal{K})~~\forall \mathcal{K}\in T_h \bigr\},\\
	M_{h} &=\bigl\{\varphi_h\in C^0(\overline{\Omega});\,\varphi_h|_\mathcal{K}\in P_j(\mathcal{K})
	~~\forall \mathcal{K}\in T_h \bigr\},
\end{align*}
when $k=2,j=1$, it is called Taylor-Hood element  that satisfy \eqref{3.1} (cf.  \cite{ber,brezzi}).\\
Finite element approximation space $W_h,Z_h$ for $\varpi,\varsigma$ variable can be chosen independently, any piecewise polynomial space is acceptable provided that
$W_h \supset M_h$, the most convenient choice is $Z_h= W_h =M_h$. Moreover,we 
define
\begin{equation}\label{3.2}
	\mathbf{V}_h:=\bigl\{\mathbf{v}_h\in \mathbf{X}_h;\,  (\mathbf{v}_h,\mathbf{r})=0\,\,
	\forall \mathbf{r}\in \mathbf{RM} \bigr\}.
\end{equation}
It is easy to check that $\mathbf{X}_h=\mathbf{V}_h\bigoplus \mathbf{RM}$. It was proved in  \cite{Feng2010}
that there holds the following inf-sup condition:
\begin{align}\label{3.3}
	\sup_{\mathbf{v}_h\in \mathbf{V}_h}\frac{(\mathrm{div}\mathbf{v}_h,\varphi_h)}{\left\|\mathbf{v}_h\right\|_{H^1(\Omega)}} 
	\geq \beta_1\left\|\varphi_h\right\|_{L^2(\Omega)} \quad \forall \varphi_h\in M_{h},\quad \beta_1>0.
\end{align}
Also, we recall the following inverse inequality for polynomial functions  \cite{bs08,brezzi,cia}:
\begin{align}
	\left\|\nabla\varphi_{h} \right\|_{L^{2}(\mathcal{K})}\leq c_{1}h^{-1}\left\|\varphi_{h} \right\|_{L^{2}(\mathcal{K})}~~~~~\forall\varphi_{h}\in P_{r}(\mathcal{K}), ~\mathcal{K}\in T_{h}.\label{3.4} 
\end{align}
Next we propose the fully discrete new mixed finite element algorithm for the problem \eqref{2-5}--\eqref{2-9} as follows:

$\mathbf{Multiphysics~finite ~element~ algorithm ~(MAFEA):}$
\begin{itemize}
	\item [Step 1:] Compute $\mathbf{u}^{0}_{h}\in \mathbf{V}_{h}$ and $q^{0}_{h}\in W_h$ by 
\begin{align*}
	\mathbf{u}^{0}_{h} &=\mathcal{R}_{h}\mathbf{u}_{0},\quad 
	p^{0}_{h} =\mathcal{Q}_{h}p_{0},\quad 
	T^{0}_{h} =\mathcal{Q}_{h}T_{0},\quad \\
	q^{0}_{h}&=\mathrm{div} \mathbf{u}^{0}_{h},\quad 
	\varpi^{0}_{h}=c_{0}p^{0}_{h}-b_0 T^{0}_{h}+\alpha q^{0}_{h},\quad \\
	\tau^{0}_{h}&=\alpha p^{0}_{h} -(\lambda+\mu) q^{0}_{h}+\beta T^{0}_{h},\quad 
	\varsigma^{0}_{h} =a_0 T^{0}_{h} -b_0 p^{0}_{h} +\beta q^{0}_{h}.
\end{align*}
	\item [Step 2:] For $n=0,1,2,\cdots$,  do the following two steps.
	
{\em (i)} Solve for $(\mathbf{u}^{n+1}_{h},\tau^{n+1}_{h},\varpi^{n+1}_h,\varsigma^{n+1}_{h})
\in \mathbf{V}_{h}\times M_{h} \times  W_{h}\times  Z_{h}$ such that
\begin{eqnarray}
\mu\left(\nabla \left(\mathbf{u}_{h}^{n+1}\right),\nabla \left(\mathbf{v}_{h}\right)\right)-\left(\tau_{h}^{n+1},\nabla \cdot \mathbf{v}_{h}\right)=\left(\mathbf{f},\mathbf{v}_{h}\right)+\left\langle\mathbf{f}_{1},\mathbf{v}_{h}\right\rangle \quad \forall \mathbf{v}_{h} \in \boldsymbol{V}_{h},\label{3-5}\\
\gamma_{6}\left(\tau_{h}^{n+1},\varphi_{h}\right)+\left(\nabla \cdot \mathbf{u}_{h}^{n+1},\varphi_{h}\right)=\gamma_{4}\left(\varpi_{h}^{n+\theta},\varphi_{h}\right)+\gamma_{1}\left(\varsigma_{h}^{n+\theta},\varphi_{h}\right) \quad \forall \varphi_{h} \in M_{h},  \label{3-6}\\
\left(d_{t} \varpi_{h}^{n+1}, y_{h}\right)+\left(k(\tau_h^{n+1})\nabla\left(\gamma_{4} \tau_{h}^{n+1}+\gamma_{5} \varpi_{h}^{n+1}+\gamma_{2} \varsigma_{h}^{n+1}\right),\nabla y_{h}\right)\qquad\qquad\qquad \no\\
=\left(g, y_{h}\right)+\left\langle g_{1}, y_{h}\right\rangle \quad \forall y_{h} \in W_{h},  \label{3-7}\\
\left(d_{t} \varsigma_{h}^{n+1}, z_{h}\right)+\left(\bt \nabla\left(\gamma_{1} \tau_{h}^{n+1}+\gamma_{2} \varpi_{h}^{n+1}+\gamma_{3} \varsigma_{h}^{n+1}\right),\nabla z_{h}\right) \qquad\qquad\qquad\no\\
=\left(\phi, z_{h}\right)+\left\langle\phi_{1}, z_{h}\right\rangle \quad \forall z_{h} \in Z_{h},\label{3-8}
\end{eqnarray}
	where $d_{t}(\varpi_{h}^{n+1})=\dfrac{\varpi_{h}^{n+1}-\varpi_{h}^{n}}{\Delta t} $ and $\theta=0,1$.
	
{\em (ii)} Update $p^{n+1}_h,T^{n+1}_h$ and $q^{n+1}_h$ by
	\begin{alignat}{2}
	T^{n+1}_h&=\gamma_{1} \tau_h^{n+1} +\gamma_{2} \varpi_h^{n+\theta}+\gamma_{3} \varsigma_h^{n+\theta},\qquad
	p^{n+1}_h=\gamma_4 \tau_h^{n+1} + \gamma_5 \varpi_h^{n+\theta}+\gamma_2 \varsigma_h^{n+\theta},\qquad\no \\
	q^{n+1}_h&=-\gamma_6 \tau_h^{n+1} + \gamma_4 \varpi_h^{n+\theta} +\gamma_{1} \varsigma_h^{n+\theta}.
 	\end{alignat}
\end{itemize}
\begin{remark}
	
	(a) At each time step, the problem \eqref{3-5}--\eqref{3-6} is a generalized Stokes problem
with a mixed boundary condition for $(\bu_h^{n+1},\tau_h^{n+1})$. Because $\lambda$ is a limited number, $\nabla \cdot \bu$ in \eqref{3-6} is
not equal to zero, which show \eqref{3-5}--\eqref{3-6} is not a saddle point
problem,  so we can use arbitrary finite element pairs to handle it such as $P_2-P_1,P_2-P_2,P_1-P_1$,etc.

(b) When $\theta=0$, step 1  consists of two decoupled sub-problems that can be solved independently, therefore we can substitute the solution $\tau_h^{n+1}$ of \eqref{3-5}--\eqref{3-6} into $k(\tau_h^{n+1})$ of \eqref{3-7},  then the second sub-problem become easy to solve.
\end{remark}

\subsection{Stability analysis}
Before discussing the stability of MFEA, we first show that the numerical solution satisfies all side constraints which are full-filled by the PDE solution.
\begin{lemma}\label{lem3.1}
	Let $\bigl\{({\textbf{u}_{h}^{n}},\tau_{h}^{n},\varpi_{h}^{n},\varsigma_{h}^{n})\bigr\}_{n>=0}$ be defined by the MFEA, then there hold
\begin{alignat}{2}
	\left(\varpi_{h}^{n}, 1\right) & =C_{\varpi}\left(t_{n}\right) & & \text { for } n=0,1,2,\cdots,\\
		\left(\varsigma_{h}^{n}, 1\right) & =C_{\varsigma}\left(t_{n}\right) & & \text { for } n=0,1,2,\cdots,\\
\left(\tau_{h}^{n}, 1\right) & =C_{\tau}\left(t_{n-1+\theta}\right) & & \text { for } n=1-\theta, 1,2,\cdots,\\
\left\langle\mathbf{u}_{h}^{n} \cdot \mathbf{n}, 1\right\rangle & =C_{\mathbf{u}}\left(t_{n-1+\theta}\right) & & \text { for } n=1-\theta, 1,2,\cdots.
\end{alignat}
	\begin{proof}
		Taking $y_{h}=1,z_{h}=1$ in \eqref{3-7}--\eqref{3-8}, $\mathbf{v}_h=\mathbf{x}$ in \eqref{3-5} and $\phi_h=1$, we have
		\begin{alignat}{2}
			\mu (\nabla {u}_h^{n},I)-d(\tau_h^{n},1)  & =(\mathbf{f},\mathbf{x})+\left\langle\mathbf{f}_{1},\mathbf{x} \right\rangle\label{3.13},\\
\gamma_{6}\left(\tau_{h}^{n}, 1\right)+\left(\nabla \cdot \mathbf{u}_{h}^{n}, 1\right)&=\gamma_{4}\left(\varpi_{h}^{n-1+\theta}, 1\right)+\gamma_{1}\left(\varsigma_{h}^{n-1+\theta}, 1\right)\label{3.14},\\
(d_{t}\varpi_{h}^{n},1)&=(\phi,1)+\left\langle\phi_{1},1 \right\rangle\label{3-16},\\
(d_{t}\varsigma_{h}^{n},1)&=(\phi,1)+\left\langle\phi_{1},1 \right\rangle.\label{3-17}
		\end{alignat}
{\rm 	Multipling $\sum_{0}^{l} $ on the two side of \eqref{3-16}--\eqref{3-17}, we have}
	\begin{alignat}{2}
		(\varpi_{h}^n,1)&=(\varpi_{h}^{0},1)+t_n((\phi,1)+(\phi_{1},1))&&=C_{\varpi}(t_n)\label{3-18},\\
		(\varsigma_{h}^n,1)&=(\varsigma_{h}^{0},1)+t_n((\phi,1)+(\phi_{1},1))&&=C_{
			\varsigma}(t_n)\label{3-19}.
	\end{alignat}
{\rm 	Using \eqref{3.13}, \eqref{3.14}, we obtain}
	\begin{alignat}{2}
		(\mu\gamma_{6}+d)(\tau_h^n,1)
		&=\mu
		(\gamma_{4}(\varpi_{h}^{n-1+\theta},1)+
		\gamma_{1}(\varsigma_{h}^{n-1+\theta},1)
		)-(\mathbf{f},\mathbf{x})-\left\langle\mathbf{f}_{1},\mathbf{x} \right\rangle\label{3-20},
	\end{alignat}
{\rm so we have that }
\begin{align*}
	(\tau_h^{n},1)&=C_{\mathbf{\tau}}(t_{n-1+\theta}).
\end{align*}
{\rm Then using Gaussian divergence theorem and \eqref{3.14}, we obtain}
	\begin{align*}
		\left\langle\mathbf{u}_{h}^{n} \cdot \mathbf{n}, 1\right\rangle&=({\rm div} \mathbf{u}_h^{n},1)\\
		&= \gamma_{4}\left(\varpi_{h}^{n-1+\theta}, 1\right)+\gamma_{1}\left(\varsigma_{h}^{n-1+\theta}, 1\right)-
		\gamma_{6}\left(\tau_{h}^{n}, 1\right) \\
		&=C_{\mathbf{u}}(t_{n-1+\theta}).
	\end{align*}
{\rm The proof is complete.}
	\end{proof}
	\end{lemma}
	\begin{lemma}
		 Let $ \left\lbrace \left( \bu_h^{n},\tau_h^{n},\varpi_h^{n},\varsigma
			_h^{n}\right) \right\rbrace_{n\geq 0}  $ be defined by MFEA, then there holds the following identity:
			\begin{alignat}{2}
			J_{h,\theta}^{l}+S_{h,\theta}^{l}=J_{h,\theta}^{0}\qquad {\rm for}\quad l\geq1,\theta=0,1,\label{3.18}
		\end{alignat}
	where
	\begin{align*}
J_{h,\theta}^{l}: & =\frac{1}{2}\left[\mu\left\|\nabla\left(\mathbf{u}_{h}^{l+1}\right)\right\|_{L^{2}(\Omega)}^{2}+\gamma_{6}\left\|\tau_{h}^{l+1}\right\|_{L^{2}(\Omega)}^{2}+\left(\gamma_{5}+\gamma_{2}\right)\left\|\varpi_{h}^{l+\theta}\right\|_{L^{2}(\Omega)}^{2}\right. \\
& +\left(\gamma_{3}+\gamma_{2}\right)\left\|\varsigma_{h}^{l+\theta}\right\|_{L^{2}(\Omega)}^{2}-\frac{\gamma_{2}}{2}\left\|\varsigma_{h}^{l+1}-\varpi_{h}^{l+1}\right\|_{L^{2}(\Omega)}^{2}-2\left(\mathbf{f},\mathbf{u}_{h}^{l+1}\right) \\
& \left.-2\left\langle\mathbf{f}_{1},\mathbf{u}_{h}^{l+1}\right\rangle\right],\\
S_{h,\theta}^{l}: & =\Delta t \sum_{n=0}^{l}\left[\frac{\mu \Delta t}{2}\left\|d_{t} \nabla\left(\mathbf{u}_{h}^{n+1}\right)\right\|_{L^{2}(\Omega)}^{2}+\frac{\gamma_{6} \Delta t}{2}\left\|d_{t} \tau_{h}^{n+1}\right\|_{L^{2}(\Omega)}^{2}\right. \\
& +\frac{\left(\gamma_{5}+\gamma_{2}\right) \Delta t}{2}\left\|d_{t} \varpi_{h}^{n+\theta}\right\|_{L^{2}(\Omega)}^{2}+\frac{\left(\gamma_{3}+\gamma_{2}\right) \Delta t}{2}\left\|d_{t} \varsigma_{h}^{n+\theta}\right\|_{L^{2}(\Omega)}^{2} \\
& -\frac{\gamma_{2}}{2} \Delta t\left\|d_{t} \varsigma_{h}^{n+\theta}-d_{t} \varpi_{h}^{n+\theta}\right\|_{L^{2}(\Omega)}^{2}+\left(k(\tau_{h}^{n+1}) \nabla p_{h}^{n+1},\nabla p_{h}^{n+1}\right) \\
& +\left(\boldsymbol{\Theta} T_{h}^{n+1},\nabla T_{h}^{n+1}\right)-(1-\theta) \gamma_{4} \Delta t\left(k(\tau_{h}^{n+1}) \left(d_{t} \nabla \tau_{h}^{n+1}\right),\nabla p_{h}^{n+1}\right) \\
& -(1-\theta) \gamma_{1} \Delta t\left(\boldsymbol{\Theta}\left(d_{t} \nabla \tau_{h}^{n+1}\right),\nabla T_{h}^{n+1}\right)-\left(g, p_{h}^{n+1}\right) \\
& \left.-\left\langle g_{1}, p_{h}^{n+1}\right\rangle-\left(\phi, T_{h}^{n+1}\right)-\left\langle\phi_{1}, T_{h}^{n+1}\right\rangle\right]. 
	\end{align*}
\end{lemma}
\begin{proof}
1. We firstly consider  the case $\theta=1$. Setting $\bv_h=d_t\bu_{h}^{n+1},\varphi_{h}=\tau_h^{n+1}$ in \eqref{3-5}--\eqref{3-6} and $y_h=p_h^{n+1},z_h=T_h^{n+1}$ in \eqref{3-7}--\eqref{3-8}, we obtain
\begin{alignat}{2}
\mu\left(\nabla\left(\mathbf{u}_{h}^{n+1}\right),\nabla\left(d_{t} \mathbf{u}_{h}^{n+1}\right)\right)-\left(\tau_{h}^{n+1},\nabla \cdot d_{t} \mathbf{u}_{h}^{n+1}\right)=\left(\mathbf{f}, d_{t} \mathbf{u}_{h}^{n+1}\right)+\left\langle\mathbf{f}_{1}, d_{t} \mathbf{u}_{h}^{n+1}\right\rangle \label{3.29},\\
\gamma_{6}\left(d_{t} \tau_{h}^{n+1},\tau_{h}^{n+1}\right)+\left(\nabla \cdot d_{t} \mathbf{u}_{h}^{n+1},\tau_{h}^{n+1}\right)=\gamma_{4}\left(d_{t} \varpi_{h}^{n+1},\tau_{h}^{n+1}\right)+\gamma_{1}\left(d_{t} \varsigma_{h}^{n+1},\tau_{h}^{n+1}\right) \label{3-30},\\
\gamma_{5}\left(d_{t} \varpi_{h}^{n+1},\varpi_{h}^{n+1}\right)+\gamma_{4}\left(d_{t} \varpi_{h}^{n+1},\tau_{h}^{n+1}\right)+\gamma_{2}\left(d_{t} \varpi_{h}^{n+1},\varsigma_{h}^{n+1}\right) \no\\
+\left(k(\tau_{h}^{n+1}) \nabla p_{h}^{n+1},\nabla p_{h}^{n+1}\right)=\left(g, p_{h}^{n+1}\right)+\left\langle g_{1}, p_{h}^{n+1}\right\rangle \label{3-31},\\
\gamma_{3}\left(d_{t} \varsigma_{h}^{n+1},\varsigma_{h}^{n+1}\right)+\gamma_{1}\left(d_{t} \varsigma_{h}^{n+1},\tau_{h}^{n+1}\right)+\gamma_{2}\left(d_{t} \varsigma_{h}^{n+1},\varpi_{h}^{n+1}\right) \no \\
+\left(\boldsymbol{\Theta} \nabla T_{h}^{n+1},\nabla T_{h}^{n+1}\right)=\left(\phi, T_{h}^{n+1}\right)+\left\langle\phi_{1}, T_{h}^{n+1}\right\rangle\label{3-32}.
\end{alignat}
{\rm The first term on the left-hand side of \eqref{3.29} can be written as }
\begin{alignat}{2}
\mu\left(\nabla\left(\mathbf{u}_{h}^{n+1}\right),\nabla\left(d_{t} \mathbf{u}_{h}^{n+1}\right)\right)&=\frac{\mu}{\Delta t}\left(\nabla\left(\mathbf{u}_{h}^{n+1}\right),\nabla\left(\mathbf{u}_{h}^{n+1}\right)-\nabla\left(\mathbf{u}_{h}^{n}\right)\right) \no\\
&=\frac{\mu}{2 \Delta t}\left[2\left(\nabla\left(\mathbf{u}_{h}^{n+1}\right),\nabla\left(\mathbf{u}_{h}^{n+1}\right)\right)-2\left(\nabla\left(\mathbf{u}_{h}^{n+1}\right),\nabla\left(\mathbf{u}_{h}^{n}\right)\right)\right. \no\\
&\left.+\left(\nabla\left(\mathbf{u}_{h}^{n}\right),\nabla\left(\mathbf{u}_{h}^{n}\right)\right)-\left(\nabla\left(\mathbf{u}_{h}^{n}\right),\nabla\left(\mathbf{u}_{h}^{n}\right)\right)\right] \no\\
&=\frac{\mu}{2 \Delta t}\left[\left(\nabla\left(\mathbf{u}_{h}^{n+1}\right)-\nabla\left(\mathbf{u}_{h}^{n}\right),\nabla\left(\mathbf{u}_{h}^{n+1}\right)-\nabla\left(\mathbf{u}_{h}^{n}\right)\right)\right. \no\\
&\left.+\left(\nabla\left(\mathbf{u}_{h}^{n+1}\right),\nabla\left(\mathbf{u}_{h}^{n+1}\right)\right)-\left(\nabla\left(\mathbf{u}_{h}^{n}\right),\nabla\left(\mathbf{u}_{h}^{n}\right)\right)\right] \no\\
&=\frac{\mu \Delta t}{2}\left\|d_{t} \nabla\left(\mathbf{u}_{h}^{n+1}\right)\right\|_{L^{2}(\Omega)}^{2}+\frac{\mu}{2} d_{t}\left\|\nabla\left(\mathbf{u}_{h}^{n+1}\right)\right\|_{L^{2}(\Omega)}^{2}.\label{3-33}
\end{alignat}
{\rm The first term on the left-hand side of \eqref{3-30}--\eqref{3-32} can be written as }
\begin{alignat}{2}
\left( d_{t}\varpi_{h}^{n+1},\varpi_h^{n+1}\right)
&= \frac{1}{2}d_t||\varpi_{h}^{n+1}||_{L^{2}(\Omega)}^{2}+\frac{\Delta t}{2}||d_t\varpi_{h}^{n+1}||_{L^{2}(\Omega)}^{2},\label{3-27}\\
\left( d_{t}\tau_{h}^{n+1},\tau_h^{n+1}\right)
&= \frac{1}{2}d_t||\tau_{h}^{n+1}||_{L^{2}(\Omega)}^{2}+\frac{\Delta t}{2}||d_t\tau_{h}^{n+1}||_{L^{2}(\Omega)}^{2},\label{3-28}\\
\left( d_{t}\varsigma_{h}^{n+1},\varsigma_h^{n+1}\right)
&= \frac{1}{2}d_t||\varsigma_{h}^{n+1}||_{L^{2}(\Omega)}^{2}+\frac{\Delta t}{2}||d_t\varsigma_{h}^{n+1}||_{L^{2}(\Omega)}^{2}.\label{3-29}
\end{alignat}
Using \eqref{3-30}, we obtain
\begin{alignat}{2}
&\gamma_{2}\left(d_{t} \varpi_{h}^{n+1},\varsigma_{h}^{n+1}\right) \no\\
&=\frac{\gamma_{2}}{2 \Delta t}\left[\left(\varpi_{h}^{n+1}-\varpi_{h}^{n},\varsigma_{h}^{n+1}\right)-\left(\varpi_{h}^{n},\varsigma_{h}^{n+1}-\varsigma_{h}^{n}\right)+\left(\varpi_{h}^{n+1},\varsigma_{h}^{n+1}\right)-\left(\varpi_{h}^{n},\varsigma_{h}^{n}\right)\right],\\
&\gamma_{2}\left(d_{t} \varsigma_{h}^{n+1},\varpi_{h}^{n+1}\right) \no \\
&=\frac{\gamma_{2}}{2 \Delta t}\left[\left(\varsigma_{h}^{n+1}-\varsigma_{h}^{n},\varpi_{h}^{n+1}\right)-\left(\varsigma_{h}^{n},\varpi_{h}^{n+1}-\varpi_{h}^{n}\right)+\left(\varsigma_{h}^{n+1},\varpi_{h}^{n+1}\right)-\left(\varsigma_{h}^{n},\varpi_{h}^{n}\right)\right].\label{3-38}
\end{alignat}
Then through adding \eqref{3.29}--\eqref{3-32} and substituting \eqref{3-33}--\eqref{3-38},   we have that 
\begin{alignat}{2}
\frac{\mu}{2} d_{t}\left\|\nabla\left(\mathbf{u}_{h}^{n+1}\right)\right\|_{L^{2}(\Omega)}^{2}+\frac{\mu \Delta t}{2}\left\|d_{t} 
\nabla\left(\mathbf{u}_{h}^{n+1}\right)\right\|_{L^{2}(\Omega)}^{2}+\frac{\gamma_{6}}{2} d_{t}\left\|\xi_{h}^{n+1}\right\|_{L^{2}(\Omega)}^{2} \no\\
+\frac{\gamma_{6}}{2} \Delta t\left\|d_{t} \xi_{h}^{n+1}\right\|_{L^{2}(\Omega)}^{2}+\frac{\gamma_{5}+\gamma_{2}}{2} d_{t}\left\| \varpi_{h}^{n+1}\right\|_{L^{2}(\Omega)}^{2}+\frac{\gamma_{5}+\gamma_{2}}{2} \Delta t\left\|d_{t} \varpi_{h}^{n+1}\right\|_{L^{2}(\Omega)}^{2} \no\\
+\frac{\gamma_{3}+\gamma_{2}}{2} d_{t}\left\|\varsigma_{h}^{n+1}\right\|_{L^{2}(\Omega)}^{2}+\frac{\gamma_{3}+\gamma_{2}}{2} \Delta t\left\|d_{t} \varsigma_{h}^{n+1}\right\|_{L^{2}(\Omega)}^{2}-\frac{\gamma_{2}}{2} \Delta t\left\|d_{t} \varsigma_{h}^{n+1}-d_{t} \varpi_{h}^{n+1}\right\|_{L^{2}(\Omega)}^{2} \no\\
-\frac{\gamma_{2}}{2} d_{t}\left\|\varsigma_{h}^{n+1}- \varpi_{h}^{n+1}\right\|_{L^{2}(\Omega)}^{2}+\left(k(\tau_{h}^{n+1}) \nabla p_{h}^{n+1},\nabla p_{h}^{n+1}\right)+\left(\boldsymbol{\Theta} \nabla T_{h}^{n+1},\nabla T_{h}^{n+1}\right) \no\\
=\left(\mathbf{f}, d_{t} \mathbf{u}_{h}^{n+1}\right)+\left\langle\mathbf{f}_{1}, d_{t} \mathbf{u}_{h}^{n+1}\right\rangle+\left(g, p_{h}^{n+1}\right)+\left\langle g_{1}, p_{h}^{n+1}\right\rangle+\left(\phi, T_{h}^{n+1}\right)+\left\langle\phi_{1}, T_{h}^{n+1}\right\rangle.\label{3-39}
\end{alignat}
{\rm  Applying the summation operator $ \Delta t\sum_{n=1} ^{l}$ to the both sides of \eqref{3-39}, we  yield the desired equality \reff{3.18}.}

{\rm 2. As for the case of $\theta=0$, from \eqref{3.14}, we can define $\varpi_h^{-1},\varsigma_h^{-1}$ by}
	\begin{alignat}{2}
(p_h^0,y_h)=\gamma_{4}(\tau_h^0,y_h)+\gamma_{5}(\varpi_h^{-1},y_h)+\gamma_{2}(\varsigma_h^{-1},y_h),\\
(T_h^0,z_h)=\gamma_{1}(\tau_h^0,z_h)+\gamma_{2}(\varpi_h^{-1},z_h)+\gamma_{3}(\varsigma_h^{-1},z_h).
	\end{alignat}
{\rm Through setting $\bv_h=d_t\bu_{h}^{n+1},\varphi_{h}=\tau_h^{n+1}$ in \eqref{3-5}--\eqref{3-6} and $y_h=p_h^{n+1},z_h=T_h^{n+1}$ in \eqref{3-7}--\eqref{3-8}  after lowing the super-index from $n+1$ to $n$ on the both sides of \eqref{3-7}--\eqref{3-8}, we have that}
\begin{alignat}{2}
	\mu\left(\nabla\left(\mathbf{u}_{h}^{n+1}\right),\nabla\left(d_{t} \mathbf{u}_{h}^{n+1}\right)\right)-\left(\tau_{h}^{n+1},\nabla \cdot d_{t} \mathbf{u}_{h}^{n+1}\right)=\left(\mathbf{f}, d_{t} \mathbf{u}_{h}^{n+1}\right)+\left\langle\mathbf{f}_{1}, d_{t} \mathbf{u}_{h}^{n+1}\right\rangle \label{1-29},\\
	\gamma_{6}\left(d_{t} \tau_{h}^{n+1},\tau_{h}^{n+1}\right)+\left(\nabla \cdot d_{t} \mathbf{u}_{h}^{n+1},\tau_{h}^{n+1}\right)=\gamma_{4}\left(d_{t} \varpi_{h}^{n+1},\tau_{h}^{n+1}\right)+\gamma_{1}\left(d_{t} \varsigma_{h}^{n+1},\tau_{h}^{n+1}\right) \label{1-30},\\
	\gamma_{5}\left(d_{t} \varpi_{h}^{n},\varpi_{h}^{n}\right)+\gamma_{4}\left(d_{t} \varpi_{h}^{n},\tau_{h}^{n+1}\right)+\gamma_{2}\left(d_{t} \varpi_{h}^{n},\varsigma_{h}^{n}\right) \no\\
	+\left(k(\tau_{h}^{n+1}) \nabla p_{h}^{n},\nabla p_{h}^{n+1}\right)=\left(g, p_{h}^{n+1}\right)+\left\langle g_{1}, p_{h}^{n+1}\right\rangle \label{1-31},\\
	\gamma_{3}\left(d_{t} \varsigma_{h}^{n},\varsigma_{h}^{n}\right)+\gamma_{1}\left(d_{t} \varsigma_{h}^{n},\tau_{h}^{n+1}\right)+\gamma_{2}\left(d_{t} \varsigma_{h}^{n},\varpi_{h}^{n}\right) \no \\
	+\left(\boldsymbol{\Theta} \nabla T_{h}^{n},\nabla T_{h}^{n+1}\right)=\left(\phi, T_{h}^{n+1}\right)+\left\langle\phi_{1}, T_{h}^{n+1}\right\rangle\label{1-32},
\end{alignat}
Moreover, it is easy to check that
\begin{alignat}{2}
	&\frac{1}{\mu_{f}}\left(k(\tau_h^{n}) \nabla\left(\gamma_{1} \tau_{h}^{n}+\gamma_{2}\varpi_{h}^{n} \right),\nabla p_{h}^{n+1} \right)\label{1-27}\\
	&=\frac{1}{\mu_{f}}\left( k(\tau_h^{n})\nabla p_{h}^{n+1},\nabla p_{h}^{n+1}  \right)
	-\frac{\gamma_1\Delta t}{\mu_{f}}\left( k(\tau_h^{n}) d_t\nabla\tau_{h}^{n+1},\nabla p_{h}^{n+1} \right),\no\\
	&\left(\boldsymbol{\Theta} \nabla T_{h}^{n},\nabla T_{h}^{n+1}\right)
	=\left(\boldsymbol{\Theta} \nabla T_{h}^{n+1},\nabla T_{h}^{n+1}\right)-
	\gamma_{1} \Delta t \left(\boldsymbol{\Theta} d_t \nabla \tau_{h}^{n+1},\nabla T_{h}^{n+1}\right),\\
	&\gamma_3\left( d_t\tau_h^{n+1},\tau_h^{n+1}\right)= \frac{\gamma_3}{2}d_t||\tau_{h}^{n+1}||_{L^{2}(\Omega)}^{2}+\frac{\gamma_3\Delta t}{2}||d_t\tau_{h}^{n+1}||_{L^{2}(\Omega)}^{2}.\label{1-28}
\end{alignat}
 Adding \eqref{1-29}¨C\eqref{1-32}, using \eqref{1-27}¨C\eqref{1-28} and applying the summation operator $ \Delta t\sum_{n=1} ^{l}$ to the both sides of the resulting equation, we can get the desired equality \reff{3.18}.
 \end{proof}
\begin{corollary}
	 Let $ \left\lbrace \left( \bu_h^{n},\tau_h^{n},\varpi_h^{n}\right) \right\rbrace_{n\geq 0}  $ be defined by MFEA with $ \theta=0 $, then there holds the following inequality:\\
	\begin{alignat}{2}
		J_{h,0}^{l}+\hat{S}_{h,0}^{l}\leq J_{h,0}^{0}\qquad {\rm for}\quad l\geq1,\label{3.31}
	\end{alignat}
	{\rm provided that $ \Delta t=O(h^{2}) $, where}
	\begin{align*}
	\hat{S}_{h,0}^{l}=&\Delta t\sum_{n=1}^{l}\left[\left( \frac{\mu}{4}\Delta t\right) \|d_{t}\nabla\left(\bu_{h}^{n+1} \right)  \| _{L^{2}(\Omega)}^{2} \right.\\                         
&\left.+\frac{k_m}{2\mu_{f}}
\|\nabla p_{h}^{n+1}\|_{L^{2}(\Omega)}^{2}+
\frac{\theta_m}{2\mu_f}\|\nabla T_h^{n+1}\|_{L^{2}(\Omega)}^{2}
+\frac{(\gamma_{2}+\gamma_{5})\Delta t}{2}\|d_{t}\varpi_{h}^{n}  \| _{L^{2}(\Omega)}^{2} \right.\\ 
&+\frac{(\gamma_{2}+\gamma_{3})\Delta t}{2}\|d_{t}\varsigma_{h}^{n}  \| _{L^{2}(\Omega)}^{2}-\frac{k_2}{2} \Norm{d_t \varsigma_h^n-d_t \varpi_h^n}{L^{2}(\Omega)}\\
&\left.+\frac{\gamma_{6}\Delta t}{2}\|d_{t}\tau_{h}^{n+1}  \| _{L^{2}(\Omega)}^{2}-\left( \phi,p_{h}^{n+1}\right)
-\left\langle \phi_1,p_{h}^{n+1} \right\rangle -\left( g,T_{h}^{n+1}\right)
-\left\langle g_1,T_{h}^{n+1} \right\rangle
\right].
	\end{align*}
\end{corollary}
\begin{proof}
 When $ \theta=0 $,  the last term in the expression of $ S_{h,\theta}^{l} $ does not have a fixed sign,  hence it needs to be controlled to ensure the positive of $ S_{h,0}^{l} $.  Using the Cauchy-Schwarz inequality, we obtain
	\begin{align}
		&\frac{\gamma_4\Delta t}{\mu_f}\left(k(\tau_h^{n})d_t\nabla\tau_h^{n+1},\nabla p_h^{n+1} \right)\no\\
		&=\frac{\gamma_4k(\tau_h^{n})}{\mu_f}\left(\nabla\tau_h^{n+1}-\nabla\tau_h^{n},\nabla p_h^{n+1}\right) \no\\
		&\leq\frac{\gamma_4k(\tau_h^{n})}{\mu_f}\|\nabla\tau_h^{n+1}-\nabla\tau_h^{n}\|_{L^{2}(\Omega)}\|\nabla p_h^{n+1}\|_{L^{2}(\Omega)}\no\\
		&\leq\frac{c_1^{2}\gamma_4^{2}k(\tau_h^{n})}{2\mu_fh^{2}}\|\tau_h^{n+1}-\tau_h^{n}\|_{L^{2}(\Omega)}^{2}+\frac{k(\tau_h^{n})}{2\mu_f}\|\nabla p_h^{n+1}\|_{L^{2}(\Omega)}^{2},\label{3-43}
	\end{align}
where
\begin{align}
&\|\tau_h^{n+1}-\tau_h^{n}\|_{L^{2}(\Omega)}\no\\
&\leq\frac{1}{\beta_1}\sup_{\bv_h\in\bV_h}\frac{\left( \div\bv_h,\tau_h^{n+1}-\tau_h^{n}\right) }{\|\bv_h\|_{H^1(\Omega)}}\no\\
&=\frac{1}{\beta_1}\sup_{\bv_h\in\bV_h}\frac{\mu\left( \nabla(\bu_h^{n+1}),\nabla(\bv_h)\right) -\mu\left( \nabla(\bu_h^{n}),\nabla(\bv_h)\right) }{\|\nabla\bv_h\|_{L^2(\Omega)}}\no\\
&\leq\frac{\mu\Delta t}{\beta_1}\|d_t\nabla(\bu_h^{n+1})\|_{L^2(\Omega)}.\label{3-44}
\end{align}
Substituting \eqref{3-44} into \eqref{3-43}, we obtain
\begin{alignat}{2}
	&\frac{\gamma_4\Delta t}{\mu_f}\left(k(\tau_h^{n})d_t\nabla\tau_h^{n+1},\nabla p_h^{n+1} \right)\label{3.32}\\
	&\leq\frac{c_1^{2}\gamma_4^{2}\mu^{2}k(\tau_h^{n})}{2\mu_fh^{2}\beta_1^{2}}(\Delta t)^{2}\|d_t\nabla(\bu_h^{n+1})\|_{L^{2}(\Omega)}^{2}+\frac{k(\tau_h^{n})}{2\mu_f}\|\nabla p_h^{n+1}\|_{L^{2}(\Omega)}^{2}\no.
\end{alignat}
Using the same method, we can get
\begin{alignat}{2}
	&\frac{\gamma_1\Delta t}{\mu_f}\left(\boldsymbol{\Theta} d_t\nabla\tau_h^{n+1},\nabla T_h^{n+1} \right)\\
&\leq\frac{c_1^{2}\gamma_1^{2}\mu^{2}\bt }{2\mu_fh^{2}\beta_1^{2}}(\Delta t)^{2}\|d_t\nabla(\bu_h^{n+1})\|_{L^{2}(\Omega)}^{2}+\frac{\bt}{2\mu_f}\|\nabla T_h^{n+1}\|_{L^{2}(\Omega)}^{2}\no.
\end{alignat}
Substituting \eqref{3.32} into $ S_{h,0}^{l} $, we obtain
\begin{align*}
	S_{h,0}^{l}\geq&\Delta t\sum_{n=1}^{l}\left[\left( \frac{\mu}{2}\Delta t-\frac{c_1^{2}\gamma_1^{2}\mu^{2}k(\tau_h^{n})}{2\mu_fh^{2}\beta_1^{2}}(\Delta t)^{2}-\frac{c_1^{2}\gamma_1^{2}\mu^{2}\bt }{2\mu_fh^{2}\beta_1^{2}}(\Delta t)^{2}\right) \|d_{t}\nabla\left(\bu_{h}^{n+1} \right)  \| _{L^{2}(\Omega)}^{2} \right.\\                         
	&\left.+\frac{k(\tau_h^{n})}{2\mu_{f}}
	\|\nabla p_{h}^{n+1}\|_{L^{2}(\Omega)}^{2}+
	\frac{\bt}{2\mu_f}\|\nabla T_h^{n+1}\|_{L^{2}(\Omega)}^{2}
	+\frac{(\gamma_{2}+\gamma_{5})\Delta t}{2}\|d_{t}\varpi_{h}^{n}  \| _{L^{2}(\Omega)}^{2} \right.\\ 
	&+\frac{(\gamma_{2}+\gamma_{3})\Delta t}{2}\|d_{t}\varsigma_{h}^{n}  \| _{L^{2}(\Omega)}^{2}-\frac{k_2}{2} \Norm{d_t \varsigma_h^n-d_t \varpi_h^n}{L^{2}(\Omega)}\\
	&\left.+\frac{\gamma_{6}\Delta t}{2}\|d_{t}\tau_{h}^{n+1}  \| _{L^{2}(\Omega)}^{2}-\left( \phi,p_{h}^{n+1}\right)
	-\left\langle \phi_1,p_{h}^{n+1} \right\rangle -\left( g,T_{h}^{n+1}\right)
	-\left\langle g_1,T_{h}^{n+1} \right\rangle
	\right].
\end{align*}
Then through letting
\begin{align*}
	\frac{\mu}{2}\Delta t-\frac{c_1^{2}\gamma_1^{2}\mu^{2}k(\tau_h^{n})}{2\mu_fh^{2}\beta_1^{2}}(\Delta t)^{2}-
	\frac{c_1^{2}\gamma_1^{2}\mu^{2}\bt }{2\mu_fh^{2}\beta_1^{2}}(\Delta t)^{2}\geq \frac{\mu}{4}\Delta t,
\end{align*}
we obtain
\begin{align*}
	\Delta t\leq \frac{C\mu_f\beta_1^{2}h^{2}}{2c_1^{2}(\gamma_1^{2}k_M+\gamma_{4}^2 \theta_M)\mu},
\end{align*}
so we  provide that $ \Delta t\leq \frac{C\mu_f\beta_1^{2}h^{2}}{2c_1^{2}(\gamma_1^{2}k_M+\gamma_{4}^2 \theta_M)\mu} $ when $ \theta=0 $. the proof is complete.
\end{proof}

\subsection{Convergence analysis}
To derive the optimal order error estimates of the fully discrete multiphysics finite element method, for any $\varphi \in L^{2}(\Omega)$, we firstly define $L^{2}(\Omega)$-projection operators $\mathcal{Q}_{h}: L^{2}(\Omega)\rightarrow X^{k}_{h}$  by
\begin{eqnarray}
	(\mathcal{Q}_{h}\varphi,\psi_{h})=(\varphi,\psi_{h})~~~~~\psi_{h}\in X^{k}_{h},\label{eq210607-1}
\end{eqnarray}
where $X^{k}_{h}:=\{\psi_{h}\in C^{0};~\psi_{h}|_{E}\in P_{k}(E)~~\forall E\in\mathcal{T}_{h}\}$, $k$ is the degree of piece-wise polynomial on $E$.

Next, for any $\varphi\in H^{1}(\Omega)$, we define its elliptic projection $\mathcal{S}_{h}: H^{1}(\Omega)\rightarrow X^{k}_{h}$ by
\begin{align}
	(k\nabla \mathcal{S}_{h}\varphi,&\nabla\varphi_{h})=(k\nabla\varphi,\nabla\varphi_{h})~~~~~\forall \varphi_{h}\in X^{k}_{h},\label{3-48}\\
	&(\mathcal{S}_{h}\varphi,1)=(\varphi,1).\label{3-49}
\end{align}

Finally, for any $\textbf{v}\in\textbf{H}^{1}(\Omega)$, we define its elliptic projection $\mathcal{R}_{h}:\textbf{H}^{1}(\Omega)\rightarrow\textbf{V}^{k}_{h}$ by
\begin{eqnarray}
	(\nabla\mathcal{R}_{h}\textbf{v},\nabla\textbf{w}_{h})=(\nabla\textbf{v},\nabla\textbf{w}_{h})~~~~~
	\forall\textbf{w}_{h}\in \textbf{V}^{k}_{h},
\end{eqnarray}
where $\textbf{V}^k_{h}:=\{\textbf{v}_{h} \in \textbf{C}^{0};~\textbf{v}_{h}|_{E}\in \textbf{P}_{k}(E), (\textbf{v}_{h},\textbf{r})=0 ~\forall \textbf{r} \in \textbf{RM}\}$.\\
From \cite{Quarteroni1997}, we know that  $\mathcal{Q}_{h},\mathcal{S}_{h}$ and $\mathcal{R}_{h}$ satisfy
\begin{align}
	\|\mathcal{Q}_{h}\varphi-\varphi\|_{L^{2}(\Omega)}+&h\|\nabla(\mathcal{Q}_{h}\varphi-\varphi)\|_{L^{2}(\Omega)}\\\nonumber
	&\leq Ch^{s+1}\|\varphi\|_{H^{s+1}(\Omega)}~~\forall \varphi\in H^{s+1}(\Omega),~~  0\leq s\leq k,\label{eq-3-47}\\	\|\mathcal{S}_{h}\varphi-\varphi\|_{L^{2}(\Omega)}+&h\|\nabla(\mathcal{S}_{h}\varphi-\varphi)\|_{L^{2}(\Omega)}\\\nonumber	&\leq Ch^{s+1}\|\varphi\|_{H^{s+1}(\Omega)}~~\forall \varphi\in H^{s+1}(\Omega),~~  0\leq s\leq k,\label{eq-3-47-2}\\
	\|\mathcal{R}_{h}\textbf{v}-\textbf{v}\|_{L^{2}(\Omega)}+&h\|\nabla(\mathcal{R}_{h}\textbf{v}-\textbf{v})\|_{L^{2}(\Omega)}\\\nonumber
	&\leq Ch^{s+1}\|\textbf{v}\|_{H^{s+1}(\Omega)}~~\forall \textbf{v}\in \textbf{H}^{s+1}(\Omega),~~  0\leq s\leq k.\label{eq-3-47-3}
\end{align}
To derive error estimates, we introduce the following notations:
\begin{alignat}{2}
	E^{n}_{\textbf{u}}&:=\textbf{u}(t_{n})-\textbf{u}^{n}_{h},~~~
	E^{n}_{\tau} &:=\tau(t_{n})-\tau^{n}_{h},~~~~
	&E^{n}_{\varpi}:=\varpi(t_{n})-\varpi^{n}_{h},\no\\
	E^{n}_{p}&:=p(t_{n})-p^{n}_{h},~~~~
	E^{n}_{\varsigma}&:=\varsigma(t_{n})-\varsigma^{n}_{h},~~~~
	&E^{n}_{q}:=q(t_{n})-q^{n}_{h}.\no
\end{alignat}
It is easy to check that
\begin{eqnarray}
	&&E^{n}_{p}=\gamma_{4}E^{n}_{\tau}+\gamma_{5}E^{n}_{\varpi}+\gamma_{2}E^{n}_{\varsigma},\qquad
	E^{n}_{T}=\gamma_{1}E^{n}_{\tau}+\gamma_{2}E^{n}_{\varpi}+\gamma_{3}E^{n}_{\varsigma},\nonumber\\
	&&E^{n}_{q}=-\gamma_{6}E^{n}_{\tau}+\gamma_{4}E^{n}_{\varpi}+\gamma_{1}E^{n}_{\varsigma}.
\end{eqnarray}
Also, we denote
\begin{eqnarray*}
	&&E^{n}_{\textbf{u}}=\textbf{u}(t_{n})-\mathcal{R}_{h}(\textbf{u}(t_{n}))+\mathcal{R}_{h}(\textbf{u}(t_{n}))-\textbf{u}^{n}_{h}:=\varLambda^{n}_{\textbf{u}}+\varPi^{n}_{\textbf{u}},\\
	&&E^{n}_{\tau}=\tau(t_{n})-\mathcal{Q}_{h}(\tau(t_{n}))+\mathcal{Q}_{h}(\tau(t_{n}))-\tau^{n}_{h}:=\Phi^{n}_{\tau}+\Psi^{n}_{\tau},\\
	&&E^{n}_{\tau}=\tau(t_{n})-\mathcal{S}_{h}(\tau(t_{n}))+\mathcal{S}_{h}(\tau(t_{n}))-\tau^{n}_{h}:=\varLambda^{n}_{\tau}+\varPi^{n}_{\tau},\\
	&&E^{n}_{\varpi}=\varpi(t_{n})-\mathcal{Q}_{h}(\varpi(t_{n}))+\mathcal{Q}_{h}(\varpi(t_{n}))-\varpi^{n}_{h}:=\Phi^{n}_{\varpi}+\Psi^{n}_{\varpi},\\
	&&E^{n}_{\varsigma}=\varsigma(t_{n})-\mathcal{Q}_{h}(\varsigma(t_{n}))+\mathcal{Q}_{h}(\varsigma(t_{n}))-\varsigma^{n}_{h}:=\Phi^{n}_{\varsigma}+\Psi^{n}_{\varsigma},\\	
	&&E^{n}_{p}=p(t_{n})-\mathcal{Q}_{h}(p(t_{n}))+\mathcal{Q}_{h}(p(t_{n}))-p^{n}_{h}:=\Phi^{n}_{p}+\Psi^{n}_{p},\\
	&&E^{n}_{p}=p(t_{n})-\mathcal{S}_{h}(p(t_{n}))+\mathcal{S}_{h}(p(t_{n}))-p^{n}_{h}:=\varLambda^{n}_{p}+\varPi^{n}_{p},\\
	&&E^{n}_{T}=T(t_{n})-\mathcal{Q}_{h}(T(t_{n}))+\mathcal{Q}_{h}(T(t_{n}))-T^{n}_{h}:=\Phi^{n}_{T}+\Psi^{n}_{T},\\
	&&E^{n}_{T}=T(t_{n})-\mathcal{S}_{h}(T(t_{n}))+\mathcal{S}_{h}(T(t_{n}))-T^{n}_{h}:=\varLambda^{n}_{T}+\varPi^{n}_{T}.
\end{eqnarray*}
\begin{lemma}\label{lem-3-4}
	Let ${(\bm{{\rm u}}^{n}_{h},\xi^{n}_{h},\varpi^{n}_{h},\gamma^{n}_{h})}$ be generated by the MFEA, then we have
	\begin{eqnarray}\label{3-52}
		&&\varSigma^{l}_{h}+\Delta t\sum_{n=0}^{l}[k_m(\nabla\widehat{\varPi}^{n}_{p},\nabla\widehat{\varPi}^{n}_{p})
		+(\bm{\Theta}\nabla\widehat{\varPi}^{n}_{T},\nabla\widehat{\varPi}^{n}_{T})
		+\frac{\Delta t}{2}(\mu\|d_{t}\nabla(\varPi^{n+1}_{\bm{{\rm u}}})\|^2_{L^{2}(\Omega)}\nonumber\\
		&&+\gamma_{6}\|d_{t}\Psi^{n+1}_{\tau}\|^2_{L^{2}(\Omega)}
		+(\gamma_{5}-\gamma_{2})\|d_{t}\Psi^{n+\theta}_{\varpi}\|^2_{L^{2}(\Omega)}+(\gamma_{3}-\gamma_{2})\|d_{t}\Psi^{n+\theta}_{\varsigma}\|^2_{L^{2}(\Omega)})]\nonumber\\
		&&\leq \widehat{\varSigma}^{-1}_{h}+\Delta t\sum_{n=0}^{l} [(\Phi^{n+1}_{\tau},\operatorname{div} d_{t}\varPi^{n+1}_{\bm{{\rm u}}})-(\operatorname{div} d_{t}\varLambda^{n+1}_{\bm{{\rm u}}},\Psi^{n+1}_{\tau})]\nonumber\\
		&&-\Delta t\sum_{n=0}^{l}[((k(\tau)-k(\tau_h^{n+1}))
		\nabla\left(\mathcal{S}_{h} p\right),\nabla \widehat{\Pi}_{p}^{n+1})]\no\\
		&&+\Delta t\sum_{n=0}^{l}[(d_{t}\Psi^{n+\theta}_{\varpi},\widehat{\varLambda}^{n+1}_{p}-\widehat{\Phi}^{n+1}_{p})+(d_{t}\Psi^{n+\theta}_{\varsigma},\widehat{\varLambda}^{n+1}_{T}-\widehat{\Phi}^{n+1}_{T})]\nonumber\\
		&&+(1-\theta)(\Delta t)^{2}\sum_{n=0}^{l}[\gamma_{4}(d^{2}_{t}\varpi(t_{n+1}),\Psi^{n+1}_{\tau})
		+\gamma_{1}(d^{2}_{t}\varsigma(t_{n+1}),\Psi^{n+1}_{\tau})]\nonumber\\
		&&+(1-\theta)(\Delta t)^{2}\sum_{n=0}^{l}[\gamma_{4}\Delta t(k(\tau_{h}^{n+1})d_{t}\nabla\varPi^{n+1}_{\tau},\nabla\widehat{\varPi}^{n+1}_{p})
		+\gamma_{1}\Delta t(\bm{\Theta}d_{t}\nabla\varPi^{n+1}_{\tau},\nabla\widehat{\varPi}^{n+1}_{T})]\nonumber\\
		&&+\Delta t\sum_{n=0}^{l}[(R^{n+\theta}_{\varpi},\widehat{\varPi}^{n+1}_{p})+(R^{n+\theta}_{\varsigma},\widehat{\varPi}^{n+1}_{T})],
	\end{eqnarray}
	where
	\begin{align}
		\widehat{\varPi}^{n+1}_{p}&:=\gamma_{4} \varPi^{n+1}_{\tau}+\gamma_{5}\varPi^{n+\theta}_{\varpi}+\gamma_{2} \varPi^{n+\theta}_{\varsigma},\\
		\widehat{\varPi}^{n+1}_{T}&:=\gamma_{1} \varPi^{n+1}_{\tau}+\gamma_{2}\varPi^{n+\theta}_{\varpi}+\gamma_{3} \varPi^{n+\theta}_{\varsigma},\\
		\varSigma^{l}_{h}:&=\frac{1}{2}(\mu\|\nabla(\varPi^{l+1}_{\bm{{\rm u}}})\|^2_{L^{2}(\Omega)}+\gamma_{6}\|\Psi^{l+1}_{\tau}\|^2_{L^{2}(\Omega)}\\\nonumber
		&+(\gamma_{5}-\gamma_{2})\|\Psi^{l+\theta}_{\varpi}\|^2_{L^{2}(\Omega)}+(\gamma_{3}-\gamma_{2})\|\Psi^{l+\theta}_{\varsigma}\|^2_{L^{2}(\Omega)}),\\
		\widehat{\varSigma}^{-1}_{h}:&=\frac{1}{2}(\mu\|\nabla(\varPi^{0}_{\bm{{\rm u}}})\|^2_{L^{2}(\Omega)}+\gamma_{6}\|\Psi^{0}_{\tau}\|^2_{L^{2}(\Omega)}\\\nonumber
		&+(\gamma_{5}+\gamma_{2})\|\Psi^{\theta-1}_{\varpi}\|^2_{L^{2}(\Omega)}+(\gamma_{3}+\gamma_{2})\|\Psi^{\theta-1}_{\varsigma}\|^2_{L^{2}(\Omega)}),\\
		R^{n+1}_{\varpi}:&=-\frac{1}{\Delta t}\int_{t_{n}}^{t_{n+1}}(t-t_{n})\varpi_{tt}(t)dt,~R^{n+1}_{\varsigma}:=-\frac{1}{\Delta t}\int_{t_{n}}^{t_{n+1}}(t-t_{n})\varsigma_{tt}(t)dt.
	\end{align}
\end{lemma}
\begin{proof}
	Subtracting (\ref{3-5}) form (\ref{2-18}), (\ref{3-6}) from (\ref{2-19}), (\ref{3-7}) from  (\ref{2-20}), (\ref{3-8}) from (\ref{2-22}), respectively, we get the following error equations:
	\begin{align}
\mu\left(\nabla\left(E_{\mathbf{u}}^{n+1}\right),\nabla\left(\mathbf{v}_{h}\right)\right)-\left(E_{\tau}^{n+1},\nabla \cdot \mathbf{v}_{h}\right)=0 \quad \forall \mathbf{v}_{h} \in \mathbf{V}_{h},\\
\gamma_{6}\left(E_{\tau}^{n+1},\varphi_{h}\right)+\left(\nabla \cdot E_{\mathbf{u}}^{n+1},\varphi_{h}\right)=\gamma_{4}\left(E_{\varpi}^{n+\theta},\varphi_{h}\right)+(1-\theta) \gamma_{4} \Delta t\left(d_{t} \varpi\left(t_{n+1}\right),\varphi_{h}\right)\no \\
+\gamma_{1}\left(E_{\varsigma}^{n+\theta},\varphi_{h}\right)+(1-\theta) \gamma_{1} \Delta t\left(d_{t} \varsigma\left(t_{n+1}\right),\varphi_{h}\right) \quad \forall \tau_{h} \in M_{h},\\
\left(d_{t} E_{\varpi}^{n+\theta}, y_{h}\right)+\left(k(\tau)
 \nabla\left(p\right)-k(\tau_h^{n+1})
 \nabla\left(p_h^{n+1}\right),\nabla y_{h}\right)+(1-\theta) \gamma_{4} \Delta t\left(k(\tau_h^{n+1}) d_{t} \nabla E_{\tau}^{n+1},\nabla y_{h}\right) \no \\
=\left(R_{\varpi}^{n+\theta}, y_{h}\right) \quad \forall y_{h} \in W_{h},\\
\left(d_{t} E_{\varsigma}^{n+\theta}, z_{h}\right)+\left(\boldsymbol{\Theta}\left(\widehat{E}_{T}^{n+1}\right),\nabla z_{h}\right)-(1-\theta) \gamma_{1} \Delta t\left(\boldsymbol{\Theta} d_{t} \nabla E_{\tau}^{n+1},\nabla z_{h}\right) \no \\
=\left(R_{\varsigma}^{n+\theta}, z_{h}\right) \quad \forall z_{h} \in Z_{h}.
	\end{align}
Using \eqref{3-48}, we obtain
\begin{align}\label{3-61}
	&k(\tau)
	\nabla\left(p\right)-k(\tau_h^{n+1})
	\nabla\left(p_h^{n+1}\right)\no\\
	=&k(\tau)
	\left(\nabla p -\nabla \mathcal{S}_{h} p\right)+(k(\tau)-k(\tau_h^{n+1}))
	\nabla\left(\mathcal{S}_{h} p\right)+k(\tau_h^{n+1})(\nabla \mathcal{S}_{h} p-\nabla p_h^{n+1})\no\\
	=&k(\tau)
	\nabla \Lambda_h^{n+1}+(k(\tau)-k(\tau_h^{n+1}))
	\nabla\left(\mathcal{S}_{h} p\right)+k(\tau_h^{n+1})\nabla \Pi_p^{n+1}.
\end{align}
Using the definition of the projection operators $\mathcal{Q}_{h},\mathcal{S}_{h},\mathcal{R}_{h}$,\eqref{3-61} and  the above error equations, we have
\begin{align}
	\mu\left(\nabla\left(\Pi_{\mathbf{u}}^{n+1}\right),\nabla\left(\mathbf{v}_{h}\right)\right)-\left(\Psi_{\tau}^{n+1},\nabla \cdot \mathbf{v}_{h}\right)=\left(\Phi_{\tau}^{n+1},\nabla \cdot \mathbf{v}_{h}\right) \quad \forall \mathbf{v}_{h} \in V_{h},\label{3-62}\\
	\gamma_{6}\left(\Psi_{\tau}^{n+1},\varphi_{h}\right)+\left(\nabla \cdot \Pi_{\mathrm{u}}^{n+1},\varphi_{h}\right)=-\left(\nabla \cdot \Lambda_{\mathbf{u}}^{n+1},\varphi_{h}\right)+\gamma_{4}\left(\Psi_{\varpi}^{n+\theta},\varphi_{h}\right) \no\\
	+(1-\theta) \gamma_{4} \Delta t\left(d_{t} \varpi\left(t_{n+1}\right),\varphi_{h}\right)+\gamma_{1}\left(\Psi_{\varsigma}^{n+\theta},\varphi_{h}\right) \no\\
	+(1-\theta) \gamma_{1} \Delta t\left(d_{t} \varsigma\left(t_{n+1}\right),\varphi_{h}\right) \quad \forall \varphi_{h} \in M_{h},\label{3-63}\\
	\left(d_{t} \Psi_{\varpi}^{n+\theta}, y_{h}\right)+\left((k(\tau)-k(\tau_h^{n+1}))
	\nabla\left(\mathcal{S}_{h} p\right)+k(\tau_h^{n+1})\nabla \Pi_p^{n+1},\nabla y_{h}\right)\no \\-(1-\theta) \gamma_{4} \Delta t\left(\boldsymbol{K} d_{t} \nabla \Pi_{\tau}^{n+1},\nabla y_{h}\right)
	=\left(R_{\varpi}^{n+\theta}, y_{h}\right) \quad \forall y_{h} \in W_{h},\label{3-64}\\
	\left(d_{t} \Psi_{\varsigma}^{n+\theta}, z_{h}\right)+\left(\boldsymbol{\Theta}\left(\widehat{\Pi}_{T}^{n+1}\right),\nabla z_{h}\right)-(1-\theta) \gamma_{1} \Delta t\left(\boldsymbol{\Theta} d_{t} \nabla \Pi_{\tau}^{n+1},\nabla z_{h}\right) \no\\
	=\left(R_{\varsigma}^{n+\theta}, z_{h}\right) \quad \forall z_{h} \in Z_{h}.\label{3-65}
\end{align}
Then taking $\textbf{v}_{h}=d_{t}\varPi^{n+1}_{\textbf{u}}$ in (\ref{3-62}), $\varphi_{h}=\Psi^{n+1}_{\tau}$ (after applying the difference operator $d_{t}$ to the equation (\ref{3-63}), $y_{h}=\widehat{\varPi}^{n+1}_{p}=\widehat{\Phi}^{n+1}_{p}-\widehat{\varLambda}^{n+1}_{p}+\gamma_{4}\Psi^{n+1}_{\tau}+\gamma_{5}\Psi^{n}_{\varpi}+\gamma_{2}\Psi^{n}_{\varsigma}$ in (\ref{3-64}) and $z_{h}=\widehat{\varPi}^{n+1}_{T}=\widehat{\Phi}^{n+1}_{T}-\widehat{\varLambda}^{n+1}_{T}+\gamma_{1}\Psi^{n+1}_{\tau}+\gamma_{2}\Psi^{n}_{\varpi}+\gamma_{3}\Psi^{n}_{\varsigma}$ in (\ref{3-65}), we get
\begin{align}
	\mu\left(\nabla\left(\Pi_{\mathrm{u}}^{n+1}\right),\nabla\left(d_{t} \Pi_{\mathbf{u}}^{n+1}\right)\right)+\gamma_{6}\left(d_{t} \Psi_{\tau}^{n+1},\Psi_{\tau}^{n+1}\right)+\gamma_{5}\left(d_{t} \Psi_{\varpi}^{n+\theta},\Psi_{\varpi}^{n+\theta}\right) \no\\
	+\gamma_{3}\left(d_{t} \Psi_{\varsigma}^{n+\theta},\Psi_{\varsigma}^{n+\omega}\right)+\gamma_{2}\left(d_{t} \Psi_{\varpi}^{n+\theta},\Psi_{\varsigma}^{n+\theta}\right)+\gamma_{2}\left(d_{t} \Psi_{\varsigma}^{n+\theta},\Psi_{\varpi}^{n+\theta}\right) \no\\
	+\left((k(\tau)-k(\tau_h^{n+1}))
	\nabla\left(\mathcal{S}_{h} p\right)+k(\tau_h^{n+1})\nabla \Pi_p^{n+1},\nabla \widehat{\Pi}_{p}^{n+1}\right)+\left(\boldsymbol{\Theta} \widehat{\Pi}_{T}^{n+1},\widehat{\Pi}_{T}^{n+1}\right) \no\\
	-(1-\theta) \gamma_{4} \Delta t\left(k(\tau_{h}^{n+1}) d_{t} \nabla \Pi_{\tau}^{n+1},\nabla \widehat{\Pi}_{p}^{n+1}\right)-(1-\theta) \gamma_{1} \Delta t\left(\boldsymbol{\Theta} d_{t} \nabla \Pi_{\tau}^{n+1},\nabla \Pi_{T}^{n+1}\right) \no\\
	=\left(\Phi_{\tau}^{n+1},\operatorname{div} d_{t} \Pi_{\mathrm{u}}^{n+1}\right)-\left(\operatorname{div} d_{t} \Lambda_{\mathrm{u}}^{n+1},\Psi_{\tau}^{n+1}\right)+\left(d_{t} \Psi_{\varpi}^{n+\theta},\widehat{\Lambda}_{p}^{n+1}-\widehat{\Phi}_{p}^{n+1}\right) \no\\
	+\left(d_{t} \Psi_{\varsigma}^{n+\theta},\widehat{\Lambda}_{T}^{n+1}-\widehat{\Phi}_{T}^{n+1}\right)+(1-\theta) \gamma_{4} \Delta t\left(d_{t}^{2} \varpi\left(t_{n+1}\right),\Psi_{\tau}^{n+1}\right) \no\\
	+(1-\theta) \gamma_{1} \Delta t\left(d_{t}^{2} \varsigma\left(t_{n+1}\right),\Psi_{\tau}^{n+1}\right)+\left(R_{h}^{n+\theta},\widehat{\Pi}_{p}^{n+1}\right)+\left(R_{\varsigma}^{n+\theta},\widehat{\Pi}_{T}^{n+1}\right) \text {. }\label{3-66}
\end{align}
From \eqref{3-27}--\eqref{3-29}, we obtain 
\begin{alignat}{2}
	\mu\left(\nabla\left(\Pi_{\mathrm{u}}^{n+1}\right),\nabla\left(d_{t} \Pi_{\mathrm{u}}^{n+1}\right)\right)=\frac{\mu}{2}\left(d_{t}\left\|\nabla\left(\Pi_{\mathrm{u}}^{n+1}\right)\right\|_{L^{2}(\Omega)}^{2}+\Delta t\left\|d_{t} \nabla\left(\Pi_{\mathrm{u}}^{n+1}\right)\right\|_{L^{2}(\Omega)}^{2}\right),\label{3-67}\\
	\gamma_{6}\left(d_{t} \Psi_{\tau}^{n+1},\Psi_{\tau}^{n+1}\right)=\frac{\gamma_{6}}{2}\left(d_{t}\left\|\Psi_{\tau}^{n+1}\right\|_{L^{2}(\Omega)}^{2}+\Delta t\left\|d_{t} \Psi_{\tau}^{n+1}\right\|_{L^{2}(\Omega)}^{2}\right),\\
	\gamma_{5}\left(d_{t} \Psi_{\varpi}^{n+\theta},\Psi_{\varpi}^{n+\theta}\right)=\frac{\gamma_{5}}{2}\left(d_{t}\left\|\Psi_{\varpi}^{n+\theta}\right\|_{L^{2}(\Omega)}^{2}+\Delta t\left\|d_{t} \Psi_{\varpi}^{n+\theta}\right\|_{L^{2}(\Omega)}^{2}\right),\\
	\gamma_{3}\left(d_{t} \Psi_{\varsigma}^{n+\theta},\Psi_{\varsigma}^{n+\theta}\right)=\frac{\gamma_{3}}{2}\left(d_{t}\left\|\Psi_{\varsigma}^{n+\theta}\right\|_{L^{2}(\Omega)}^{2}+\Delta t\left\|d_{t} \Psi_{\varsigma}^{n+\theta}\right\|_{L^{2}(\Omega)}^{2}\right),\\
	\gamma_{2}\left(d_{t} \Psi_{\varpi}^{n+\theta},\Psi_{\varsigma}^{n+\theta}\right)+\gamma_{2}\left(d_{t} \Psi_{\varsigma}^{n+\theta},\Psi_{\varpi}^{n+\theta}\right) \leq \frac{\gamma_{2}}{2}\left(\Delta t\left\|d_{t} \Psi_{\varpi}^{n+\theta}\right\|_{L^{2}(\Omega)}^{2}\right. \no \\
	\left.+\Delta t\left\|d_{t} \Psi_{\varsigma}^{n+\theta}\right\|_{L^{2}(\Omega)}^{2}\right)+\gamma_{2} d_{t}\left(\Psi_{\varpi}^{n+\theta},\Psi_{\varsigma}^{n+\theta}\right) \text {. } \label{3-72}
\end{alignat}
Substituting \eqref{3-67}--\eqref{3-72} into \eqref{3-66}, and applying the summation operation $\Delta t \textstyle \sum_{n=0}^{l}$ to the both sides of \eqref{3-66}, we implies that \eqref{3-52} holds. The proof is complete.
\end{proof}

\begin{lemma}\label{th-3-5}
Assuming that $\Delta t =O(h^2)$ when $\theta=0 $ and $\Delta t >0$ when $\theta=1$, we consider $P_{k}-P_{j}-P_{j}-P_{j}$ for $\bu,\tau,\varpi,\varsigma$ then there holds
\begin{alignat}{2}
&\max_{0\le n\leq l} [\sqrt{\mu}h^{j-k+1}\left\|\nabla\left(\Pi_{\mathbf{u}}^{n+1}\right)\right\|_{L^{2}(\Omega)}+\sqrt{\gamma_{6}}\left\|\Psi_{\tau}^{n+1}\right\|_{L^{2}(\Omega)}+\left(\sqrt{\gamma_{5}-\gamma_{2}}\right)\left\|\Psi_{\varpi}^{n+1}\right\|_{L^{2}(\Omega)}+\no\\
&\quad\left(\sqrt{\gamma_{3}-\gamma_{2}}\right)\left\|\Psi_{\varsigma}^{n+1}\right\|_{L^{2}(\Omega)} ]\no\\
&\quad+\Delta t \sum_{n=0}^{l}\left(\left\|  \widehat{\Pi}_{p}^{n+1}\right\|_{L^{2}(\Omega)}+\left\| \widehat{\Pi}_{T}^{n+1}\right\|_{L^{2}(\Omega)}\right)^{\frac{1}{2}}\no\\
&\leq C_1(t_f)\Delta t+ C_2(t_f) h^{j+1},\label{3.73}
\end{alignat}
where
\begin{alignat}{2}
	C_1(t)&=C(\Norm{\varpi_{t}}{L^2(0,t_f;L^2(\Omega))}^2+
	\Norm{\varsigma_{t}}{L^2(0,t_f;L^2(\Omega))}^2
	)\no\\
	&+C(\Norm{\varpi_{tt}}{L^2(0,t_f;H^1(\Omega))'}^2+
	\Norm{\varsigma_{tt}}{L^2(0,t_f;H^1(\Omega))'}^2),\no\\
	C_2(t)&=C(\Norm{\tau}{L^{\infty}(0,t_f;H^2(\Omega))}
	+\Norm{p}{L^{\infty}(0,t_f;H^2(\Omega))}+
	\Norm{T}{L^{\infty}(0,t_f;H^2(\Omega))})\no\\
	&+C(\Norm{\tau}{L^{2}(0,t_f;H^2(\Omega))}+\Norm{\tau_t}{L^{2}(0,t_f;H^2(\Omega))}+
	\Norm{p}{L^{2}(0,t_f;H^2(\Omega))}\no\\&+
	\Norm{T}{L^{2}(0,t_f;H^2(\Omega))}+
	\Norm{{\rm div}(\bu)_t}{L^{2}(0,t_f;H^2(\Omega))}).
\end{alignat}
\end{lemma}
\begin{proof}
To derive the above inequality, we need to bound each term on the right-hand side of \eqref{3-52}. Using the fact $\varPi_{\bu}^0,\varPi_{\tau}^0=0,\varPi_{\varpi}^{-1}=0$ and $\varPi_{\varsigma}^{-1}=0$, we have
		\begin{eqnarray}\label{3-73}
		&&\varSigma^{l}_{h}+\Delta t\sum_{n=0}^{l}[(k\nabla\widehat{\varPi}^{n}_{p},\nabla\widehat{\varPi}^{n}_{p})
		+(\bm{\Theta}\nabla\widehat{\varPi}^{n}_{T},\nabla\widehat{\varPi}^{n}_{T})
		+\frac{\Delta t}{2}(\mu\|d_{t}\nabla(\varPi^{n+1}_{\bm{{\rm u}}})\|^2_{L^{2}(\Omega)}\nonumber\\
		&&+\gamma_{6}\|d_{t}\Psi^{n+1}_{\tau}\|^2_{L^{2}(\Omega)}
		+(\gamma_{5}-\gamma_{2})\|d_{t}\Psi^{n+\theta}_{\varpi}\|^2_{L^{2}(\Omega)}+(\gamma_{3}-\gamma_{2})\|d_{t}\Psi^{n+\theta}_{\varsigma}\|^2_{L^{2}(\Omega)})]\nonumber\\
		&&\leq \Delta t\sum_{n=0}^{l} [(\Phi^{n+1}_{\tau},\operatorname{div} d_{t}\varPi^{n+1}_{\bm{{\rm u}}})-(\operatorname{div} d_{t}\varLambda^{n+1}_{\bm{{\rm u}}},\Psi^{n+1}_{\tau})]\nonumber\\
		&&-\Delta t\sum_{n=0}^{l}[(k(\tau)-k(\tau_h^{n+1}))
		\nabla\mathcal{S}_{h} p,\nabla \widehat{\Pi}_{p}^{n+1})]\no\\
		&&+\Delta t\sum_{n=0}^{l}[(d_{t}\Psi^{n+\theta}_{\varpi},\widehat{\varLambda}^{n+1}_{p}-\widehat{\Phi}^{n+1}_{p})+(d_{t}\Psi^{n+\theta}_{\varsigma},\widehat{\varLambda}^{n+1}_{T}-\widehat{\Phi}^{n+1}_{T})]\nonumber\\
		&&+(1-\theta)(\Delta t)^{2}\sum_{n=0}^{l}[\gamma_{4}(d^{2}_{t}\theta(t_{n+1}),\Psi^{n+1}_{\tau})
		+\gamma_{1}(d^{2}_{t}\varsigma(t_{n+1}),\Psi^{n+1}_{\tau})]\nonumber\\
		&&+(1-\theta)(\Delta t)^{2}\sum_{n=0}^{l}[\gamma_{4}\Delta t(k(\tau_{h}^{n+1})d_{t}\nabla\varPi^{n+1}_{\tau},\nabla\widehat{\varPi}^{n+1}_{p})
		+\gamma_{1}\Delta t(\bm{\Theta}d_{t}\nabla\varPi^{n+1}_{\tau},\nabla\widehat{\varPi}^{n+1}_{T})]\nonumber\\
		&&+\Delta t\sum_{n=0}^{l}[(R^{n+\theta}_{\varpi},\widehat{\varPi}^{n+1}_{p})+(R^{n+\theta}_{\varsigma},\widehat{\varPi}^{n+1}_{T})].
	\end{eqnarray}
We now estimate each term on the right-hand side of \eqref{3-73}. The last term on the right-hand side of \eqref{3-73} can be bounded by 
	\begin{align}
		&\left|\left(R_{\varpi}^{n+\theta},\widehat{\Pi}_{p}^{n+1}\right)\right| \leq\left\|R_{\varpi}^{n+\theta}\right\|_{H^{1}(\Omega)}\left\|\widehat{\Pi}_{p}^{n+1}\right\|_{L^{2}(\Omega)} \no\\
		& \leq \frac{k_{m}}{4}\left\|\widehat{\Pi}_{p}^{n+1}\right\|_{L^{2}(\Omega)}^{2}+\frac{1}{k_{m}}\left\|R_{\varpi}^{n+\theta}\right\|_{H^{1}(\Omega)^{\prime}}^{2} \no\\
		&\leq \frac{k_{m}}{8}\left\|\nabla \widehat{\Pi}_{p}^{n+1}\right\|_{L^{2}(\Omega)}^{2}+\frac{2 \Delta t}{3 k_{m}}\left\|\varpi_{t t}\right\|_{L^{2}\left(\left(t_{n+\theta-1}, t_{n+\theta}\right) ; H^{1}(\Omega)^{\prime}\right)}^{2},\label{3-74}
	\end{align}
where we have used the fact that
\begin{align}
	\left\|R_{\varpi}^{n+\theta}\right\|_{H^{1}(\Omega)^{\prime}}^{2} \leq \frac{\Delta t}{3} \int_{t_{n+\theta-1}}^{t_{n+\theta}}\left\|\varpi_{t t}\right\|_{H^{1}(\Omega)^{\prime}}^{2} d t.
\end{align}
Similarly, we can obtain
\begin{alignat}{2}
	\left|\left(R_{\varsigma}^{n+\theta},\widehat{\Pi}_{T}^{n+1}\right)\right| &\leq \frac{\theta_{m}}{8}\left\|\nabla \widehat{\Pi}_{T}^{n+1}\right\|_{L^{2}(\Omega)}^{2}+\frac{2 \Delta t}{3 \theta_{m}}\left\|\varsigma_{t t}\right\|_{L^{2}\left(\left(t_{n+\theta-1}, t_{n+\theta}\right) ; H^{1}(\Omega)^{\prime}\right)}^{2},\\
	|(R^{n}_{\varsigma},\widehat{\varPi}^{n+1}_{T})|&=|(R^{n}_{\varsigma},\widehat{\Phi}^{n+1}_{T}+\widehat{\Psi}^{n+1}_{T}-\widehat{\varLambda}^{n+1}_{T})|\nonumber\\
	&\leq \frac{\Delta t}{2}\|\varsigma_{tt}\|^{2}_{L^{2}(t_{n-1},t_{n};H^{1}(\Omega)')}+\frac{1}{2}\|\widehat{\Phi}^{n+1}_{T}\|^{2}_{L^{2}(\Omega)}\nonumber\\
	&+\frac{1}{2}\|\widehat{\Psi}^{n+1}_{T}\|^{2}_{L^{2}(\Omega)}+\frac{1}{2}\|\widehat{\varLambda}^{n+1}_{T}\|^{2}_{L^{2}(\Omega)}.
\end{alignat}
The first term on the right-hand side of (\ref{3-73}) can be bounded by
\begin{eqnarray}
	&&\Delta t\sum_{n=1}^{l} [(\Phi^{n+1}_{\tau},\operatorname{div} d_{t}\varPi^{n+1}_{\textbf{u}})-(\operatorname{div} d_{t}\varLambda^{n+1}_{\textbf{u}},\Psi^{n+1}_{\tau})]\nonumber\\
	&&=(\Phi^{l+1}_{\tau},\operatorname{div}\varPi^{l+1}_{\textbf{u}})-\Delta t\sum_{n=1}^{l}[(d_{t}\Phi^{n+1}_{\tau},\operatorname{div} \varPi^{n+1}_{\textbf{u}})+(\operatorname{div} d_{t}\varLambda^{n+1}_{\textbf{u}},\Psi^{n+1}_{\tau})]\nonumber\\
	&&\leq\frac{C}{\mu}\|\Phi^{l+1}_{\tau}\|^{2}_{L^{2}(\Omega)}+\frac{\mu}{4}\|\nabla(\varPi^{l+1}_{\textbf{u}})\|^{2}_{L^{2}(\Omega)}+\Delta t\sum_{n=1}^{l}\Big[\frac{C}{\mu}\|d_{t}\Phi^{n+1}_{\tau}\|^{2}_{L^{2}(\Omega)}\nonumber\\
	&&+\frac{\mu}{4}\|\nabla(\varPi^{n+1}_{\textbf{u}})\|^{2}_{L^{2}(\Omega)}+\frac{1}{\gamma_{6}}\|\operatorname{div} d_{t}\varLambda^{n+1}_{\textbf{u}}\|^{2}_{L^{2}(\Omega)}+\frac{\gamma_{6}}{4}\|\Psi^{n+1}_{\tau}\|^{2}_{L^{2}(\Omega)}\Big].
\end{eqnarray}
The second term on the right-hand side of (\ref{3-73}) can be bounded by
\begin{align}
	&(k(\tau)-k(\tau_h^{n+1}))
	\nabla\mathcal{S}_{h} p,\nabla \widehat{\Pi}_{p}^{n+1})\no\\
	&\leq L \Norm{\tau-\tau_h^{n+1}}{L^{2}(\Omega)}
	\Norm{	\nabla\mathcal{S}_{h} p}{L^{2}(\Omega)}
	\Norm{ \nabla \widehat{\Pi}_{p}^{n+1}}{L^{2}(\Omega)}\no\\
	&\leq \frac{L}{4 \epsilon_2}\Norm{\tau-\tau_h^{n+1}}{L^{2}(\Omega)}
	\Norm{	\nabla\mathcal{S}_{h} p}{L^{2}(\Omega)}
	+\epsilon_2 \Norm{ \nabla \widehat{\Pi}_{p}^{n+1}}{L^{2}(\Omega)}\no\\
	&\leq  \frac{L}{4 \epsilon_2}\Norm{\varLambda^{n+1}_{\tau}+\varPi^{n+1}_{\tau}}{L^{2}(\Omega)}
	\Norm{	\nabla\mathcal{S}_{h} p}{L^{2}(\Omega)}
	+\epsilon_2 \Norm{ \nabla \widehat{\Pi}_{p}^{n+1}}{L^{2}(\Omega)}\no\\
	&\leq  \frac{L}{4 \epsilon_2}\left[\Norm{\varLambda^{n+1}_{\tau}}{L^{2}(\Omega)}
	+\Norm{	\varPi^{n+1}_{\tau}}{L^{2}(\Omega)}\right]
	+\epsilon_2 \Norm{ \nabla \widehat{\Pi}_{p}^{n+1}}{L^{2}(\Omega)}.
\end{align}
The third term on the right-hand side of (\ref{3-73}) can be bounded by
\begin{align}
	&\Delta t\sum_{n=0}^{l}(d_{t}\Psi^{n+\theta}_{\varpi},\widehat{\varLambda}^{n+1}_{p}-\widehat{\Phi}^{n+1}_{p})\nonumber\\
	&=\Delta t\sum_{n=0}^{l}[\Delta t(d_t\Psi_{\varpi}^{n+\theta},d_t(\widehat{\varLambda}^{n+1}_{p}-\widehat{\Phi}^{n+1}_{p}))
	+d_t(\Psi^{n+\theta}_{\varpi},\widehat{\varLambda}^{n+1}_{p}-\widehat{\Phi}^{n+1}_{p})\nonumber\\
	&-(\Psi^{n+\theta}_{\varpi},d_t(\widehat{\varLambda}^{n+1}_{T}-\widehat{\Phi}^{n+1}_{p}))]
	\leq (\Delta t)^{2}\sum^{l}_{n=0}\big[\frac{\gamma_5-\gamma_2}{4}\|d_t\Psi^{n+\theta}_{\varpi}\|^{2}_{L^{2}(\Omega)}\nonumber\\
	&+\frac{1}{\gamma_5-\gamma_2}(\|d_t\widehat{\varLambda}^{n+1}_{p}\|^{2}_{L^{2}(\Omega)}+\|d_t\widehat{\Phi}^{n+1}_{p}\|^{2}_{L^{2}(\Omega)})\big]+\frac{\gamma_5-\gamma_2}{4}\|\Psi^{n+\theta}_{\varpi}\|^{2}_{L^{2}(\Omega)}\nonumber\\
	&+\frac{1}{\gamma_5-\gamma_2}(\|\widehat{\varLambda}^{l+1}_{p}\|^{2}_{L^{2}(\Omega)}+\|\widehat{\Phi}^{l+1}_{p}\|^{2}_{L^{2}(\Omega)})+\Delta t\sum^{l}_{n=0}\big[\frac{\gamma_5-\gamma_2}{4}\|\Psi^{n+\theta}_{\varpi}\|^{2}_{L^{2}(\Omega)}\nonumber\\
	&+\frac{1}{\gamma_5-\gamma_2}(\|d_t\widehat{\varLambda}^{n+1}_{p}\|^{2}_{L^{2}(\Omega)}+\|d_t\widehat{\Phi}^{n+1}_{p}\|^{2}_{L^{2}(\Omega)})\big],	
\end{align}
and
\begin{align}
	&\Delta t\sum_{n=0}^{l}(d_{t}\Psi^{n+\theta}_{\varsigma},\widehat{\varLambda}^{n+1}_{T}-\widehat{\Phi}^{n+1}_{T})\nonumber\\
	&\leq (\Delta t)^{2}\sum^{l}_{n=0}\big[\frac{\gamma_3-\gamma_2}{2}\|d_t\Psi^{n+\theta}_{\varsigma}\|^{2}_{L^{2}(\Omega)}+\frac{1}{2(\gamma_3-\gamma_2)}(\|d_t\widehat{\varLambda}^{n+1}_{T}\|^{2}_{L^{2}(\Omega)}\nonumber\\
	&+\|d_t\widehat{\Phi}^{n+1}_{T}\|^{2}_{L^{2}(\Omega)})\big]+\frac{\gamma_3-\gamma_2}{4}\|\Psi^{l+\theta}_{\varsigma}\|^{2}_{L^{2}(\Omega)}\nonumber\\
	&
	+\frac{1}{\gamma_3-\gamma_2}(\|\widehat{\varLambda}^{l+1}_{T}\|^{2}_{L^{2}(\Omega)}+\|\widehat{\Phi}^{l+1}_{T}\|^{2}_{L^{2}(\Omega)})+\Delta t\sum^{l}_{n=0}\big[\frac{\gamma_3-\gamma_2}{4}\|\Psi^{n+\theta}_{\varsigma}\|^{2}_{L^{2}(\Omega)}\nonumber\\
	&+\frac{1}{\gamma_3-\gamma_2}(\|d_t\widehat{\varLambda}^{n+1}_{T}\|^{2}_{L^{2}(\Omega)}+\|d_t\widehat{\Phi}^{n+1}_{T}\|^{2}_{L^{2}(\Omega)})\big].	
\end{align}
To bound the forth term on the right-hand side of (\ref{3-72}), we firstly use the summation by parts formula and $d_{t}\varpi_{h}(t_{0})=0$ and $d_{t}\varsigma_{h}(t_{0})=0$ to get
\begin{alignat}{2}\label{eq-3-82}
	\sum_{n=0}^{l}(d^{2}_{t}\varpi(t_{n+1}),\Psi^{n+1}_{\tau})=&\frac{1}{\Delta t}(d_{t}\varpi(t_{l+1}),\Psi^{l+1}_{\tau})-\sum^{l}_{n=1}(d_{t}\varpi(t_{n}),d_{t}\Psi^{n+1}_{\tau}),\\
	\sum_{n=0}^{l}(d^{2}_{t}\varsigma(t_{n+1}),\Psi^{n+1}_{\tau})=&\frac{1}{\Delta t}(d_{t}\varsigma(t_{l+1}),\Psi^{l+1}_{\tau})-\sum^{l}_{n=1}(d_{t}\varsigma(t_{n}),d_{t}\Psi^{n+1}_{\tau}).\label{eq-3-82-2}
\end{alignat}
Now, we bound each term on the right-hand side of (\ref{eq-3-82}) and (\ref{eq-3-82-2}) as 
\begin{eqnarray}
	&&\frac{1}{\Delta t}(d_{t}\varpi(t_{l+1}),\Psi^{l+1}_{\tau})\leq\frac{1}{\Delta t}\|d_{t}\varpi(t_{l+1})\|_{L^{2}(\Omega)}\|\Psi^{l+1}_{\tau}\|_{L^{2}(\Omega)}\nonumber\\
	&&\leq\frac{2\gamma_{4}}{\gamma_{6}}\|\varpi_{t}\|^{2}_{L^{2}(t_{l},t_{l+1};L^{2}(\Omega))}
	+\frac{\gamma_{6}}{8\gamma_{4}(\Delta t)^{2}}\|\Psi^{l+1}_{\tau}\|^{2}_{L^{2}(\Omega)},
\end{eqnarray}

\begin{eqnarray}
	\frac{1}{\Delta t}(d_{t}\varsigma(t_{l+1}),\Psi^{l+1}_{\tau}) \leq\frac{2\gamma_{1}}{\gamma_{6}}\|\varsigma_{t}\|^{2}_{L^{2}(t_{l},t_{l+1};L^{2}(\Omega))}+\frac{\gamma_{6}}{8\gamma_{1}(\Delta t)^{2}}\|\Psi^{l+1}_{\tau}\|^{2}_{L^{2}(\Omega)}.
\end{eqnarray}
\begin{eqnarray}
	&&\sum^{l}_{n=1}(d_{t}\varpi(t_{n}),d_{t}\Psi^{n+1}_{\tau})\leq\sum_{n=1}^{l}\|d_{t}\varpi(t_{n})\|_{L^{2}(\Omega)}\|d_{t}\Psi^{n+1}_{\tau}\|_{L^{2}(\Omega)}\nonumber\\
	&&\leq \sum_{n=1}^{l}\big(\frac{\gamma_{4}}{\gamma_{6}}\|d_{t}\varpi(t_{n})\|^{2}_{L^{2}(\Omega)}+\frac{\gamma_{6}}{4\gamma_{4}}\|d_{t}\Psi^{n+1}_{\tau}\|^{2}_{L^{2}(\Omega)}\big)\nonumber\\
	&&\leq\frac{\gamma_{4}}{\gamma_{6}}\|\varpi_{t}\|^{2}_{L^{2}(0,t_f;L^{2}(\Omega))}+\sum_{n=1}^{l}\frac{\gamma_{6}}{4\gamma_{4}}\|d_{t}\Psi^{n+1}_{\tau}\|^{2}_{L^{2}(\Omega)},
\end{eqnarray}

\begin{eqnarray}
	\sum^{l}_{n=1}(d_{t}\varsigma(t_{n}),d_{t}\Psi^{n+1}_{\tau})
	\leq\frac{\gamma_{1}}{\gamma_{6}}\|\varsigma_{t}\|^{2}_{L^{2}(0,t_f;L^{2}(\Omega))}+\sum_{n=1}^{l}\frac{\gamma_{6}}{4\gamma_{1}}\|d_{t}\Psi^{n+1}_{\tau}\|^{2}_{L^{2}(\Omega)}.
\end{eqnarray}
The fifth term of the right-hand side of (\ref{3-72}) can be bounded by
\begin{align}
	&\sum_{n=0}^{l}\gamma_{4}(k(\tau_{h}^{n+1})\nabla\varPi^{n}_{p},d_{t}\nabla \widehat{\varPi}^{n+1}_{\tau})\nonumber\\
	&\leq \sum_{n=0}^{l}\frac{\gamma_{4}k_{M}}{h}\|d_{t}\varPi^{n+1}_{\tau}\|_{L^{2}(\Omega)}\|\nabla\widehat{\varPi}^{n+1}_{p}\|_{L^{2}(\Omega)}\nonumber\\
	&\leq\frac{\gamma_{4}k_{M}}{h\beta_{1}}\sum_{n=0}^{l}\sup_{\textbf{v}_{h}\in \textbf{V}_{h}}\frac{\mu(d_{t}\nabla(\varPi_{\textbf{u}}^{n+1}),\nabla(\textbf{v}_{h}))-(d_{t}\varLambda_{\tau}^{n+1},\nabla\cdot\textbf{v}_{h})}{\|\nabla\textbf{v}_{h}\|_{L^{2}(\Omega)}}\|\nabla\widehat{\varPi}^{n+1}_{p}\|_{L^{2}(\Omega)}\nonumber\\
	&\leq\sum_{n=0}^{l}\Big[\frac{\mu}{4}\|d_{t}\nabla(\varPi_{\textbf{u}}^{n+1})\|^{2}_{L^{2}(\Omega)}+\frac{2k^{2}_{4}k^{2}_{M}}{\gamma_{m}h^{2}\beta^{2}_{1}-8k^{2}_{4}k^{2}_{M}\mu}\|d_{t}\varLambda_{\tau}^{n+1}\|^{2}_{L^{2}(\Omega)}\nonumber\\
	&+\frac{k_{M}}{8}\|\nabla\widehat{\varPi}^{n+1}_{p}\|^{2}_{L^{2}(\Omega)}\Big],
\end{align}
and
\begin{alignat}{2}\label{eq-3-91}
	\sum_{n=0}^{l}&\gamma_{1}(\bm{\Theta}\nabla\widehat{\varPi}^{n}_{T}, d_{t}\nabla \varPi^{n+1}_{\tau})\nonumber\\
	&\leq\sum_{n=0}^{l}\Big[\frac{\mu}{4}\|d_{t}\nabla(\varPi_{\textbf{u}}^{n+1})\|^{2}_{L^{2}(\Omega)}+\frac{2k^{2}_{1}\theta^{2}_{M}}{\theta_{m}h^{2}\beta^{2}_{1}-8k^{2}_{1}\theta^{2}_{M}\mu}\|d_{t}\varLambda_{\tau}^{n+1}\|^{2}_{L^{2}(\Omega)}\nonumber\\
	&+\frac{\theta_{m}}{8}\|\nabla\widehat{\varPi}^{n}_{T}\|^{2}_{L^{2}(\Omega)}\Big].
\end{alignat}
Substituting \eqref{3-74}--\eqref{eq-3-91} and using the discrete Gronwall lemma, we have 
\begin{align}
	&\mu\left\|\nabla\left(\Pi_{\mathbf{u}}^{l+1}\right)\right\|_{L^{2}(\Omega)}^{2}+\gamma_{6}\left\|\Psi_{\tau}^{l+1}\right\|_{L^{2}(\Omega)}^{2}+\left(\gamma_{5}-\gamma_{2}\right)\left\|\Psi_{\varpi}^{l+\theta}\right\|_{L^{2}(\Omega)}^{2}+\left(\gamma_{3}-\gamma_{2}\right)\left\|\Psi_{\varsigma}^{l+\theta}\right\|_{L^{2}(\Omega)}^{2} \no\\
	&\qquad+\Delta t \sum_{n=0}^{l}\left(\left\|k \nabla \widehat{\Pi}_{p}^{n+1}\right\|_{L^{2}(\Omega)}^{2}+\left\|\boldsymbol{\Theta} \nabla \widehat{\Pi}_{T}^{n+1}\right\|_{L^{2}(\Omega)}^{2}\right) \no\\
	&\qquad \leq \frac{2 \Delta t}{3 k_{m}}\left\|\varpi_{t t}\right\|_{L^{2}\left((0,\tau) ; H^{1}(\Omega)^{\prime}\right)}^{2}+\frac{2 \Delta t}{3 \theta_{m}}\left\|\varsigma_{t t}\right\|_{L^{2}\left((0,\tau) ; H^{1}(\Omega)^{\prime}\right)}^{2} \no\\
	&\qquad+(\Delta t)^{2}\left(\frac{2 \gamma_{4}^{2}}{\gamma_{6}}\left\|\varpi_{t}\right\|_{L^{2}\left((0,\tau) ; L^{2}(\Omega)\right)}^{2}+\frac{2 \gamma_{1}^{2}}{\gamma_{6}}\left\|\varsigma_{t}\right\|_{L^{2}\left((0,\tau) ; L^{2}(\Omega)\right)}^{2}\right) \no\\
	&\qquad+C\left[\Norm{\varLambda^{n+1}_{\tau}}{L^{2}(\Omega)}^2+\left\|\Phi_{\tau}^{l+1}\right\|_{L^{2}(\Omega)}^{2}+\left\|\widehat{\Lambda}_{p}^{l+1}\right\|_{L^{2}(\Omega)}^{2}+\left\|\widehat{\Phi}_{p}^{l+1}\right\|_{L^{2}(\Omega)}^{2}+\left\|\widehat{\Lambda}_{T}^{l+1}\right\|_{L^{2}(\Omega)}^{2}+\left\|\widehat{\Phi}_{T}^{l+1}\right\|_{L^{2}(\Omega)}^{2}\right. \no\\
	&\qquad+\Delta t \sum_{n=1}^{l}\left(\left\|d_{t} \Lambda_{\tau}^{n+1}\right\|_{L^{2}(\Omega)}^{2}+\left\|d_{t} \Phi_{\tau}^{n+1}\right\|_{L^{2}(\Omega)}^{2}\right)+\Delta t \sum_{n=0}^{l}\left\|\operatorname{div} d_{t} \Lambda_{\mathbf{u}}^{n+1}\right\|_{L^{2}(\Omega)}^{2}\no\\
	&\qquad+\Delta t \sum_{n=0}^{l}\left(\left\|\widehat{\Lambda}_{p}^{n+1}\right\|_{L^{2}(\Omega)}^{2}+\left\|\widehat{\Phi}_{p}^{n+1}\right\|_{L^{2}(\Omega)}^{2}+\left\|\widehat{\Lambda}_{T}^{n+1}\right\|_{L^{2}(\Omega)}^{2}+\left\|\widehat{\Phi}_{T}^{n+1}\right\|_{L^{2}(\Omega)}^{2}\right) \no\\
	&\qquad\left.+(\Delta t)^{2} \sum_{n=0}^{l}\left(\left\|\widehat{\Lambda}_{p}^{n+1}\right\|_{L^{2}(\Omega)}^{2}+\left\|\widehat{\Phi}_{p}^{n+1}\right\|_{L^{2}(\Omega)}^{2}+\left\|\widehat{\Lambda}_{T}^{n+1}\right\|_{L^{2}(\Omega)}^{2}+\left\|\widehat{\Phi}_{T}^{n+1}\right\|_{L^{2}(\Omega)}^{2}\right)\right],
\end{align}
provided that $\Delta t\leq \frac{C\mu_f\beta_1^{2}h^{2}}{2c_1^{2}(\gamma_1^{2}k_M+\gamma_{4}^2 \theta_m)\mu} $ when $ \theta=0 $, but it holds for all $\Delta t>0$ when $\theta=1$. Hence \eqref{3.73} follows from the approximation properties of the projection operators $Qh, Rh
, Sh$ and inverse inequality \eqref{3.4}. The proof is complete.
\end{proof}

We conclude this section by stating the main theorem as follows.
\begin{theorem}\label{th-3-6}
	Under the assumption of Theorem \ref{th-3-5}, the solution of the MFEA($P_{k}-P_{j}-P_{j}-P_{j}$) satisfies the following error estimates
	\begin{align}
		&\max_{0\leq n\leq l}[\sqrt{\mu}h^{j-k+1}\|\nabla(\bm{{\rm u}}(t_{n})-\bm{{\rm u}}^{n}_{h})\|_{L^{2}(\Omega)}+\sqrt{\gamma_{6}}\|\tau(t_{n})-\tau^{n}_{h}\|_{L^{2}(\Omega)}\nonumber\\
		&+\sqrt{\gamma_{5}-\gamma_{2}}\|\varpi(t_{n})-\varpi^{n}_{h}\|_{L^{2}(\Omega)}+\sqrt{\gamma_{3}-\gamma_{2}}\|\varsigma(t_{n})-\varsigma^{n}_{h}\|_{L^{2}(\Omega)}]\nonumber\\
		&\leq\widehat{C}_{1}(t_f)\Delta t+\widehat{C}_{2}(t_f)h^{j+1},\\
		&[\Delta t\sum_{n=0}^{l}\|\nabla(p(t_{n})-p^{n}_{h})\|^2_{L^{2}(\Omega)}+\|\nabla(T(t_{n})-T^{n}_{h})\|^2_{L^{2}(\Omega)}]^{\frac{1}{2}}\nonumber\\
		&\leq \widehat{C}_{1}(t_f)\Delta t+\widehat{C}_{2}(t_f)h^{j},
	\end{align}
	provided that $\Delta t=O(h^{2})$ when $\theta=0$ and $\Delta t>0$ when $\theta=1$, where
	\begin{align*}
		\widehat{C}_{1}(t_f):&=C_{1}(t_f),\\
		\widehat{C}_{2}(t_f):&=C_{2}(t_f)+\|\tau\|_{L^{\infty}(0,t_f;H^{2}(\Omega))}+\|\varpi\|_{L^{\infty}(0,t_f;H^{2}(\Omega))}\\
		&+\|\varsigma\|_{L^{\infty}(0,t_f;H^{2}(\Omega))}+\|\nabla \bm{{\rm u}}\|_{L^{\infty}(0,t_f;H^{2}(\Omega))}.
	\end{align*}
\end{theorem}
\begin{proof}
	The above estimates follow immediately from an application of the triangle inequality on
	\begin{eqnarray*}
		&&\textbf{u}(t_{n})-\textbf{u}^{n}_{h}=\varLambda^{n}_{\textbf{u}}+\varPi^{n}_{\textbf{u}},
		\quad\tau(t_{n})-\tau^{n}_{h}=\Phi^{n}_{\tau}+\Psi^{n}_{\tau},\\
		&&\tau(t_{n})-\tau^{n}_{h}=\varLambda^{n}_{\tau}+\varPi^{n}_{\tau},
		\quad\varpi(t_{n})-\varpi^{n}_{h}=\Phi^{n}_{\varpi}+\Psi^{n}_{\varpi},\\
		&&\varsigma(t_{n})-\varsigma^{n}_{h}=\Phi^{n}_{\varsigma}+\Psi^{n}_{\varsigma},
		\quad p(t_{n})-p^{n}_{h}=\Phi^{n}_{p}+\Psi^{n}_{p},\\
		&&p(t_{n})-p^{n}_{h}=\varLambda^{n}_{p}+\varPi^{n}_{p},
		\quad T(t_{n})-T^{n}_{h}=\Phi^{n}_{T}+\Psi^{n}_{T},\\
		&&T(t_{n})-T^{n}_{h}=\varLambda^{n}_{T}+\varPi^{n}_{T},
	\end{eqnarray*}
	and appealing to (\ref{eq-3-47}), (\ref{eq-3-47-2}), (\ref{eq-3-47-3}) and Theorem \ref{th-3-5}. The proof is complete.
\end{proof}
\section{ Numerical tests}\label{sec-4}		
In order to verify the effectiveness and accuracy of our proposed numerical method, we design four two-dimensional numerical experiments in this section.  In the first experiment, we use different types of finite element pairs to solve the same problem, and verify that our method has good adaptability for any finite element pair. In the second test, we solve another problem of different exact solution  using the  finite element pair ($P_2-P_1-P_1-P_1$), check whether the obtained errors and convergence orders are consistent with the theoretical results. In the third test, we consider a benchmark problem without an exact solution. In our solution results, we don't observe the locking phenomenon mentioned by Philips \cite{Wheeler2007}. At the same time, we take the numerical solution with a spatial step size of 1/180 as the exact solution, and calculate the convergence results at different step sizes, which are in complete agreement with the theoretical results. Finally, we adjust the value of the key parameter $b$ in the nonlinear term $k(\nabla \cdot \bu,p,T)$ based on the benchmark model, and find that it has a significant impact on the experimental results, indicating the importance of considering nonlinear hydraulic conduction $k$.\\

\textbf{Test 1.}\label{test1}
Let \ $\Omega= [0,1]\times[0,1]$,\ $\Gamma_{1} = \{(x,0);~0\leq x\leq1\}$,\ $~\Gamma_{2}= \{(1,y);~0\leq y\leq1 \}$,\ $\Gamma_{3}= \{(x,1);~0\leq x\leq1 \}$,\ $~\Gamma_{4}= \{ (0,y);~0\leq y\leq1 \}$, and $t_f=1,\Delta t=O(h^2)$. We consider the problem (\ref{1-1})--(\ref{1-3}) with the following source functions:
\begin{align*}
	\bm{{\rm f}} &=(\mu \pi^3+3\lambda \pi^{\frac{3}{4}}+(\alpha+\beta)\pi)(\exp(s)\cos(\pi x)\cos(\frac{\pi y}{2}))(1,1)^{'},\\
	\phi &=(a_{0}-b_{0}-\frac{3\beta \pi^2}{4}+\frac{5 \bt \pi^2}{4})(\exp(s)\sin(\pi x)\cos(\frac{\pi y}{2})),\\
	g &=(c_{0}-b_{0}-\frac{3\alpha \pi^2}{4}+\frac{5 k \pi^2}{4})(\exp(s)\sin(\pi x)\cos(\frac{\pi y}{2}))\\&+k_x(\exp(s)\cos(\pi x)\cos(\frac{\pi y}{2}))-k_y\frac{\pi}{2}(\exp(s)\sin(\pi x)\sin(\frac{\pi y}{2})),\\
	&\qquad ~~~~~~~~~~k=ak_{0}e^{-b(\frac{3}{4}(\lambda+\mu )\pi^2-\alpha-\beta)(\exp(s)\sin(\pi x)\cos(\frac{\pi y}{2})) },
\end{align*}
and the following boundary and initial conditions:
\begin{eqnarray}
	p=e^t\mathrm{sin}(\pi x)\mathrm{cos}(\frac{\pi y}{2}) \quad  &\mathrm{on} & \quad\partial\Omega_{t},\nonumber\\
	T=e^t\mathrm{sin}(\pi x)\mathrm{cos}(\frac{\pi y}{2})  \quad &\mathrm{on} & \quad\partial\Omega_{t},\nonumber\\
	u_1=\pi e^t\mathrm{cos}(\pi x)\mathrm{cos}(\frac{\pi y}{2}) \quad &\mathrm{on}& \quad\Gamma_j\times (0,t_f),~j=2,~4,\nonumber\\
	u_2=\frac{\pi}{2} e^t\mathrm{sin}(\pi x)\mathrm{sin}(\frac{\pi y}{2}) \quad &\mathrm{on} & \quad\Gamma_j\times(0,t_f),~j=1,~3,\nonumber\\
	\sigma(\pmb{u})\bm{n}-\alpha pI\bm{n}=\textbf{f}_1 \quad  &\mathrm{on}& \quad\partial\Omega_{t},\nonumber\\
	\textbf{u}(x,0)=\pi \big(\mathrm{cos}(\pi x)\mathrm{cos}(\frac{\pi y}{2}),\frac{1}{2} \mathrm{sin}(\pi x)\mathrm{sin}(\frac{\pi y}{2}) \big)^{'}\quad&\mathrm{in}& \quad\Omega,\nonumber\\
	p(x,0)=\mathrm{sin}(\pi x)\mathrm{cos}(\frac{\pi y}{2}),~T(x,0)=\mathrm{sin}(\pi x)\mathrm{cos}(\frac{\pi y}{2}) \quad&\mathrm{in}& \quad\Omega,\nonumber
\end{eqnarray}
where $\bm{{\rm f}}_{1}=e^t\mathrm{sin}(\pi x)\mathrm{cos}(\frac{\pi y}{2})((-\mu \pi^2-\frac{3}{4}\pi^{2}\lambda-(\alpha+\beta))n_1,(\frac{\pi^{2}\mu}{4}-\frac{3\pi^2\lambda}{4}-(\alpha+\beta))n_2)^{'}$.

It is easy to check that the exact solution are
\begin{align*}
	\bm{{\rm u}}=(\pi e^t\mathrm{cos}(\pi x)\mathrm{cos}(\frac{\pi y}{2}),\frac{\pi}{2} e^t\mathrm{sin}(\pi x)\mathrm{sin}(\frac{\pi y}{2}))^{'},\\
	p=e^t\mathrm{sin}(\pi x)\mathrm{cos}(\frac{\pi y}{2}),\quad T=e^t\mathrm{sin}(\pi x)\mathrm{cos}(\frac{\pi y}{2}).
\end{align*}
\begin{table}[H]
	\centering
	\caption{ Physical parameters}\label{table 5}
	\begin{tabular}{c l c }
		\hline
		Parameter   &  \quad Description      &   \quad   Value    \\
		\hline
		$a_0$       &\quad Effective thermal capacity               &  \quad 2e-1\\
		$b_0$       &\quad Thermal dilation coefficient              &  \quad 1e-1 \\
		$c_0$       &\quad Constrained specific storage coefficient &  \quad 2e-1 \\
		$\alpha$    &\quad Biot-Willis constant                     &  \quad  0.01 \\
		$\beta$     &\quad Thermal stress coefficient.               &  \quad  0.01\\
		$a$         &\quad Mutation coefficient                      & \quad    1\\
	$b$         & \quad Stress coupling coefficient              & \quad    1\\
	$k_0$         &\quad Initial permeability                       &  \quad 1e-5\\
		$\bm{\Theta}$     &\quad Effective thermal conductivity           &  \quad 1e-5$I$\\
		$E$         &\quad Young's modulus                          &  \quad  2e4\\
		$\nu$       &\quad Poisson ratio                            &  \quad  0.4\\
		\hline
	\end{tabular}
\end{table}

\begin{table}[htbp]
	\vspace{-2.0em}
	\begin{center}
		\caption{Error and convergence rates of $u_h^n$, $p_h^n$, $T_h^n$ with $P_2-P_1-P_1-P_1$}\label{table 4}
		\resizebox{\textwidth}{12mm}{
			\begin{tabular}{ccccccccccccc}
				\hline
				$h$  & $\frac{\|e_u\|_{L^2(\Omega)}}{\|u\|_{L^2(\Omega)}}$  &  CR  &  $\frac{\|e_u\|_{H^1(\Omega)}}{\|u\|_{H^1(\Omega)}}$  &  CR & $\frac{\|e_p\|_{L^2(\Omega)}}{\|p\|_{L^2(\Omega)}}$ & CR  &  $\frac{\|e_p\|_{H^1(\Omega)}}{\|p\|_{H^1(\Omega)}}$  &  CR  &  $\frac{\|e_T\|_{L^2(\Omega)}}{\|T\|_{L^2(\Omega)}}$  &  CR  & $\frac{\|e_T\|_{H^1(\Omega)}}{\|T\|_{H^1(\Omega)}}$ & CR \\ 
				\hline
				$1/4$   &5.5071e-02&  &3.8972e-02&  &8.8396e-02&  &3.0042e-01&  &1.1135e-01&  &2.8577e-01&  \\
				$1/8$  &6.8062e-03&3.0164 &9.2731e-03&2.0713 &1.8020e-02&2.2944 &1.4904e-01&1.0113 &2.8869e-02&1.9475 &1.4492e-01&0.9796 \\
				$1/16$  &8.4970e-04&3.0018 &2.2529e-03&2.0412 &4.0409e-03&2.1569 &7.3713e-02&1.0157 &7.2865e-03&1.9862 &7.2720e-02&0.9948 \\
				$1/32$  &1.0622e-04&2.9999 &5.5480e-04&2.0218 &9.5594e-04&2.0797 &3.6632e-02&1.0088 &1.8260e-03&1.9965 &3.6393e-02&0.9987 \\
				\hline
		\end{tabular}}
	\end{center}
\end{table}

\begin{table}[htbp]
	\vspace{-1.0em}
	\begin{center}
		\caption{Order of convergence of time discretization of Test 1 }\label{table311}
		\resizebox{\textwidth}{12mm}{
			\begin{tabular}{ccccccccccc}
				\hline
				$\Delta t$ & $\left\|\mathbf{u}_{h}^{\Delta t}-\mathbf{u}_{h}^{\frac{1}{2} \Delta t}\right\|_{L^2(\Omega)}$ & $\rho_{\Delta t,\mathbf{u}_{h}}$ & $\left\|p_{h}^{\Delta t}-p_{h}^{ \frac{1}{2} \Delta t}\right\|_{L^2(\Omega)}$ & $\rho_{\Delta t, p_h}$ & $\left\|T_{h}^{\Delta t}-T_{h}^{\frac{1}{2} \Delta t}\right\|_{L^2(\Omega)}$ & $\rho_{\Delta t, T_h}$ \\
				\hline
				$\frac{1}{10}$ &$ 4.9287e-10$& 1.9383 &$6.7904e-02$&2.0246 &$1.3030e-04$&1.9510   \\
				$\frac{1}{20}$ &$2.5428e-10$& 1.9688&$3.3539e-02$&2.0124 &$6.6786e-05$&1.9753 \\
				$\frac{1}{40}$ &$1.2916e-10$& 1.9843&$1.6666e-02$& 2.0062&$3.3811e-05$&1.9876 \\
				$\frac{1}{80}$ &$6.5089e-11$&  &$8.3073e-03$&  &$1.7011e-05$&   \\
				\hline
		\end{tabular}}
	\end{center}
\end{table}
\begin{figure}[htbp]
	\centering
	\subfigure[]{
		\centering
		\includegraphics[width=2.5in]{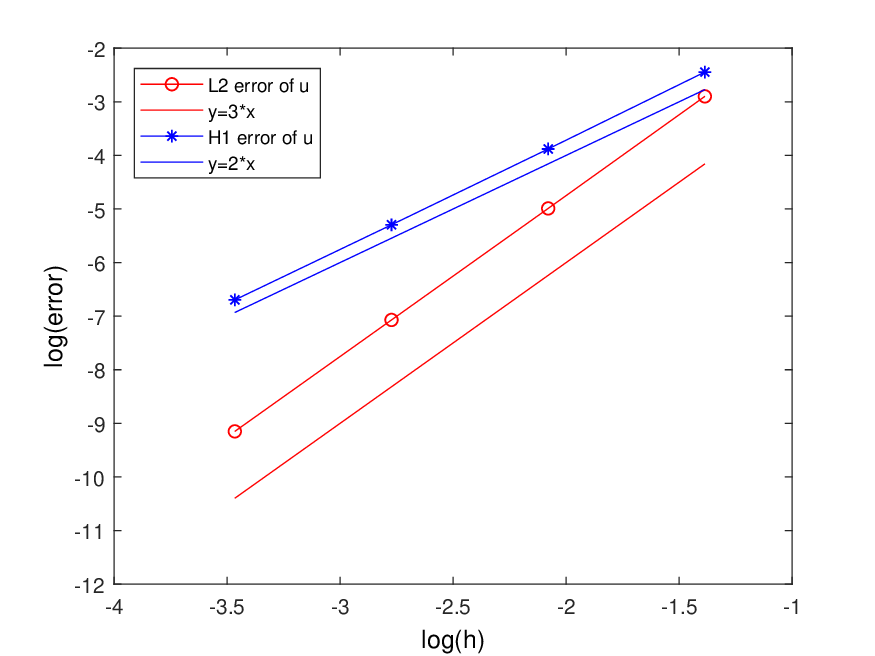}
		\label{fig1}
	}
	\subfigure[]{
		\centering
		\includegraphics[width=2.5in]{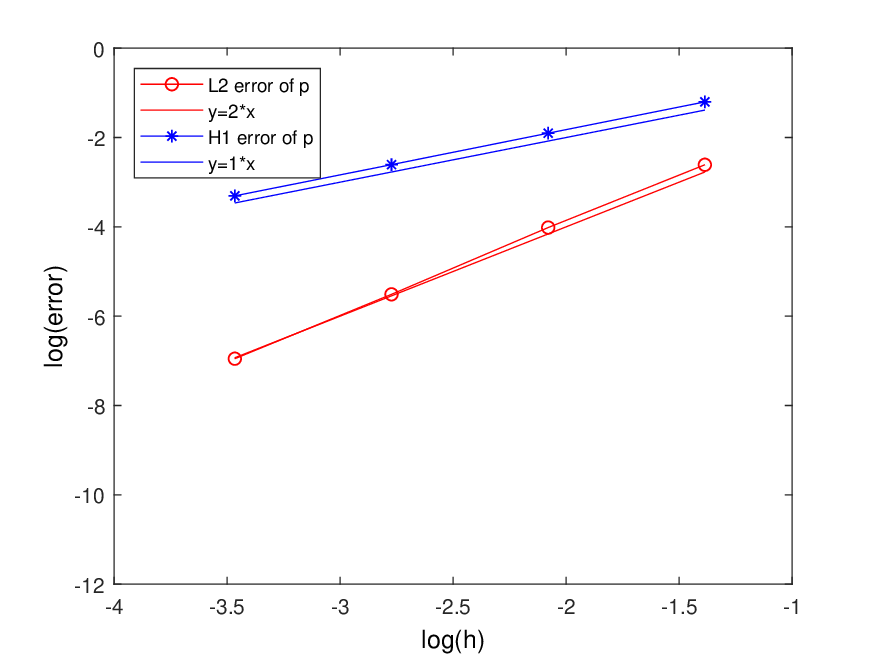}
		\label{fig2}
	}
	\subfigure[]{
	\centering
	\includegraphics[width=2.5in]{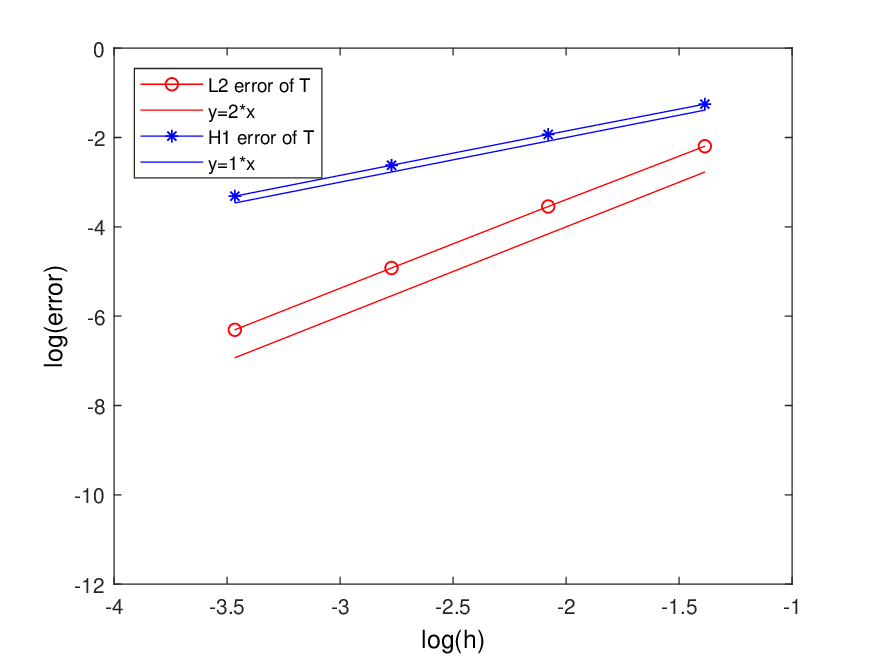}
	\label{fig221}
}
	\subfigure[]{
	\centering
	\includegraphics[width=2.5in]{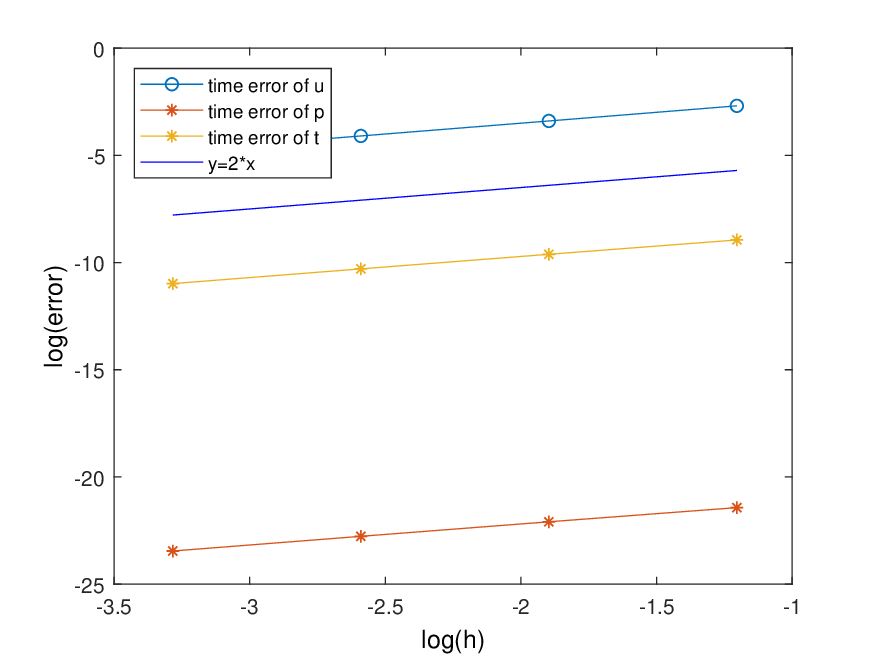}
	\label{fig222}
}
	\caption{(a) spatial convergence order for $u^{n}_{h}$, (b) spatial convergence rate for $p^{n}_{h}$,
	\\~~~~~~~~\quad \qquad(c) spatial convergence rate for $ T^{n}_{h}$,(d) time convergence order for $u^{n}_{h},p^{n}_{h},T^{n}_{h}$ }
\end{figure}
\begin{figure}[htbp]
	\centering
	\subfigure[]{
		\centering
		\includegraphics[width=2.5in]{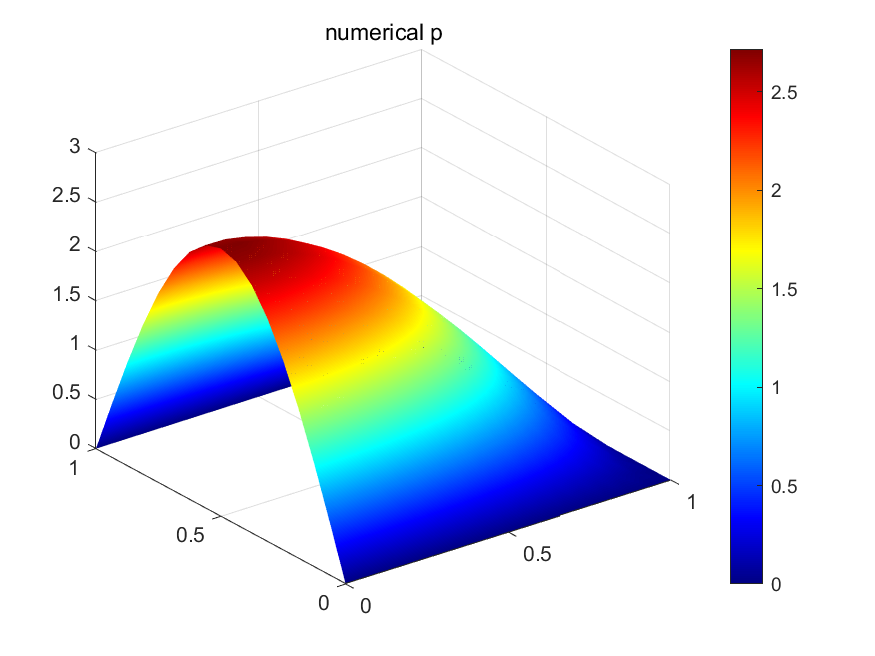}
		\label{fig51}
	}
	\subfigure[]{
		\centering
		\includegraphics[width=2.5in]{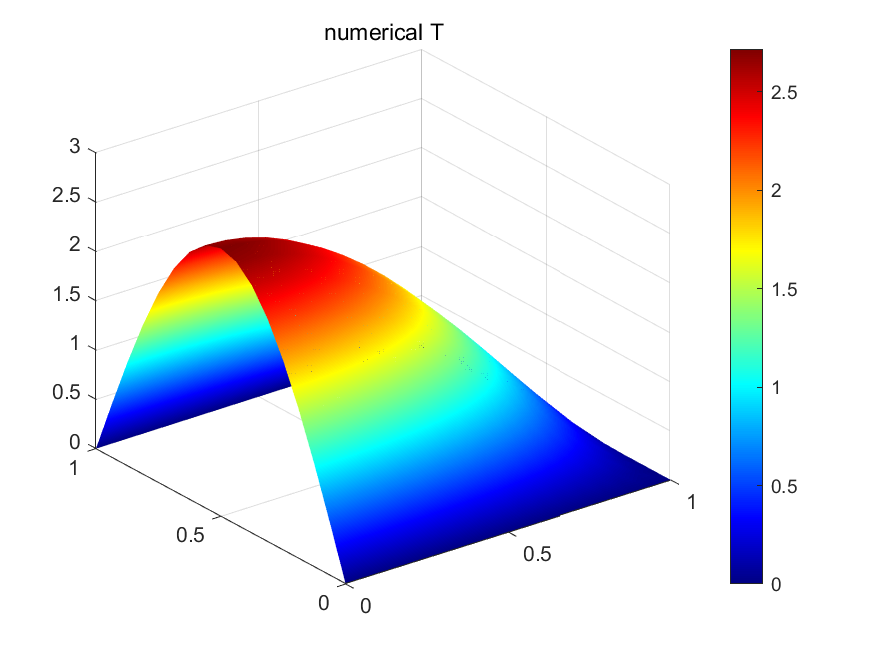}
		\label{fig61}
	}
	\caption{(a) surface plot of $p_h^n$ at the terminal time $t_f$, (b) surface plot of $T_h^n$ at the terminal time $t_f$.}
\end{figure}

\begin{figure}
	\centering
	\includegraphics[height=5cm,width=7cm]{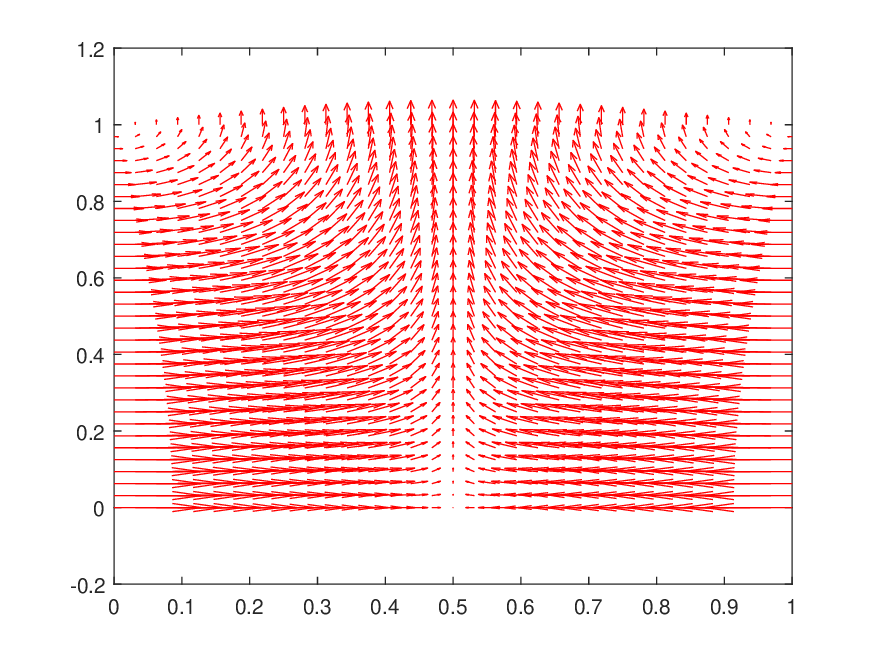}
	\caption{Arrow plot of the displacement $\textbf{u}_h^n$.}
	\label{fig71}
\end{figure}
\begin{figure}[H]
	\centering
	\subfigure[]{
		\centering
		\includegraphics[width=2.5in]{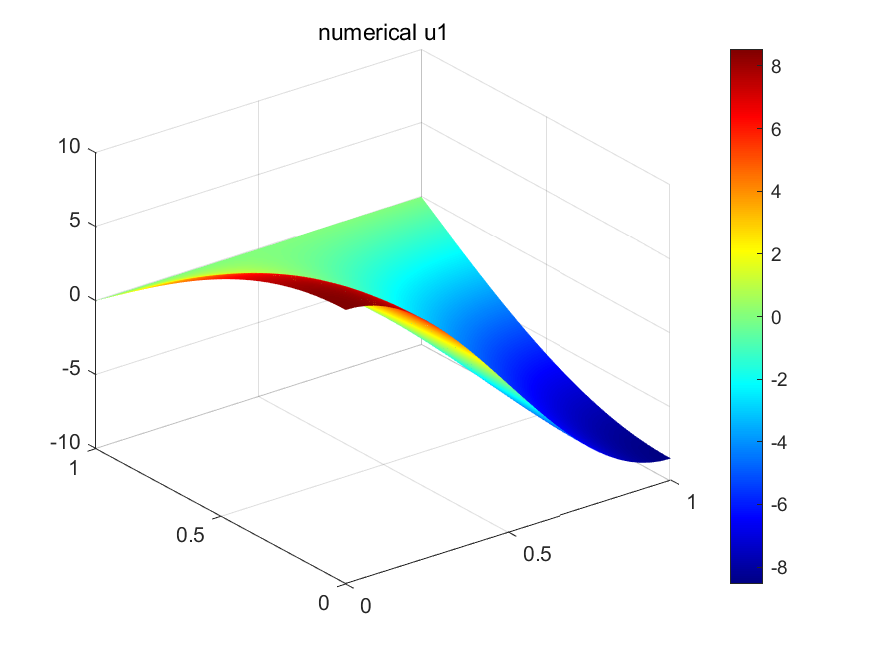}
	\label{fig81}}%
\subfigure[]{
	\centering
	\includegraphics[width=2.5in]{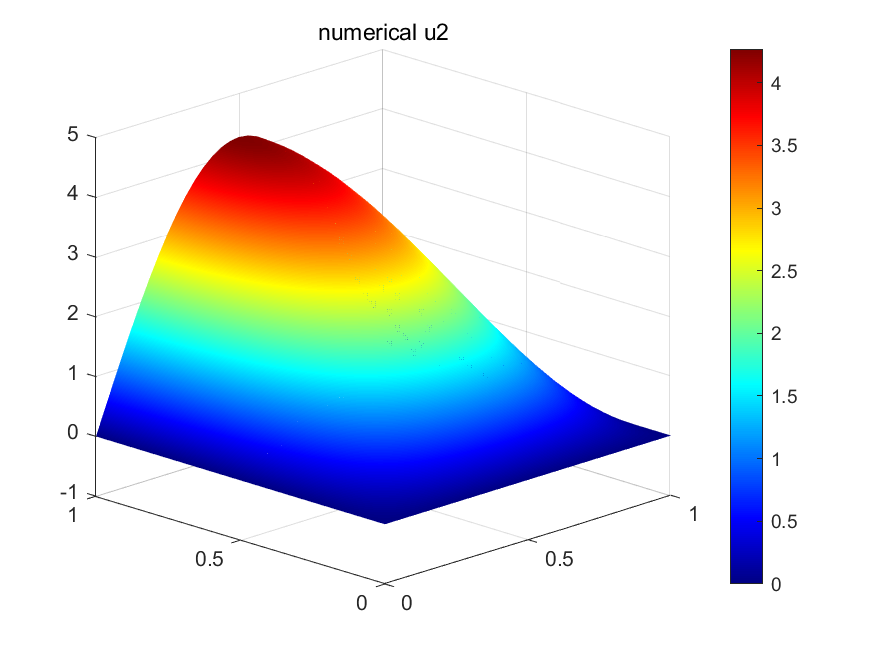}
\label{fig91}}%
\centering
\caption{(a) surface plot of $u_{1h}^{n}$ at the terminal time $t_f$, (b) surface plot of $u_{2h}^{n}$ at the terminal time $t_f$.}
\vspace{-2em}
\end{figure}

Table \ref{table2} displays the $L^{\infty}(0,t_f;L^{2}(\Omega))$-norm error and $L^{\infty}(0,t_f;H^{1}(\Omega))$-norm error of $u$, $p$, $T$ and the convergence order with respect to $h$ at terminal time $s$. Evidently, the spatial rates of convergence are consistent with Theorem \ref{th-3-6}.
Figure \ref{fig1} and Figure \ref{fig2} describe the spatial convergence order of $u^{n}_{h}, p^{n}_{h}$, $T^{n}_{h}$. Figure \ref{fig51}, Figure \ref{fig61}, Figure \ref{fig81} and Figure \ref{fig91} show, respectively, the surface plot of $p^{n}_{h}$, $T^{n}_{h}$, $u^{n}  _{1h}$ and $u^{n}_{2h}$ at the terminal time $s$ and Figure \ref{fig71} shows arrow plot of  $\textbf{u}_h^n$. They coincide with the theoretical results.\\

To test the results of problem under different pairs of finite elements, we calculated the four variables of the reformulated equations using $P_1-P_1-P_1-P_1$ and $P_2-P_2-P_2-P_2$ elements respectively, and observed the calculation results and convergence orders. The numerical solutions are shown in the following two diagrams.
\begin{table}[H]
	\vspace{-1.8em}
	\begin{center}
		\caption{Error and convergence rates of $u_h^n$, $p_h^n$, $T_h^n$ with $P_1-P_1-P_1-P_1$}\label{table 44}
		\resizebox{\textwidth}{12mm}{
			\begin{tabular}{ccccccccccccc}
				\hline
				$h$  & $\frac{\|e_u\|_{L^2(\Omega)}}{\|u\|_{L^2(\Omega)}}$  &  CR  &  $\frac{\|e_u\|_{H^1(\Omega)}}{\|u\|_{H^1(\Omega)}}$  &  CR & $\frac{\|e_p\|_{L^2(\Omega)}}{\|p\|_{L^2(\Omega)}}$ & CR  &  $\frac{\|e_p\|_{H^1(\Omega)}}{\|p\|_{H^1(\Omega)}}$  &  CR  &  $\frac{\|e_T\|_{L^2(\Omega)}}{\|T\|_{L^2(\Omega)}}$  &  CR  & $\frac{\|e_T\|_{H^1(\Omega)}}{\|T\|_{H^1(\Omega)}}$ & CR \\ 
				\hline
				$1/4$   &4.3904e-01& &2.8655e-01& &2.9446e-01& &3.9292e-01& &1.1156e-01& &2.8577e-01&  \\
				$1/8$  &1.2237e-01&1.8431 &1.4521e-01&0.9806 &7.9041e-02&1.8974 &1.7515e-01&1.1657 &2.8937e-02&1.9468 &1.4492e-01&0.9796  \\
				$1/16$  &3.1686e-02&1.9493 &7.2770e-02&0.9967 &2.0516e-02&1.9458 &8.1700e-02&1.1002 &7.3050e-03&1.9859 &7.2720e-02&0.9948  \\
				$1/32$  &8.0000e-03&1.9858 &3.6400e-02&0.9994 &5.2262e-03&1.9729 &3.8933e-02&1.0694 &1.8308e-03&1.9964 &3.6393e-02&0.9987  \\
				\hline
		\end{tabular}}
	\end{center}
\end{table}

\begin{table}[htbp]
	\vspace{-2.0em}
	\begin{center}
		\caption{Error and convergence rates of $u_h^n$, $p_h^n$, $T_h^n$ with $P_2-P_2-P_2-P_2$}\label{table 45}
		\resizebox{\textwidth}{12mm}{
			\begin{tabular}{ccccccccccccc}
				\hline
				$h$  & $\frac{\|e_u\|_{L^2(\Omega)}}{\|u\|_{L^2(\Omega)}}$  &  CR  &  $\frac{\|e_u\|_{H^1(\Omega)}}{\|u\|_{H^1(\Omega)}}$  &  CR & $\frac{\|e_p\|_{L^2(\Omega)}}{\|p\|_{L^2(\Omega)}}$ & CR  &  $\frac{\|e_p\|_{H^1(\Omega)}}{\|p\|_{H^1(\Omega)}}$  &  CR  &  $\frac{\|e_T\|_{L^2(\Omega)}}{\|T\|_{L^2(\Omega)}}$  &  CR  & $\frac{\|e_T\|_{H^1(\Omega)}}{\|T\|_{H^1(\Omega)}}$ & CR \\ 
				\hline
				$1/4$   &2.2304e-02& &3.4153e-02& &1.4276e-01& &6.8548e-01& &5.1892e-03& &3.4220e-02&  \\
				$1/8$  &2.7857e-03&3.0012 &8.7627e-03&1.9626 &2.2594e-02&2.6596 &2.5301e-01&1.4379 &6.4878e-04&2.9997 &8.6909e-03&1.9773  \\
				$1/16$  &3.4916e-04&2.9961 &2.1944e-03&1.9976 &3.2316e-03&2.8056 &7.6273e-02&1.7299 &8.1669e-05&2.9899 &2.1819e-03&1.9939 \\
				$1/32$  &4.3841e-05&2.9935 &5.4745e-04&2.0030 &4.2936e-04&2.9120 &2.0638e-02&1.8858 &1.0490e-05&2.9608 &5.4607e-04&1.9985  \\
				\hline
		\end{tabular}}
	\end{center}
\end{table}
We found that, whether using $P_1-P_1-P_1-P_1$ or $P_2-P_2-P_2-P_2$ elements, the calculation results were consistent with the theoretical analysis. These show that the reformulation of the thermo-poroelasticity model is of great significance, as it makes the first two equations can freely choose any finite element format to solve no longer subjecting to the constraints of the saddle point problem. In this way, we can choose the most suitable finite element format according to different problems and needs, thereby improving the efficiency and accuracy of the numerical solution of the equations.

\textbf{Test 2.} \label{test3}This test problem is same as one of  \cite{Brun2020}, we take $\Omega=[0,1]\times[0,1]\subset R^2$ and prescribe the following smooth solutions for the temperature,  pressure and displacement:
\begin{equation}
	\begin{array}{l}
		T(x, t)=t x_{1}\left(1-x_{1}\right) x_{2}\left(1-x_{2}\right),\\
		p(x, t)=t x_{1}\left(1-x_{1}\right) x_{2}\left(1-x_{2}\right),\\
		\bm{{\rm u}}(x, t)=t x_{1}\left(1-x_{1}\right) x_{2}\left(1-x_{2}\right)(1,1)^{'},
	\end{array}
\end{equation}
and $\textbf{f},\phi$ and $g$ can be calculated explicitly by using (\ref{1-1})-(\ref{1-3}). The boundary and initial condition are given by
\begin{eqnarray}
	p=tx(1-x)y(1-y)  \quad &\mathrm{on} & \quad\partial\Omega_{t},\nonumber\\
	T=tx(1-x)y(1-y)  \quad&\mathrm{on} & \quad\partial\Omega_{t},\nonumber\\
	u_1=tx(1-x)y(1-y) \quad&\mathrm{on}& \quad\Gamma_j\times (0,t_f),~j=2,~4,\nonumber\\
	u_2=tx(1-x)y(1-y)  \quad&\mathrm{on} & \quad\Gamma_j\times(0,t_f),~j=1,~3,\nonumber\\
	\sigma(\pmb{t_f})\bm{n}-\alpha pI\bm{n}-\beta TI\bm{n}=\textbf{f}_1 \quad &\mathrm{on}& \quad\partial\Omega_{t},\nonumber\\
	\textbf{u}(x,0)=\bm{0},~ p(x,0)=0,~T(x,0)=0 \quad&\mathrm{in}& \quad\Omega.\nonumber
\end{eqnarray}
\begin{table}[H]
	\centering
	\caption{ Physical parameters}\label{table1}
	\begin{tabular}{c l c }
		\hline
		Parameter   &  \quad Description      &   \quad   Value    \\
		\hline
		$a_0$       &\quad Effective thermal capacity               &  \quad 2e5\\
		$b_0$       &\quad Thermal dilation coefficient              &  \quad 1e5 \\
		$c_0$       &\quad Constrained specific storage coefficient &  \quad 2e5 \\
		$\alpha$    &\quad Biot-Willis constant                     &  \quad  0.01 \\
		$\beta$     &\quad Thermal stress coefficient.               &  \quad  0.01\\
		$a$         &\quad Mutation coefficient                      & \quad    1\\
		$b$         & \quad Stress coupling coefficient              & \quad    1\\
		$k_0$         &\quad Initial permeability tensor                      &  \quad $0.1I$\\
		$\bm{\Theta}$     &\quad Effective thermal conductivity           &  \quad $0.1I$\\
		$E$         &\quad Young's modulus                          &  \quad  2e7\\
		$\nu$       &\quad Poisson ratio                            &  \quad  0.4\\
		\hline
	\end{tabular}
\end{table}

\begin{table}[htbp]
	\vspace{-2.0em}
	\begin{center}
		\caption{Error and convergence rates of $u_h^n$, $p_h^n$, $T_h^n$}\label{table2}
		\resizebox{\textwidth}{12mm}{
			\begin{tabular}{ccccccccccccc}
				\hline
				$h$  & $\frac{\|e_u\|_{L^2(\Omega)}}{\|u\|_{L^2(\Omega)}}$  &  CR  &  $\frac{\|e_u\|_{H^1(\Omega)}}{\|u\|_{H^1(\Omega)}}$  &  CR & $\frac{\|e_p\|_{L^2(\Omega)}}{\|p\|_{L^2(\Omega)}}$ & CR  &  $\frac{\|e_p\|_{H^1(\Omega)}}{\|p\|_{H^1(\Omega)}}$  &  CR  &  $\frac{\|e_T\|_{L^2(\Omega)}}{\|T\|_{L^2(\Omega)}}$  &  CR  & $\frac{\|e_T\|_{H^1(\Omega)}}{\|T\|_{H^1(\Omega)}}$ & CR \\ 
				\hline
				$1/4$   &7.9244e-03&&5.4080e-02&&7.7938e-02&&4.3241e-01&&7.7926e-02&&4.3239e-01& \\
				$1/8$  &9.3526e-04&3.0829&1.3910e-02&1.9590&1.6853e-02&2.2093&2.1002e-01&1.0418&1.6824e-02&2.2116&2.0995e-01&1.0423  \\
				$1/16$  &1.1146e-04&3.0689&3.5210e-03&1.9821&3.8903e-03&2.1151&1.0373e-01&1.0178&3.8382e-03&2.1320&1.0349e-01&1.0206 \\
				$1/32$  &1.3597e-05&3.0352&8.8573e-04&1.9910&9.3103e-04&2.0630&5.1573e-02&1.0081&9.1006e-04&2.0764&5.1408e-02&1.0094 \\
				\hline
		\end{tabular}}
	\end{center}
\end{table}
\begin{table}[htbp]
	\vspace{-1.0em}
	\begin{center}
		\caption{Order of convergence of time discretization of Test 1 }\label{table211}
		\resizebox{\textwidth}{12mm}{
			\begin{tabular}{ccccccccccc}
				\hline
				$\Delta t$ & $\left\|\mathbf{u}_{h}^{\Delta t}-\mathbf{u}_{h}^{\frac{1}{2} \Delta t}\right\|_{L^2(\Omega)}$ & $\rho_{\Delta t,\mathbf{u}_{h}}$ & $\left\|p_{h}^{\Delta t}-p_{h}^{ \frac{1}{2} \Delta t}\right\|_{L^2(\Omega)}$ & $\rho_{\Delta t, p_h}$ & $\left\|T_{h}^{\Delta t}-T_{h}^{\frac{1}{2} \Delta t}\right\|_{L^2(\Omega)}$ & $\rho_{\Delta t, T_h}$ \\
				\hline
				$\frac{1}{10}$ &3.5779e-14&2.0000 & 2.5178e-10 &2.3567  &1.2589e-10&2.3567  \\
				$\frac{1}{20}$ &1.7889e-14&2.0000 & 1.0684e-10&2.2207  &5.3418e-11&2.2207  \\
				$\frac{1}{40}$ &8.9447e-15&2.0000 & 4.8108e-11&2.1203  &2.4054e-11&2.1203  \\
				$\frac{1}{80}$ &4.4724e-15&  & 2.2689e-11& &1.1345e-11&  & \\
				\hline
		\end{tabular}}
	\end{center}
\end{table}
\begin{figure}[htbp]
	\centering
	\subfigure[]{
		\centering
		\includegraphics[width=2.5in]{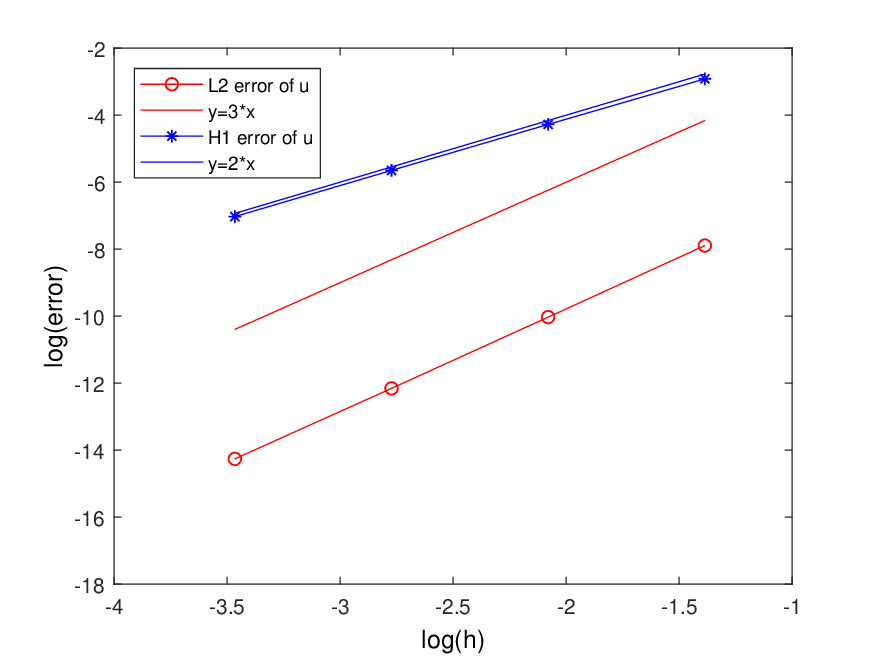}
		\label{fig111}
	}
	\subfigure[]{
		\centering
		\includegraphics[width=2.5in]{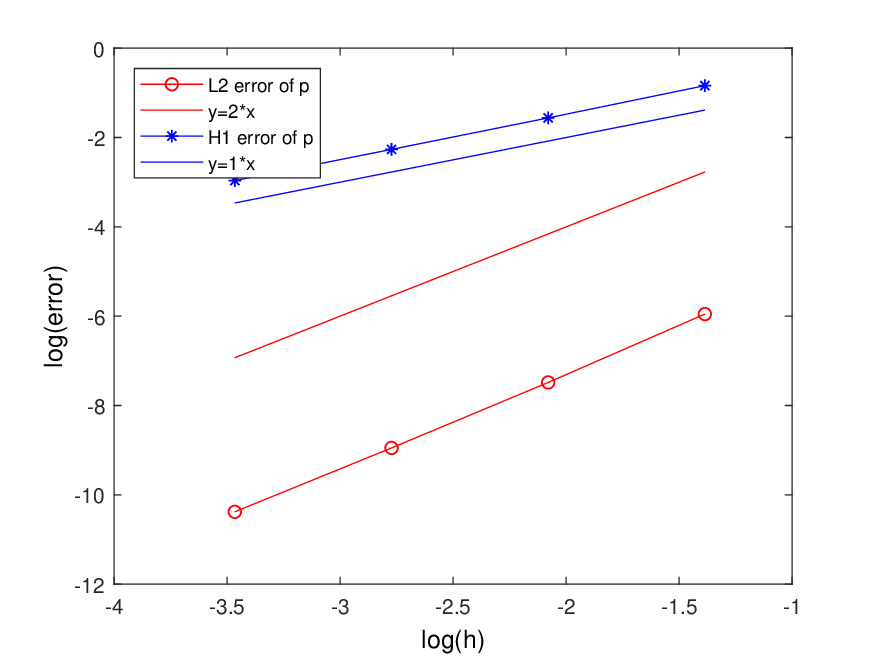}
		\label{fig211}
	}
	\subfigure[]{
		\centering
		\includegraphics[width=2.5in]{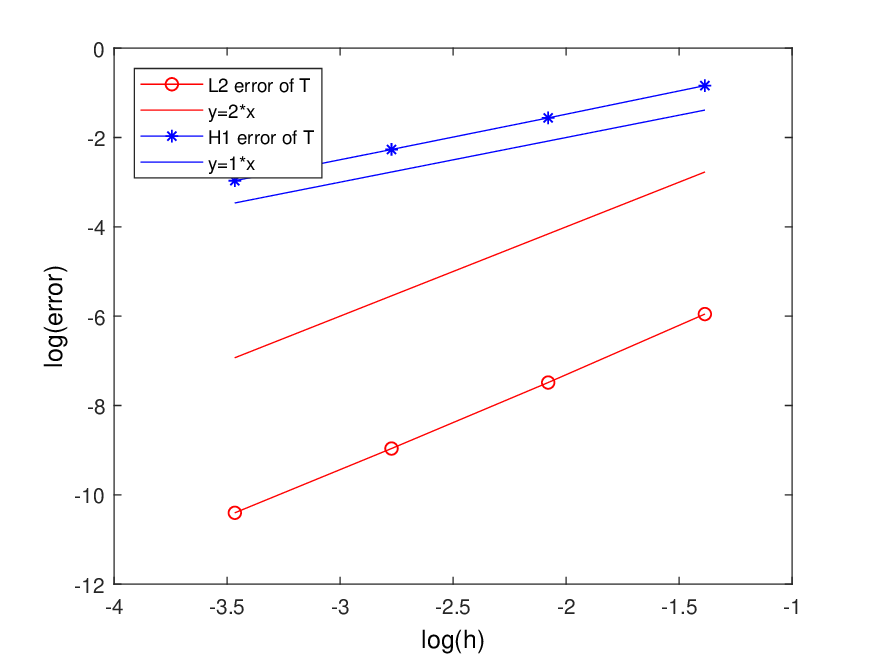}
		\label{fig2221}
	}
	\subfigure[]{
		\centering
		\includegraphics[width=2.5in]{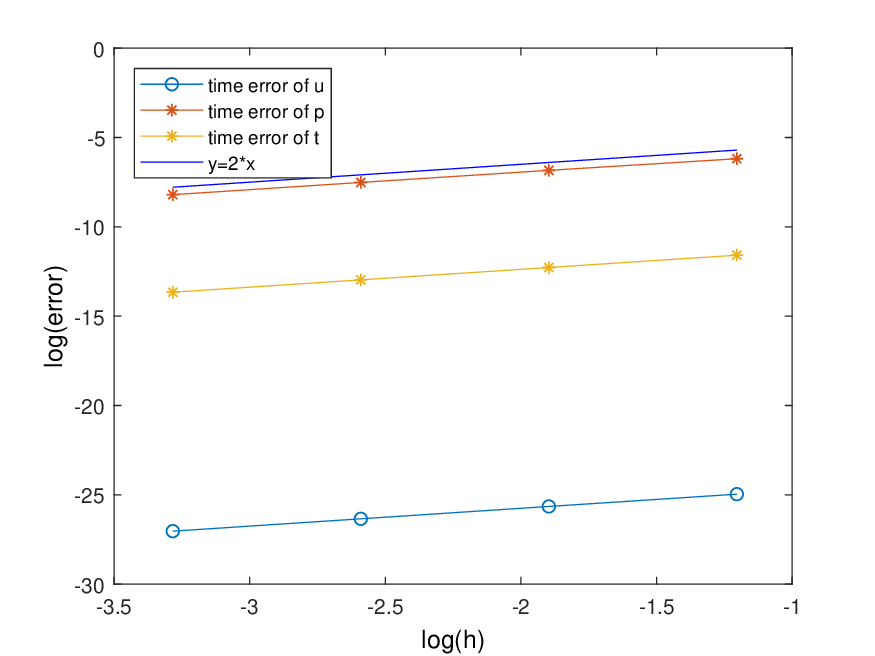}
		\label{fig2232}
	}
	\caption{(a) spatial convergence order for $u^{n}_{h}$, (b) spatial convergence rate for $p^{n}_{h}$,
		\\~~~~~~~~\quad \qquad(c) spatial convergence rate for $ T^{n}_{h}$,(d) time convergence order for $u^{n}_{h},p^{n}_{h},T^{n}_{h}$ }
\end{figure}
\begin{figure}[H]
	\centering
	\subfigure[]{
		\centering
		\includegraphics[width=2.5in]{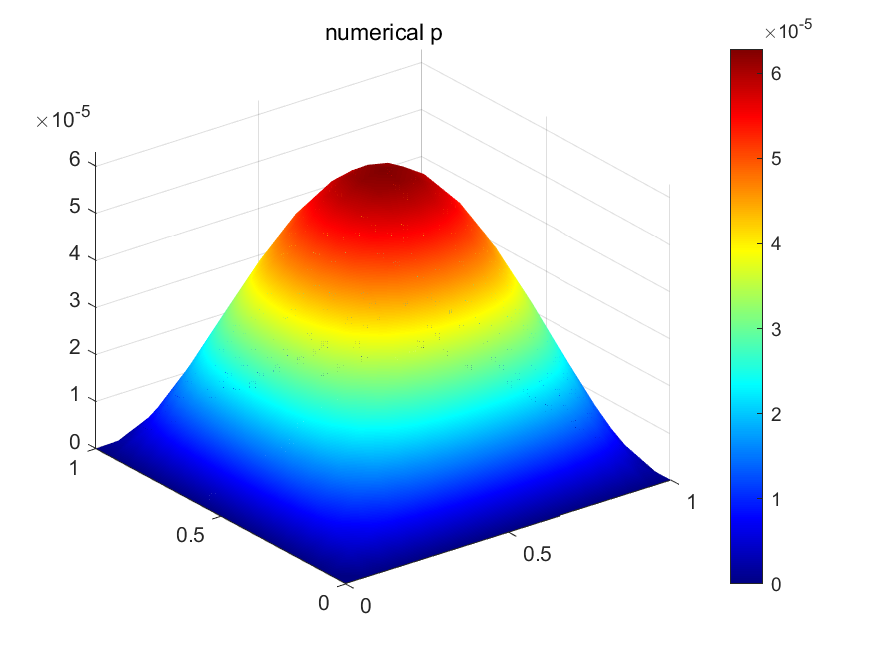}
		\label{fig6}
	}
	\subfigure[]{
		\centering
		\includegraphics[width=2.5in]{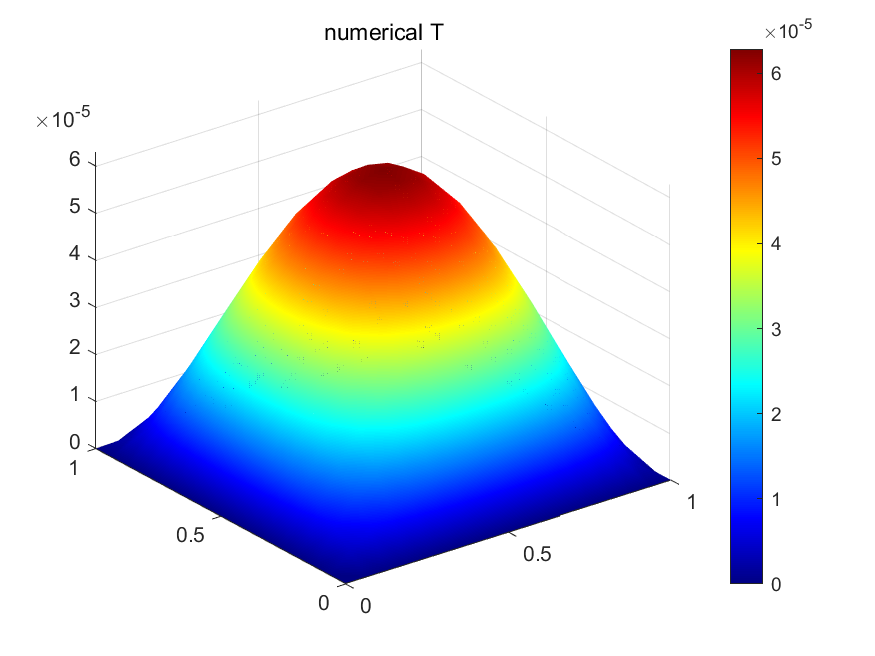}
		\label{fig7}
	}
	\caption{(a) and (b) are surface plot of the pressure $p_h^n$ and temperature $T_h^n$ at the terminal time $t_f$ respectively.}
\end{figure}
\begin{figure}[H]
	\centering
	\includegraphics[height=5cm,width=7cm]{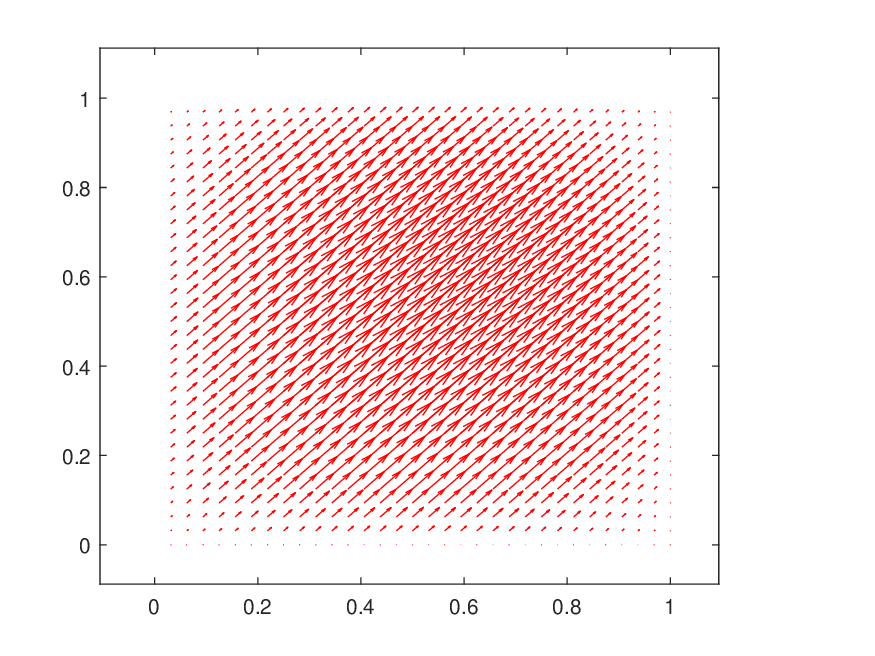}
	\caption{Arrow plot of the computed displacement $\textbf{u}_h^n$.}
	\label{fig8}
\end{figure}
\vspace{-0.8cm}
\begin{figure}[H]
	\centering
	\subfigure[]{
		\centering
		\includegraphics[width=2.5in]{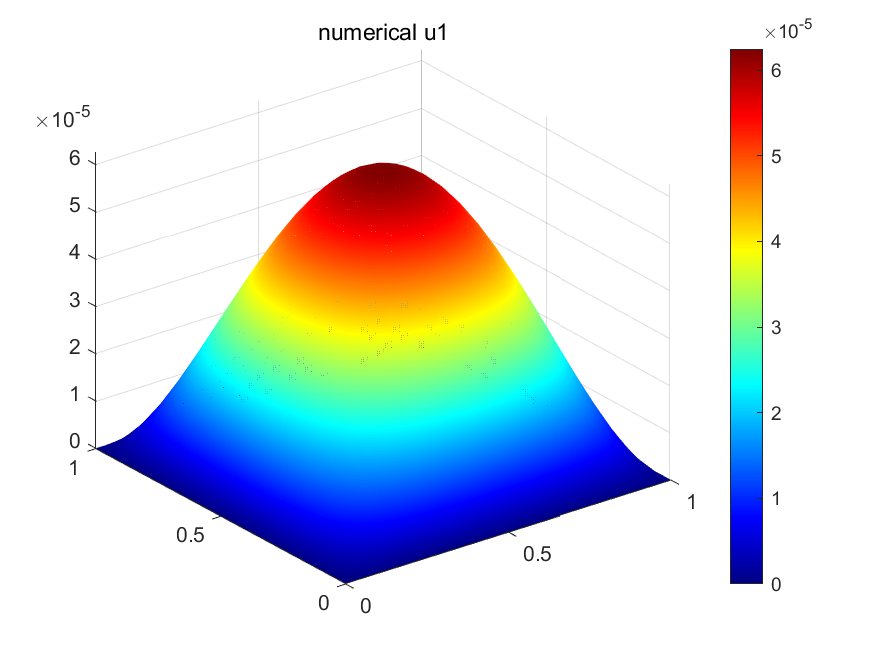}
		\label{fig9}}%
	\subfigure[]{
		\centering
		\includegraphics[width=2.5in]{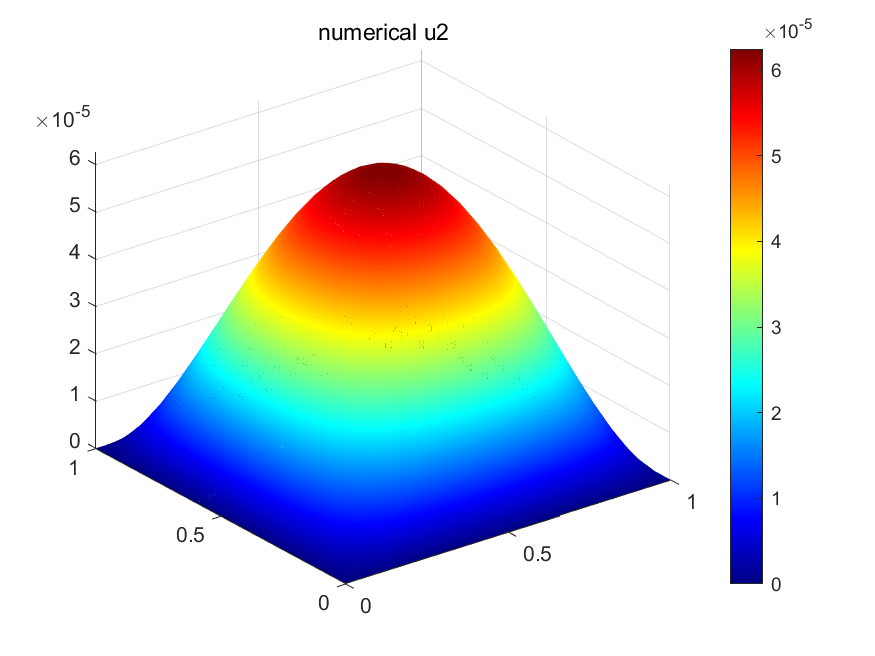}
		\label{fig10}}%
	\centering
	\caption{ (a) and (b) are Surface plot of $u_{1h}^{n}$ and $u_{2h}^{n}$ at the terminal time $t_f$ respectively.}
	
\end{figure}
Table \ref{table2} displays the $L^{\infty}(0,t_f;L^2(\Omega))$ and $L^{\infty}(0,t_f;H^1(\Omega))$-norm errors of $\textbf{u}$, $p$ and $T$ and shows that the convergence order with respect to $h$ is optimal, which verify the Theorem \ref{th-3-6} and Table \ref{table211} give the convergence order with respect to $\Delta t$ is optimal when $h=\frac{1}{8}$ and $t_f = 1e-3$.

Figure \ref{fig111} and Figure \ref{fig211} describe the spatial convergence order of $u^{n}_{h}, p^{n}_{h}$, $T^{n}_{h}$.
Figure \ref{fig6}, Figure \ref{fig7} and Figure \ref{fig9} and Figure \ref{fig10} show, respectively, the surface plot of the computed $p^{n}_h$, $T^{n}_h$, $u^{n}_{1h}$ and $u^{n}_{2h}$ at the terminal time $t_f$, Figure \ref{fig8} show the arrow plot of the computed displacement $\textbf{u}$.
 
\textbf{Test 3.} 
In this test, we consider so-call Barry-Mercer's problem (cf.  \cite{Wheeler2007}). The computational domain is $\Omega=[0,1]\times[0,1]$, and $t_f=1$. Barry-Mercer's problem has no source, that is, $\textbf{f}\equiv0$ and $g\equiv0$, we prescribe homogeneous boundary conditions and zero source term and initial condition for the heat problem and takes the following boundary and initial conditions
\begin{eqnarray*}
p=0  \quad &\mathrm{on}&\quad \Gamma_j\times (0,t_f),~j=1,~2,~3,\\
p=p_2 \quad&\mathrm{on}& \quad\Gamma_j\times (0,t_f),~j=4,\\
T=0 \quad &\mathrm{on}&\quad \Gamma_j\times (0,t_f),~j=1,~2,~3,\\
T=T_2 \quad&\mathrm{on}& \quad\Gamma_j\times (0,t_f),~j=4,\\
u_1=0\quad&\mathrm{on}& \quad\Gamma_j\times (0,t_f),~j=1,~3,\\
u_2=0 \quad &\mathrm{on} & \quad\Gamma_j\times(0,t_f),~j=2,~4,\\
\sigma(\pmb{t_f})\bm{n}-\alpha pI\bm{n}=\textbf{f}_1:=(0,\alpha p+\beta T)^{'} \quad &\mathrm{on}&\quad \partial\Omega_{t},\\
\textbf{u}(x,0)=\bm{0},~ p(x,0)=0,~T(x,0)=0 \quad&\mathrm{in}&\quad \Omega,\nonumber
\end{eqnarray*}
where
\begin{equation*}
p_{2}\left(x_{1}, t\right)=\left\{\begin{array}{ll}
\sin t &  { \rm if }\  y \in[0.2,0.8) \times(0, t_f),\\
0 & \text { others. }
\end{array}\right.
T_{2}\left(x_{1}, t\right)=\left\{\begin{array}{ll}

\sin t &  { \rm if }\  y \in[0.2,0.8) \times(0, t_f),\\
0 & \text { others. }
\end{array}\right.
\end{equation*}
\begin{table}[H]
\centering
\caption{ Physical parameters}\label{table9}
\begin{tabular}{c l c }
\hline
Parameter    &\quad Description                              &  \quad Value    \\
\hline
$a_0$        &\quad Effective thermal capacity               &  \quad 1e-1 or 1e-10\\
$b_0$        &\quad Thermal dilation coefficient             &  \quad 0 \\
$c_0$        &\quad Constrained specific storage coefficient &  \quad 1e-10 or 1e-1\\
$\alpha$     &\quad Biot-Willis constant                     &  \quad 1 \\
$\beta$      &\quad Thermal stress coefficient              &  \quad 1\\
$a$         &\quad Mutation coefficient                      & \quad    1\\
$b$         & \quad Stress coupling coefficient              & \quad    1\\
$k_0$         &\quad Initial permeability tensor                      &  \quad $0.1I$\\
$\bm{\Theta}$&\quad Effective thermal conductivity           &  \quad 1e-8$I$\\
$E$          &\quad Young's modulus                          &  \quad  2.8e5\\
$\nu$        &\quad Poisson ratio                            &  \quad  0.42\\
\hline
\end{tabular}
\end{table}

\begin{table}[htbp]
	\vspace{-2.0em}
	\begin{center}
		\caption{Error and convergence rates of $u_h^n$, $p_h^n$, $T_h^n$}\label{table150}
		\resizebox{\textwidth}{12mm}{
			\begin{tabular}{ccccccccccccc}
				\hline
				$h$  & $\|e_u\|_{L^2(\Omega)}$  &  CR  &  $\|e_u\|_{H^1(\Omega)}$  &  CR &
				 $\|e_p\|_{L^2(\Omega)}$ & CR  &  $\|e_p\|_{H^1(\Omega)}$  &  CR  &  $\|e_T\|_{L^2(\Omega)}$  &  CR  & $\|e_T\|_{H^1(\Omega)}$ & CR \\ 
				\hline
				$1/4$   &4.7823e-14& &8.8956e-13& &6.9291e-10& &1.1737e-08& &6.9291e-10& &1.1737e-08&  \\
				$1/8$  &5.0556e-14&0.9459 &4.5739e-13&1.9448 &3.7452e-10&1.8501 &1.4050e-08&0.8353 &3.7452e-10&1.8501 &1.4050e-08&0.8353   \\
				$1/16$  &2.2291e-14&2.2680 &2.3849e-13&1.9179 &1.7965e-10&2.0847 &1.6183e-08&0.8682 &1.7965e-10&2.0847 &1.6183e-08&0.8682  \\
				$1/32$  &1.1110e-14&2.0064 &1.2933e-13&1.8441 &8.3189e-11&2.1595 &1.1700e-08&1.3831 &8.3189e-11&2.1595 &1.1700e-08&1.3831  \\
				\hline
		\end{tabular}}
	\end{center}
\end{table}

 Since we lack a true analytical solution, we used a small spatial step size of 1/180 to generate a reference solution, and then solved the numerical solution with different spatial step sizes, and calculated the error between them and the reference solution. From table \ref{table150} by gradually reducing the spatial step size, we found that the error also decreased according to a certain rate, and the rate of decrease remained stable. This shows that our algorithm can effectively solve the Barry problem, and has good convergence and stability.
\begin{figure}[H]
	\centering
\subfigure[]{
\centering
\includegraphics[width=2.3in]{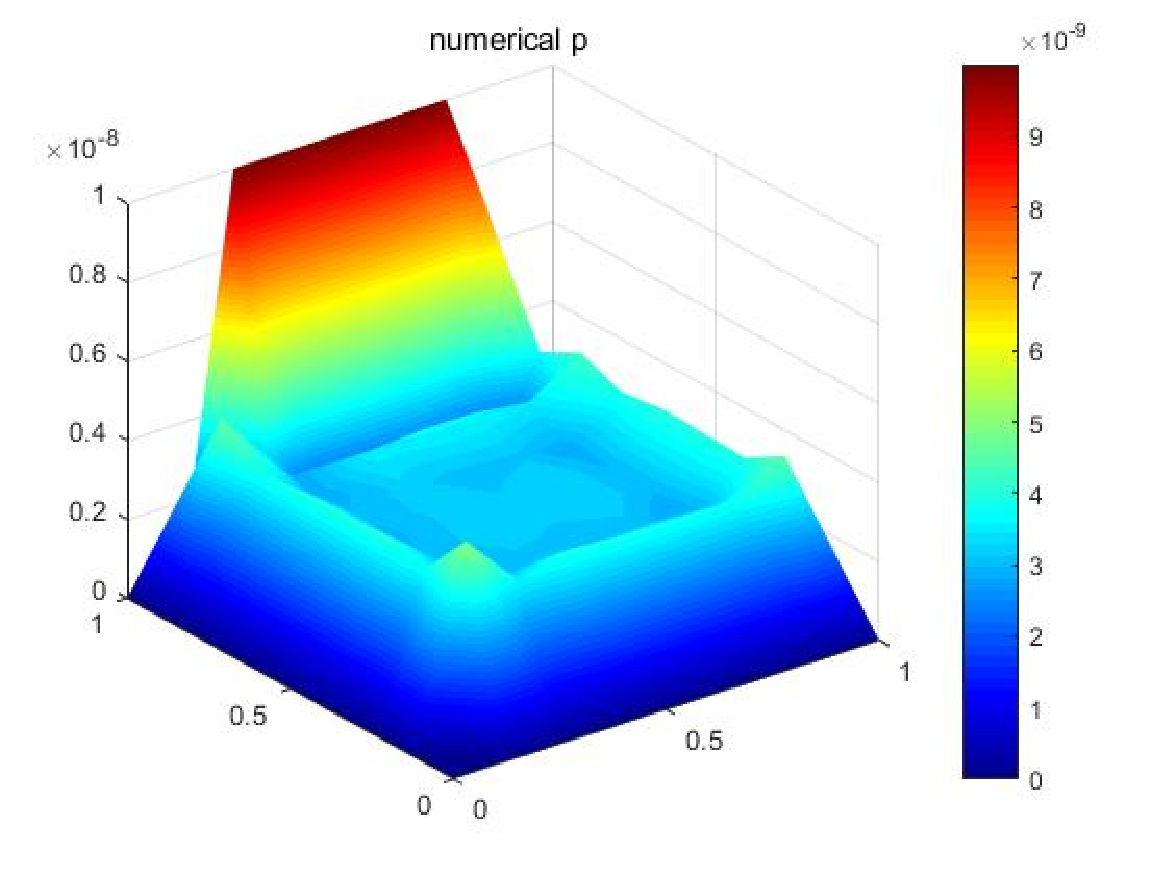}
\label{fig11}
}
\subfigure[]{
\centering
\includegraphics[width=2.3in]{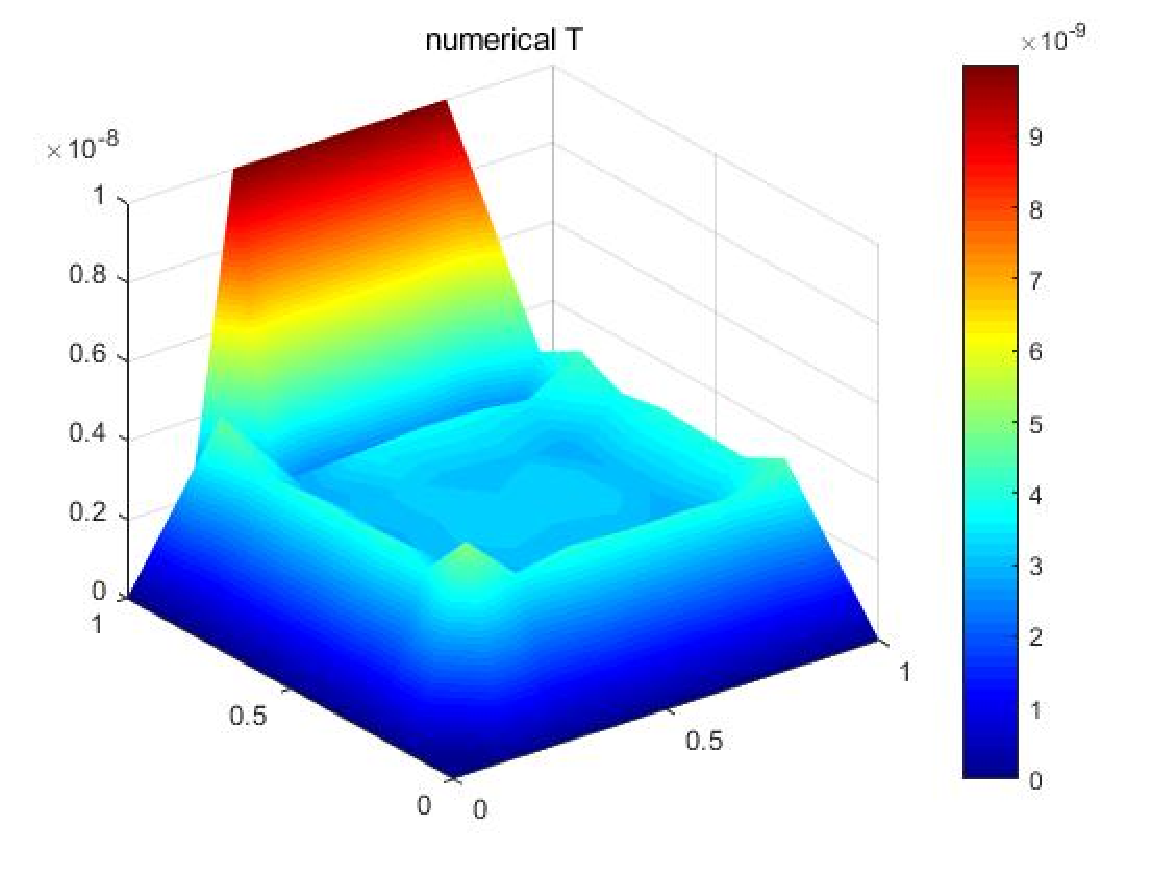}
\label{fig12}}\caption {The numerical results by using $P_2-P_1-P_1$ element pair for the variables of  $\textbf{u}, p$ and $T$ of the problem \reff{1-1}--\reff{1-3}: (a) locking in pressure field, (b) locking in temperature field.}
\end{figure}

\begin{figure}[htbp]
	\centering
\subfigure[]{
\centering
\includegraphics[width=2.3in]{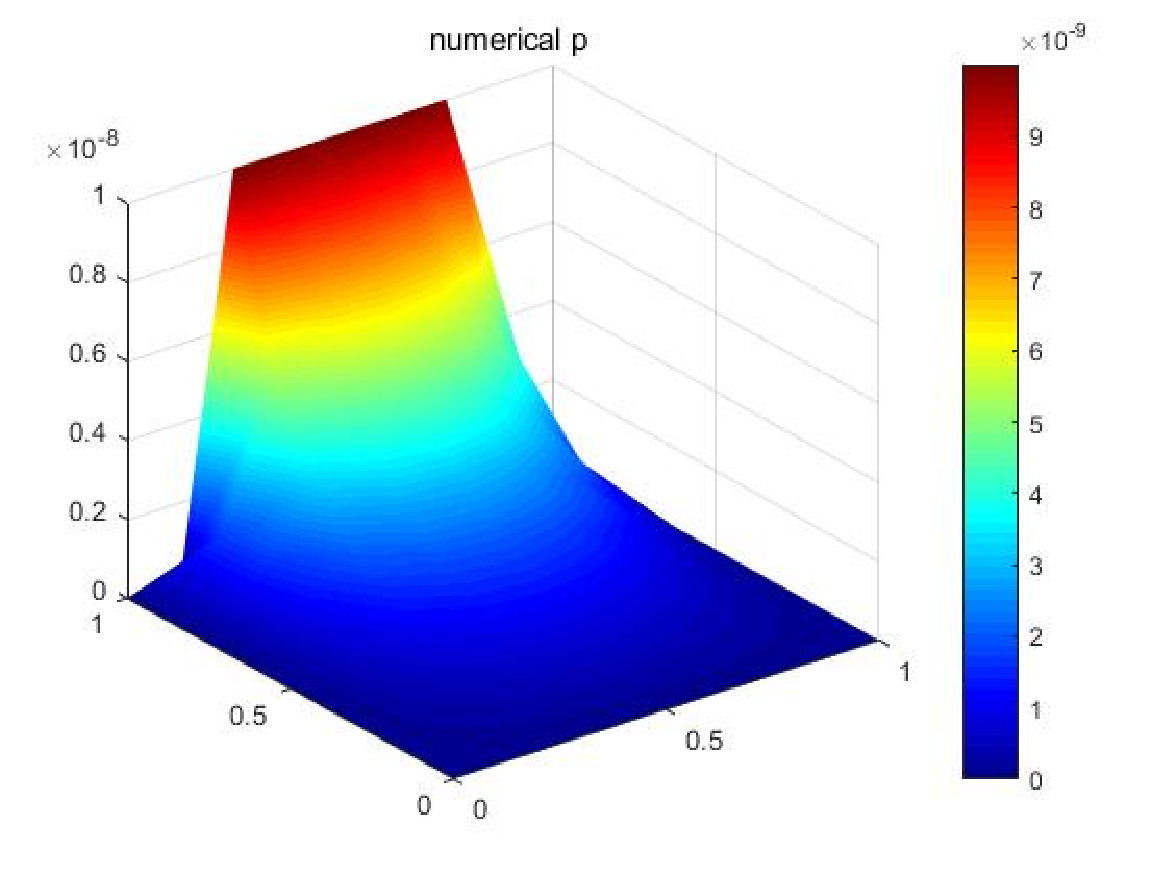}
\label{fig19}
}
\subfigure[]{
\centering
\includegraphics[width=2.3in]{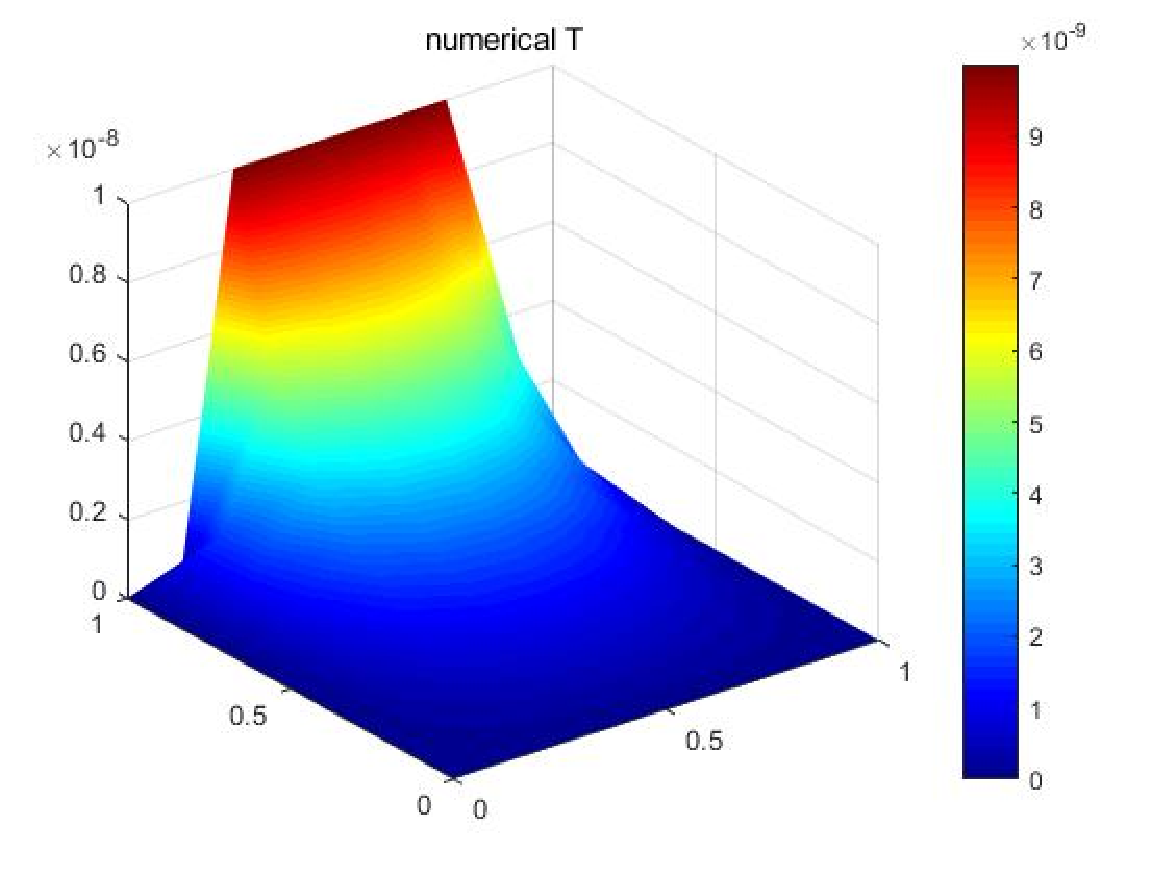}
\label{fig20}}
\caption{The numerical results by using $P_2-P_1-P_1-P_1$ element pair for the variables of  $\textbf{u},\tau,\varpi$ and $\varsigma$ of the proposed MFEA: (a) no locking in the pressure field, (b) no locking in the temperature field.}
\end{figure}
\begin{figure}[H]
\centering
\subfigure[]{
\centering
\includegraphics[width=2.5in]{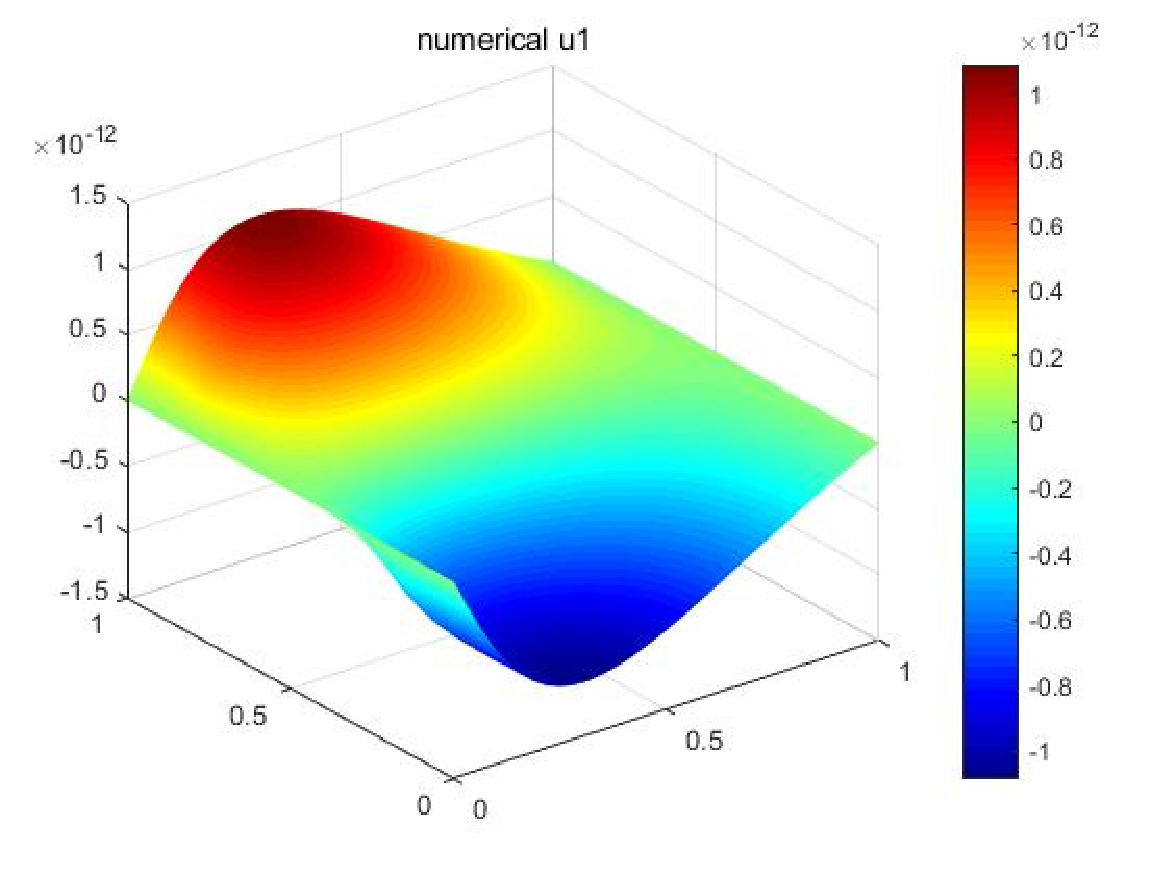}
\label{fig117}}%
\subfigure[]{
\centering
\includegraphics[width=2.5in]{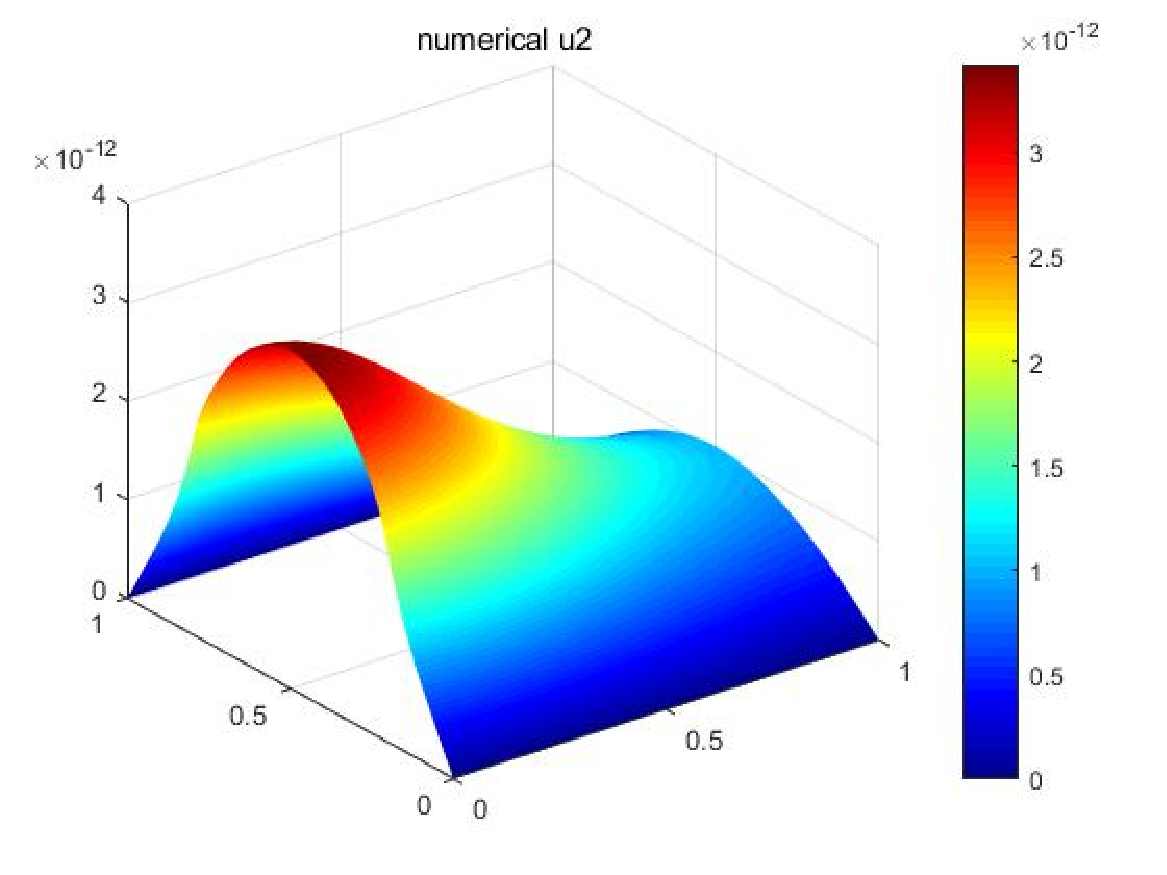}
\label{fig118}}%
\centering
\caption{ (a) and (b) are the surface plot of  $u_{1h}^{n}$ and $u_{2h}^{n}$  at time $t_f$, respectively.}
\end{figure}

\begin{figure}[H]
\centering
\includegraphics[height=5cm,width=7cm]{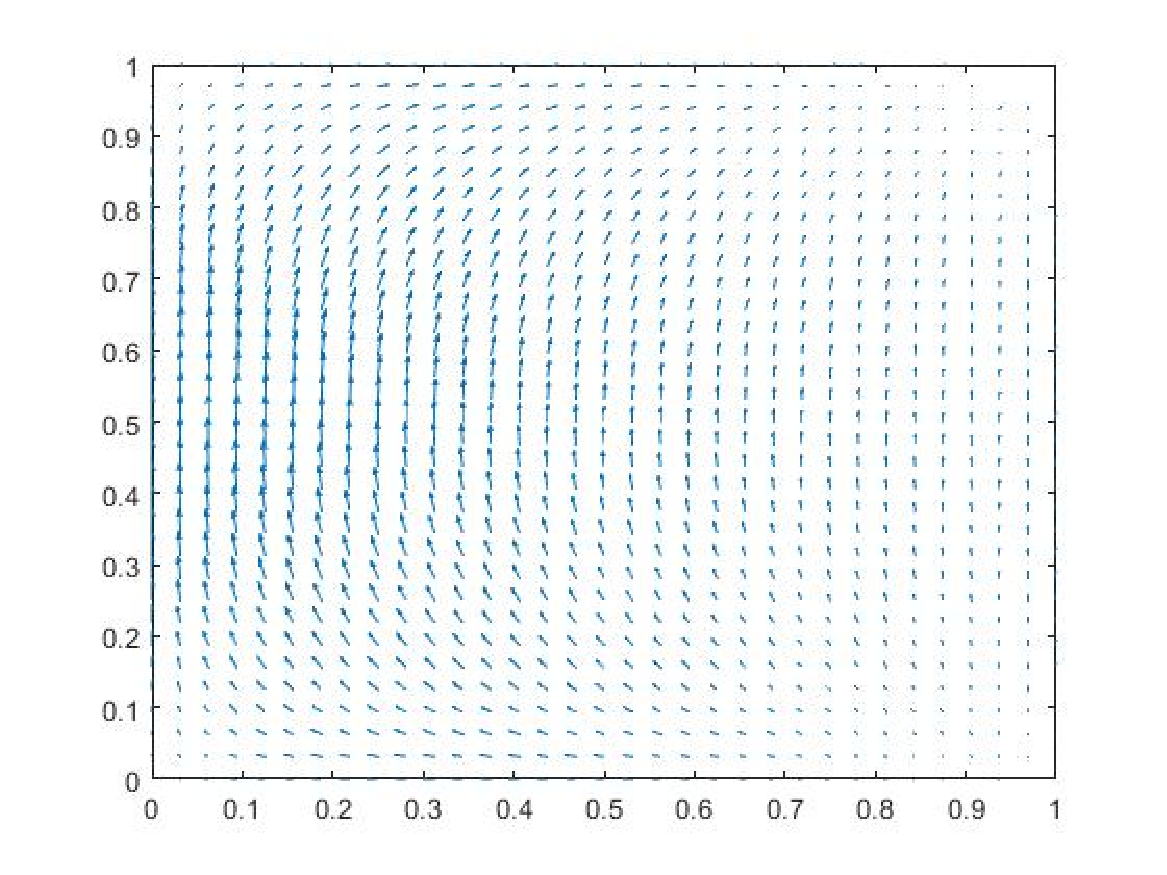}
\caption{Arrow plot of the displacement $\textbf{u}_h^n$.}
\label{fig119}
\end{figure}

Figure \ref{fig11} shows that pressure oscillations  occur by using $P_2-P_1-P_1$ element pair for the variables of  $\textbf{u}, p$ and $T$ of the problem \reff{1-1}--\reff{1-3} when $c_{0}=0, b_{0}=0$, and the permeability is very small for very short times, Figure \ref{fig12} shows that temperature oscillations occur by using $P_2-P_1-P_1$ element pair for the variables of  $\textbf{u}, p$ and $T$ of the problem \reff{1-1}--\reff{1-3} when $a_{0}=0, b_{0}=0$, and the thermal conductivity is very small for very short times. From Figure \ref{fig19} and Figure \ref{fig20}, we see that there is no locking phenomenon, which confirms that our approach and numerical methods have a built-in mechanism to prevent the "locking phenomenon". Figure \ref{fig117} and Figure \ref{fig118} display the surface plot of $u^{n}_{1h}$ and $u^{n}_{2h}$ at the terminal time $t_f$, Figure \ref{fig119} show the arrow plot of the computed displacement $\textbf{u}^{n}_{h}$.\\

\textbf{Test~4.}
To check the function $k$ influence the result of the Barry problem we set $k_0$ as a fixed number, let's  coupling coefficient $b $ in $ak_{0}e^{-b \left((\lambda+\mu)\mathrm{div} \mathbf{u}-\alpha p- \beta T\right )}$ be $ 1e-2,1,1e2$ in turn and for $p,T$, then we use the following boundary conditions
\begin{eqnarray*}
	p=0  \quad &\mathrm{on}&\quad \Gamma_j\times (0,t_f),~j=1,~2,~4,\\
	p=p_2 \quad&\mathrm{on}& \quad\Gamma_j\times (0,t_f),~j=3,\\
	T=0 \quad &\mathrm{on}&\quad \Gamma_j\times (0,t_f),~j=1,~2,~4,\\
	T=T_2 \quad&\mathrm{on}& \quad\Gamma_j\times (0,t_f),~j=3,
\end{eqnarray*}
and the other conditions are same as the forth test \ref{test3}.
\begin{figure}[H]
	\centering
	\subfigure[]{
		\centering
		\includegraphics[width=2.5in]{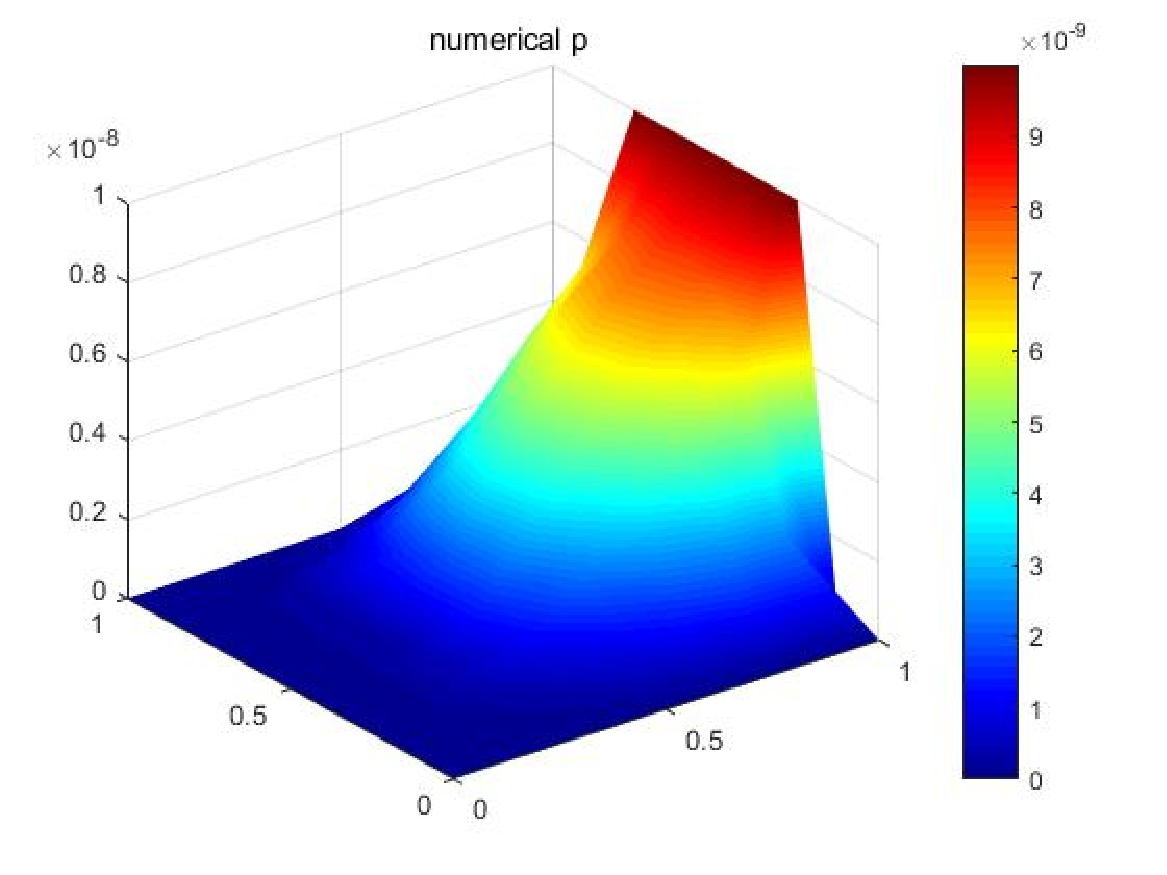}
		\label{fig120}}%
	\subfigure[]{
		\centering
		\includegraphics[width=2.5in]{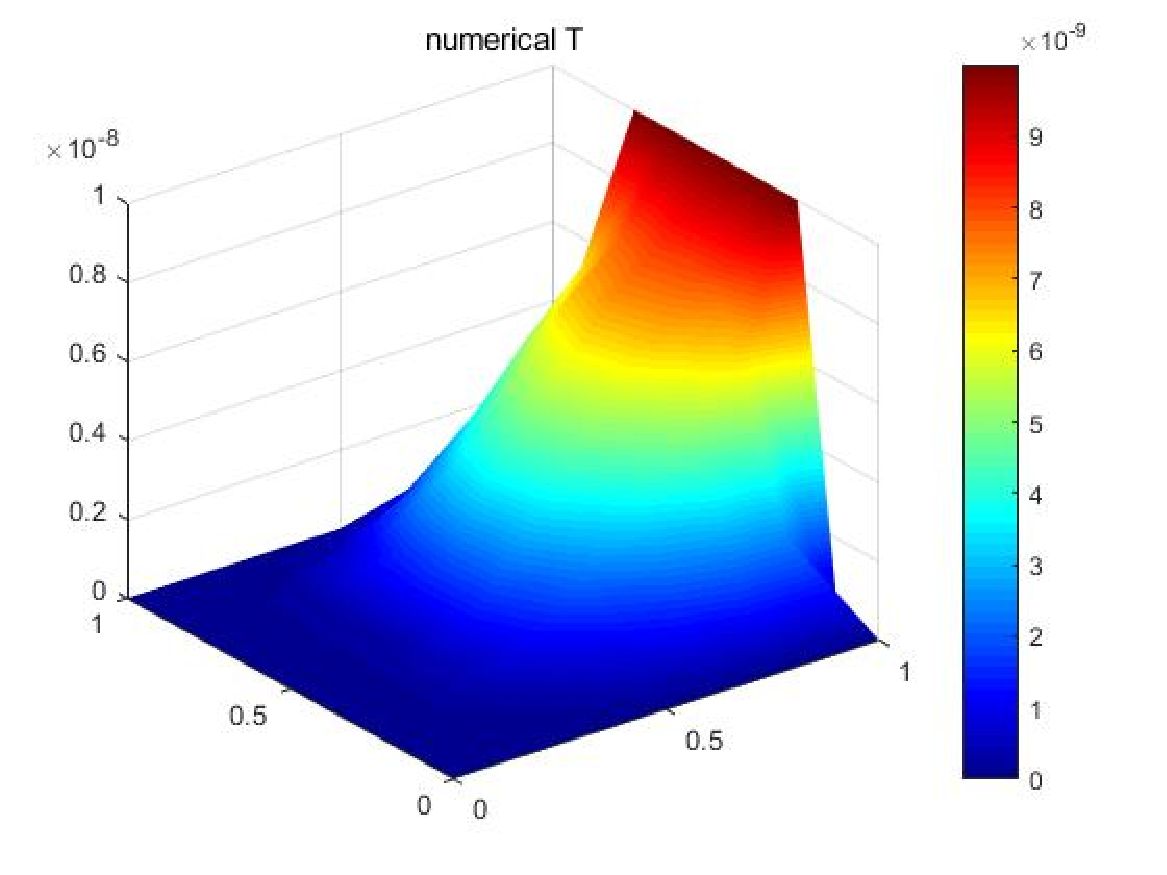}
		\label{fig121}}%
	\centering
	\caption{(a) and (b) are the surface plot of  $p_{h}^{n}$ and $T_{h}^{n}$  at time $t_f$ when $b$=1e-2.}
\end{figure}
\begin{figure}[H]
	\centering
	\subfigure[]{
		\centering
		\includegraphics[width=2.5in]{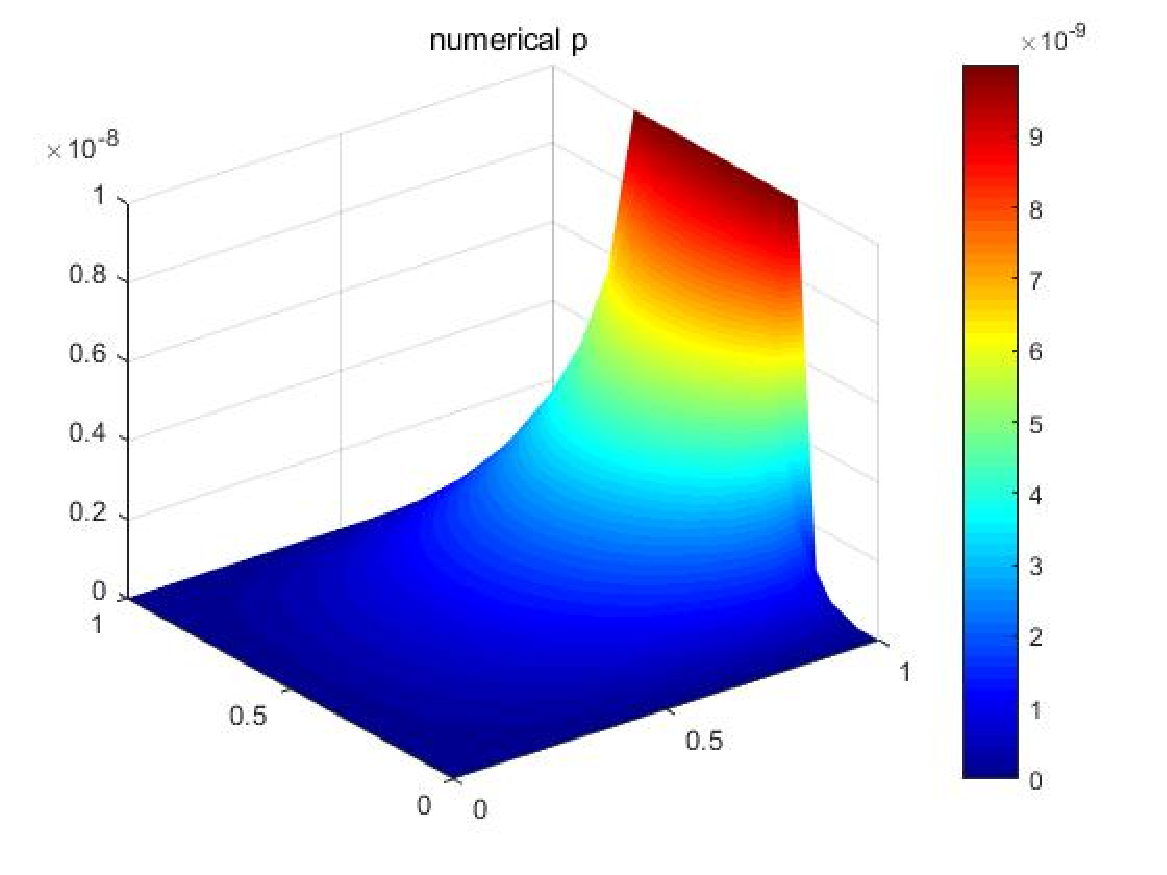}
		\label{fig122}}%
	\subfigure[]{
		\centering
		\includegraphics[width=2.5in]{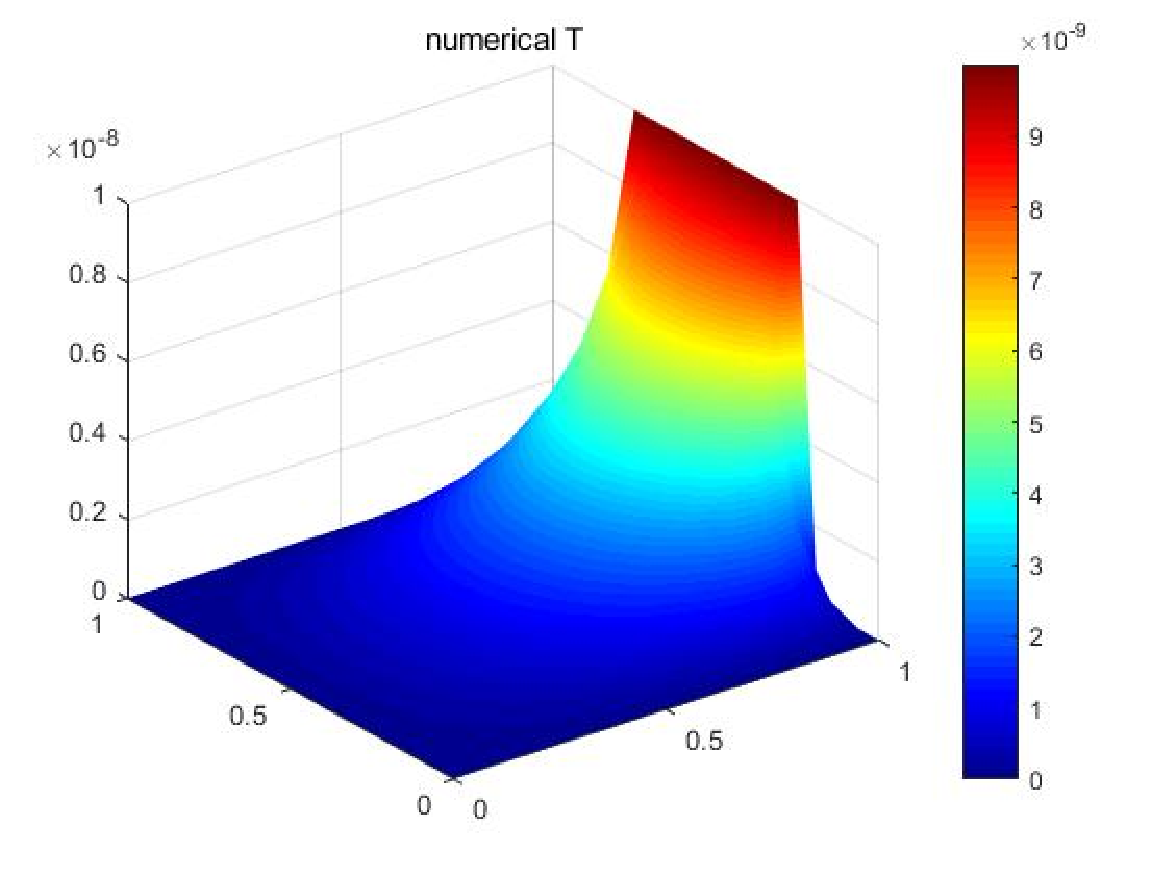}
		\label{fig123}}%
	\centering
	\caption{(a) and (b) are the surface plot of  $p_{h}^{n}$ and $T_{h}^{n}$  at time $t_f$ when $b$=1.}
\end{figure}
\begin{figure}[H]
	\centering
	\subfigure[]{
		\centering
		\includegraphics[width=2.5in]{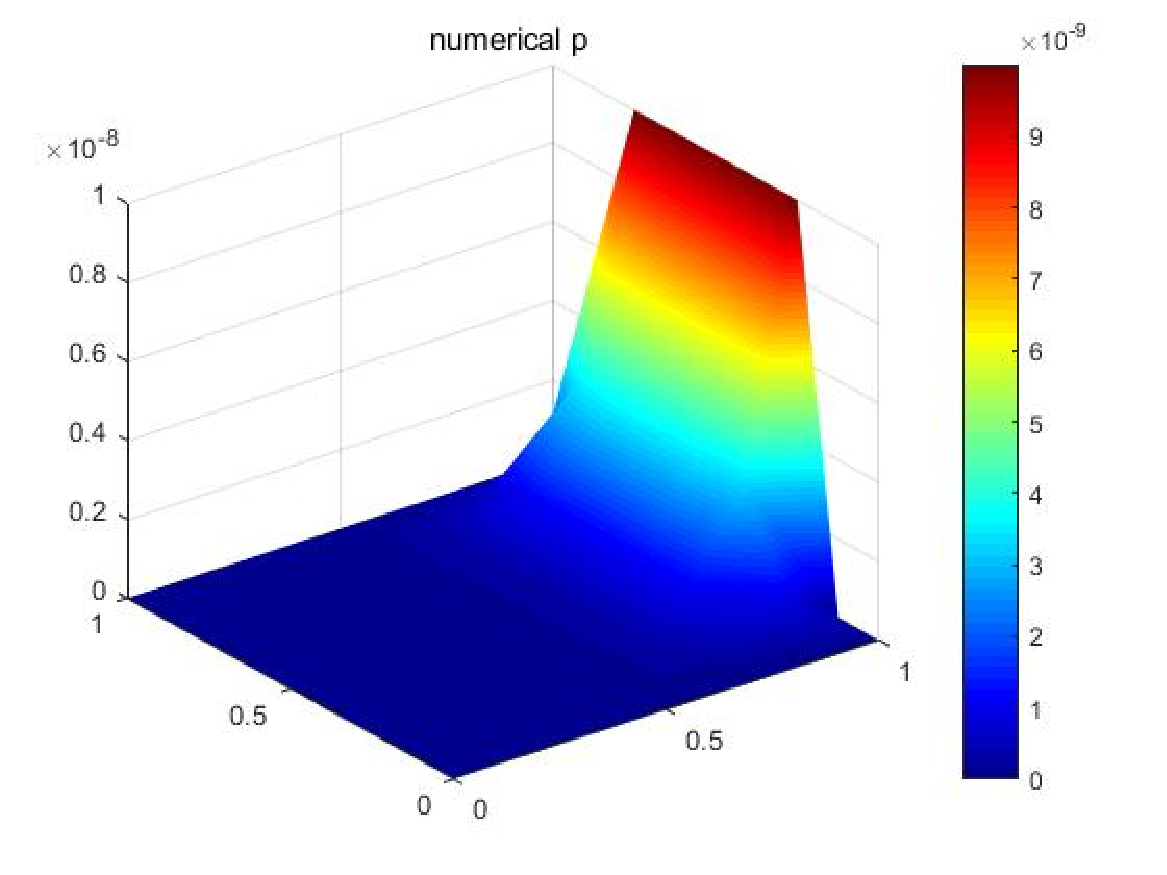}
		\label{fig124}}%
	\subfigure[]{
		\centering
		\includegraphics[width=2.5in]{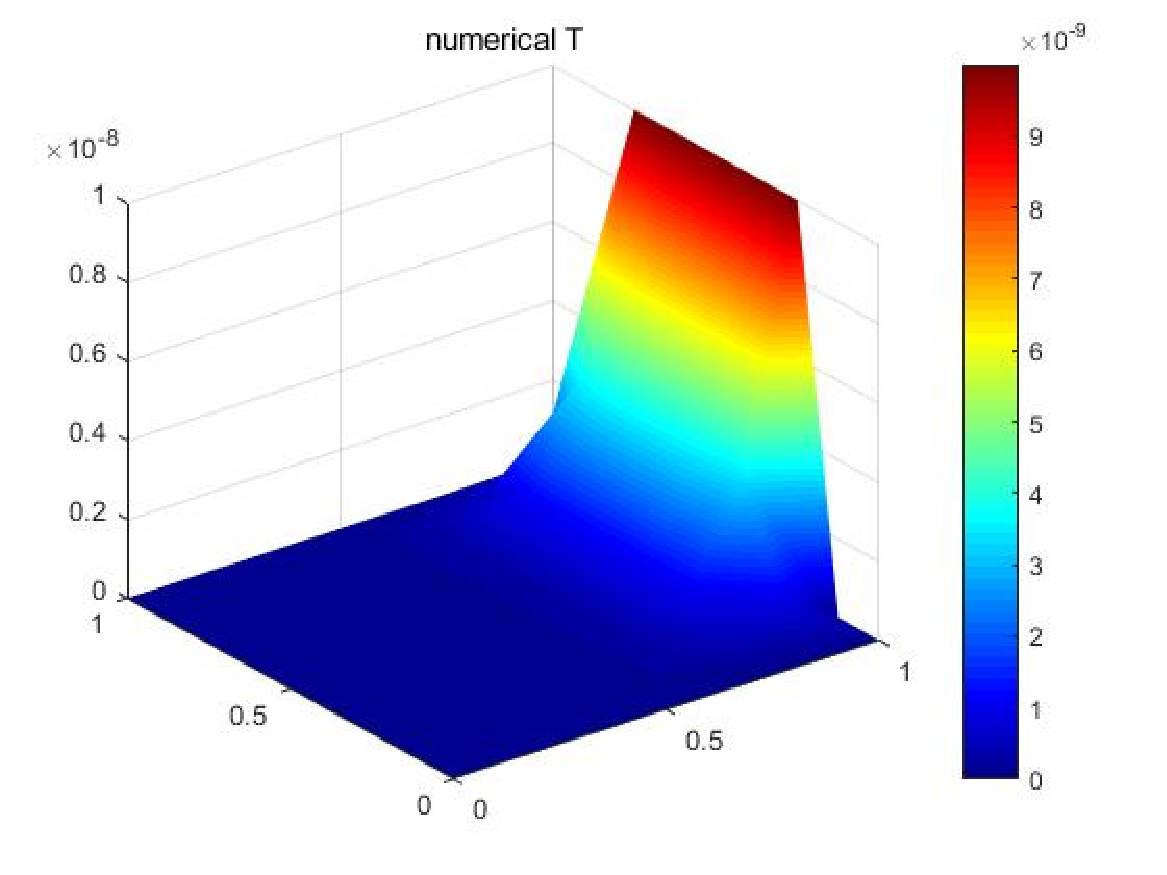}
		\label{fig125}}%
	\centering
	\caption{(a) and (b) are the surface plot of  $p_{h}^{n}$ and $T_{h}^{n}$  at time $t_f$ when $b$=1e2.}
\end{figure}

According to the observation of the six images above, we can find that the coupling coefficient b has a significant effect on the pressure and temperature, and it can change the variation speed and trend of the pressure and temperature with time. The coupling coefficient b is a physical quantity that describes the interaction strength between two systems, and it has applications in electromagnetism, mechanics, quantum mechanics and other fields. In this experiment, the coupling coefficient $b$ may reflect different physical processes, such as heat conduction, phase transition, etc. By analyzing the effect of the coupling coefficient b on the pressure and temperature, we can better understand the nature and laws of these physical processes.
\section{Conclusion}
In this paper, we studies the  thermo-poroelasticity model of considering the non-linearity of the permeability $k$ according to the pressure, temperature, and pressure; provides an effective and reliable numerical calculation method for the thermo-fluid-solid interaction problem in porous media, which has important significance for guiding practical engineering practice. 

The main contributions of this paper:
 By introducing 3 new variables to reformulate the original model, we reveal the action process of the medium in the thermo-poroelasticity model  more clearly, as well as the mutual influence of physical phenomena such as fluid flow, temperature distribution, deformation and stress. We introduce the
 $B$-operator to simplify the model, and transform it into an abstract nonlinear parabolic system, then use the
 improved Rother¡¯s method to prove the existence and uniqueness of the weak solution. We propose a new multiphysics finite element method based on the reformulated model, which can solve two subproblems  respectively,  reducing the difficulty of solving the nonlinear terms. Meanwhile, by reformulation,  the problem is no longer affected by the saddle point, thus we can choose suitable finite element pair to handle variables according to the actual situation. Then we give the optimal error convergence order of the variables under any finite element pair, and verify it in numerical experiments. Our method provides a new idea and tool for the thermo-poroelasticity problem, and also provides a reliable numerical simulation scheme for practical problems such as groundwater flow and land subsidence.

%
%
%
%

\end{document}